\documentclass[12pt]{article}
\usepackage{amsmath}
\usepackage{amsmath,amssymb, amsthm, latexsym, amsfonts, epsfig, color,graphicx,}
\usepackage{mathrsfs}
\usepackage{indentfirst}
\usepackage{chngcntr}
\usepackage{bbm}
\usepackage{newtxmath} 
\usepackage[colorlinks=True,linkcolor=blue,anchorcolor=blue,citecolor=blue,CJKbookmarks=true]{hyperref}
\usepackage{geometry}
\begin{document}
\renewcommand{\a}{\alpha}
\newcommand{\D}{\Delta}
\newcommand{\ddt}{\frac{d}{dt}}
\counterwithin{equation}{section}
\newcommand{\e}{\epsilon}
\newcommand{\eps}{\varepsilon}
\newtheorem{theorem}{Theorem}[section]
\newtheorem{proposition}{Proposition}[section]
\newtheorem{lemma}[theorem]{Lemma}
\newtheorem{remark}[theorem]{Remark}
\newtheorem{example}{Example}[section]
\newtheorem{definition}{Definition}[section]
\newtheorem{corollary}[theorem]{Corollary}
\makeatletter
\newcommand{\rmnum}[1]{\romannumeral #1}
\newcommand{\Rmnum}[1]{\expandafter\@slowromancap\romannumeral #1@}
\makeatother

\title{\bf Observability and unique continuation inequalities for the Schr\"{o}dinger equations with inverse-square potentials\footnote{This work was supported by  National Natural Science Foundation of China (61671009, 12171178).}}
\author{Hui Xu,\,\,  Longben Wei\footnote{Corresponding author.},~~Zhiwen Duan\\
{\small {\it  School of Mathematics and Statistics, Huazhong University of Science }}\\
{\small {\it and Technology, Wuhan, {\rm 430074,} P.R.China}}  \\
{\small {\it Email: d202180016@hust.edu.cn~(H.Xu)}}\\
{\small {\it Email: d202080006@hust.edu.cn~(L.Wei)}}\\
{\small {\it Email: duanzhw@hust.edu.cn~(Z.Duan)}}}
\date{}
\maketitle
\centerline{\large\bf Abstract }
This paper is inspired by Wang, Wang and Zhang's work [ Observability and unique continuation inequalities for the Schrödinger equation. J. Eur. Math. Soc. \textbf{21}, 3513–3572 (2019)], where they present several observability and unique continuation inequalities for the free Schr\"{o}dinger equation in $\mathbb{R}^{n}$.
We extend all such observability and unique continuation inequalities for the Schr\"{o}dinger equations on half-line with inverse-square potentials. Technically, the proofs essentially rely on the representation of the solution, a Nazarov type uncertainty principle for the Hankel transform and an interpolation inequality for functions whose Hankel transform have compact support.

\section{Introduction}
In this paper, we will present several observability and unique continuation inequalities (at either two points in time or one point in time) for the following Schr\"{o}dinger equations:
\begin{equation}\label{equation1.1}
\begin{cases}
i\partial_tu(t,x)=\left( -\partial_{x}^{2}+(\nu^{2}-\frac{1}{4})\frac{1}{x^2}\right) u(t,x),\,\,x\in \mathbb{R}^+,~t>0, \\
u(0,x)=u_0\in L^{2}(\mathbb{R}^{+}),
\end{cases}
\end{equation}
where the fixed constant $\nu\geq0$. (Here and in what follows, $L^{2}(\mathbb{R^{+}})=L^{2}(\mathbb{R}^{+};\mathbb{C})$. The same is said about $C_{0}^{\infty}(\mathbb{R}^{+})$.) The Schr\"{o}dinger equation \eqref{equation1.1} with inverse-square potential is of interest in quantum mechanics. The family of differential operators $-\partial_{x}^{2}+(\nu^{2}-\frac{1}{4})\frac{1}{x^2}$ is very special, they appear in numerous applications, e.g., as the radial part of the Laplacian in any dimension. Their eigenfunctions can be expressed in terms of Bessel-type functions, and they have a surprisingly long and intricate theory, see \cite{HUJR,CH,EG}. We also note that the heat flow associates to the inverse square potential has been studied in the theory of combustion (see \cite{JE} and references therein). The mathematical interest in these equations however comes mainly from the fact that the potential term is homogeneous of degree and therefore scales exactly the same as the Laplacian. This in particular implies that perturbation methods cannot
be used in studying the effect of this potential. Indeed, the decay is in some sense the borderline case for the existence of global-in-time estimates for Schr\"{o}dinger equation with a potential (see \cite{IW}). In particular, it is known that a negative potential $V$ decaying slower than inverse-square results in the spectrum of  being unbounded from below (\cite{RS2}, Section XIII, pp. 87–88).
\par Now to proceed our introduction, we need to recall some background about the observability inequality for the Schrödinger equation. The classical observability inequality for the Schrödinger equation:
\begin{equation}\label{equation1.2}
\begin{cases}
	i\partial_{t}u(t,x)=Hu(t,x),~(t,x)\in\mathbb{R}\times \mathbb{M},\\
	u(0,x)=u_0\in L^{2}(\mathbb{M}),
\end{cases}
\end{equation}
reads that when $u$ solves \eqref{equation1.2},
\begin{equation}\label{equation1.3}
\int_{\mathbb{M}}|u(0,x)|^{2}dx\leq C_{obs}\int_0^T\int_\Omega|u(t,x)|^2dxdt,
\end{equation}
where $T>0$, $\Omega$ is a subset of the manifold $\mathbb{M}$ and the constant $C_{obs}$ is called the observable constant or cost constant, $H$ denotes the self-adjoint extension of the Schr\"{o}dinger operator $-\Delta_g+V$ on $L^{2}(\mathbb{M})$ where $V$ is a potential. This is the observability inequality for the Schr\"{o}dinger equation to involve observation in a time interval. This type of the observability has been extensively studied in the literature  on compact Riemannian manifolds, we refer readers to \cite{A1, A2, BBZ, BZ, BZ1, BZ2, J} for results on some compact Riemannian manifolds. For the Schr\"{o}dinger equation on non-compact Riemannian manifolds, there are relatively few existing results, for which new difficulties arise due to the presence of infinity in space. But recently, there has been a growing interest in the question of observability for the Schr\"{o}dinger equation in the Euclidean space. For example, when $\mathbb{M}=\mathbb{R}^{n},\,V=0$, the inequality \eqref{equation1.3} holds if $\Omega=\{x\in \mathbb{R}^{n}:\,|x|\geq r\}$ in all dimensions $n\geq1$, and the article \cite{HWW} gives a sharp result that the inequality \eqref{equation1.3} holds if and only if $\Omega$ is thick (a more general set class) in dimension one. We refer readers to \cite{MT, KLJ, AP, JK} for more general such observability estimates in the Euclidean space.
\par Recently, Wang, Wang and Zhang \cite{WWZ} have proved the following new type of observability inequality: Given $x_1,\,x_2\in \mathbb{R}^{n}$, $r_1,\,r_2>0$ and $T>S\geq0$, there is a positive constant $C=C(n)$ such that for all $u(x,t)$ solving \eqref{equation1.2} with $V=0$, 
\begin{equation}\label{equation1.4}
\int_{\mathbb{R}^{n}}|u_0(x)|^{2}dx\leq Ce^{Cr_1r_2\frac{1}{T-S}}\left( \int_{B_{r_1}^{c}(x_1)}|u(x,S;u_0)|^{2}dx+\int_{B_{r_{2}}^{c}(x_2)}|u(x,T;u_0)|^{2}dx\right) .
\end{equation}
This improves the inequality \eqref{equation1.3} since only two time points appear on the right-hand side of \eqref{equation1.4}. For the same reason, \eqref{equation1.4} is called observability inequality at two time points. The proof of \eqref{equation1.4} in \cite{WWZ} is based on the following basic identity in the free case,
\begin{equation}\label{equation1.5}
(2it)^{\frac{n}{2}}e^{-i|x|^{2}/4t}u(t,x)=\widehat{e^{i|\cdot|^{2}/4t}u_0}(x/2t),\,\,\,for\,\,all\,\,t>0,~x\in \mathbb{R}^{n},
\end{equation}
where $\widehat{\cdotp}$ denotes the Fourier transform. With identity \eqref{equation1.5} in hand, they proved in \cite{WWZ} that the estimate \eqref{equation1.4} is equivalent to the following Nazarov's uncertainty principle built up in \cite{PJ} (see also \cite{VB, F}): If $A,B$ are subsets of $\mathbb{R}^{n}$ of finite measure, then
\begin{equation}\label{equation1.6}
\int_{\mathbb{R}^{n}}|f(x)|^{2}dx\leq C(n,A,B)\left( \int_{\mathbb{R}^{n}\setminus A}|f(x)|^{2}dx+\int_{\mathbb{R}^{n}\setminus B}|\hat{f}(\xi)|^{2}d\xi\right) ,\,\,f\in L^{2}(\mathbb{R}^{n}),
\end{equation}
with
\begin{equation*}
C(n,A,B)=Ce^{Cmin\{|A||B|,|B|^{1/n}\omega(A), |A|^{1/n}\omega(B)\}},
\end{equation*}
where $\omega(A)$ denotes the mean width of $A$, $C>0$ is an absolute constant. We need to point out that the identity \eqref{equation1.5} is a crucial tool in this article, based on this identity and other tools, they obtained more quantitative estimates besides \eqref{equation1.4}.
\par Three natural questions were raised in \cite{WWZ} that, can their results be extended to the following situations? (a) Schr\"{o}dinger equations with nonzero potentials. (b) Homogeneous Schr\"{o}dinger equations on a bounded domain $\Omega$. (c) Schr\"{o}dinger equations on half space $\mathbb{R}^{n}_{+}$ (where $\mathbb{R}^{n}_{+}:=\{(x_1,...,x_n)\in \mathbb{R}^{n}:x_n>0\}$). Recently, for question (a), Huang  and  Soffer \cite{HS} considered a class of decaying potentials $V$ (Don't include our situation) and established observability inequalities similar to \eqref{equation1.4} at two points in time for $H=-\Delta+V$ in $\mathbb{R}^{n}$, due to the generality of the potential and lack of similar identity \eqref{equation1.5}, the observability inequality similar to \eqref{equation1.4} established in \cite{HS} is restricted to the case $x_1=x_2=0$ and $r_1r_2\sim T$, their proof is based on an operator type Nazarov uncertainty
principle and minimal escape velocity estimates. For potentials that are increasing to infinity when $|x|\mapsto\infty$, to our best knowledge, there is only one result which was established for the Hermite Schrödinger equation in \cite{HWW}. Further observability inequalities at two time points can be found in \cite{LW,YM} for the linear KdV equation and in \cite{LWS} for nonlinear Schrödinger equation.
\par Now back to our model \eqref{equation1.1}, which can be considered as a positive answer to question (a) and (c) mentioned above in dimension one, we will first establish the following identity (see Lemma \ref{L1} below) similar to \eqref{equation1.5}
\begin{equation}\label{equation1.7}
(2t)^{\frac{1}{2}}e^{\frac{i(\nu+1)\pi}{2}}e^{-i|x|^{2}/4t}u(t,x)=F_\nu(e^{i|\cdot|^{2}/4t}u_0)(x/2t),\,\,\,for\,\,all\,\,t>0,\,and\,\,u(t,x)\,\,solves\,\,\eqref{equation1.1},
\end{equation}
where $F_\nu$ denotes the well known Hankel transform
\begin{equation}\label{equation1.8}
F_\nu(f)(x):=\int_{0}^{\infty}\sqrt{xy}J_\nu(xy)f(y)dy,\  x\in \mathbb{R}^{+},\,\,f\in L^{2}(\mathbb{R}^{+}).
\end{equation}
With identity \eqref{equation1.7} in hand, we can extend all the results in \cite{WWZ} to our equation \eqref{equation1.1}. And to our best knowledge, there is no such identity in the higher dimension, which restricts the discussion of this article to dimension one.
\par There are three main theorems in this paper. The first one gives an observability inequality at two points in time for the equation \eqref{equation1.1}.
\begin{theorem}\label{T1}
Let $A, B$ be two measurable sets in $\mathbb{R}^{+}$ with finite measure and $T>S\geq0$, then for every $\nu\geq0$, there is a positive constant $C=C(\nu,A,B,T-S)$, such that for all $u$ solving equation \eqref{equation1.1} with $u_0\in L^{2}(\mathbb{R}^{+})$, we have
\begin{equation}\label{T1.1}
\int_{\mathbb{R}^{+}}|u_0(x)|^{2}dx\leq C\left( \int_{A^{c}}|u(x,S;u_0)|^{2}dx+\int_{B^{c}}|u(x,T;u_0)|^{2}dx\right) .
\end{equation}
In particular,  in the following three cases, the constant $C(\nu,A,B,T-S)$ can be more explicit:
\begin{itemize}
\item[(i)] If $\nu=\frac{k}{2},\,k\in\{0,1,...\}$ , then there exists a constant C depending only on $\nu$ such that
\begin{equation}\label{T1.03}
\int_{\mathbb{R}^{+}}|u_0(x)|^{2}dx\leq Ce^{C\frac{\mu_{\nu}(A)\mu_{\nu}(B)}{(T-S)^{2(\nu+1)}}}\left( \int_{A^{c}}|u(x,S;u_0)|^{2}dx+\int_{B^{c}}|u(x,T;u_0)|^{2}dx\right) ,
\end{equation}
where $\mu_{\nu}(A)=\int_{A}x^{2\nu+1}dx$.
\item[(ii)] For general $\nu$, if $|A||B|<C_\nu=\left( 2\frac{\varGamma(2\nu)}{\varGamma(\nu+\frac{1}{2})}+2^{\nu+1}\right)^{-2}$, then
\begin{equation}\label{T1.02}
\int_{\mathbb{R}^{+}}|u_0(x)|^{2}dx\leq \frac{2\sqrt{C_{\nu}(T-S)}-\sqrt{\pi|A||B|}}{\sqrt{C_{\nu}(T-S)}-\sqrt{\pi|A||B|}}\left( \int_{A^{c}}|u(x,S;u_0)|^{2}dx+\int_{B^{c}}|u(x,T;u_0)|^{2}dx\right) .
\end{equation}
\item[(iii)] If $A=[0,a]$, $B=[0,b]$, $a,b>0$, then there exists a positive constant $C$ depending only on $\nu$ such that
 \begin{equation}\label{T1.01}
\int_{\mathbb{R}^{+}}|u_0(x)|^{2}dx\leq Ce^{C(1+\frac{ab}{T-S})}\left( \int_{[0,a]^{c}}|u(x,S;u_0)|^{2}dx+\int_{[0,b]^{c}}|u(x,T;u_0)|^{2}dx\right).
\end{equation}
\end{itemize}
\end{theorem}
The second one gives a unique continuation inequality at one time point for equation (\ref{equation1.1}) when the initial data have exponential decay at infinity.  It is interesting that this class of initial data is consistent with the free case and independent of the constant $\nu$.
\begin{theorem}\label{T3}
Given $\lambda,b,T>0$, the following conclusions hold: 
\begin{itemize}
\item[(i)] There exist constants $C=C(\nu)>0$ and $\theta=\theta(\nu)\in(0,1)$  such that for any $u_0\in C_0^{\infty}(\mathbb{R}^{+})$, we have\\
\\$\int_{\mathbb{R}^{+}}|u_0(x)|^{2}dx$\\
\begin{equation}\label{T3.1}
\begin{split}
\leq C\left( 1+\frac{b^{2\nu+2}}{(\lambda T)^{2\nu+2}}\right) 
\left( \int_{[0,b]^{c}}|u(x,T;u_0)|^{2}dx\right) ^{\theta^{1+b/(\lambda T)}} \left( \int_{\mathbb{R}^{+}}e^{\lambda x}|u_0(x)|^{2}dx\right) ^{1-\theta^{1+b/(\lambda T)}}.
\end{split}
\end{equation}
\item[(ii)] There exists a constant $C=C(\nu)>0$ such that for any $\beta>1$ and  $\gamma\in(0,1)$ and all $u_0\in C_0^{\infty}(\mathbb{R}^{+})$, we have
\begin{equation}\label{T3.2}
	\begin{split}
		\int_{\mathbb{R}^{+}}|u_0(x)|^{2}dx\leq Ce^{(\frac{C^{\beta}b^{\beta}}{\lambda(1-\gamma)T^{\beta}})^{1/(\beta-1)}}\left( \int_{[0,b]^c}|u(x,T;u_0)|^{2}dx\right) ^{\gamma}\left( \int_{\mathbb{R}^{+}}e^{\lambda x^{\beta}}|u_0(x)|^{2}dx\right) ^{1-\gamma}.
	\end{split}
\end{equation}
\item[(iii)] Let $\alpha(s)$, $s\in \mathbb{R}^{+}$, be an increasing function with $\lim_{s\rightarrow\infty}\alpha(s)/s=0$. Then for each $\gamma \in (0,1)$, there is no positive constant $C$ such that for any $u_{0}\in C_{0}^{\infty}(\mathbb{R}^{+})$,
\begin{equation}\label{T3.30}
\int_{\mathbb{R}^{+}}|u_0(x)|^{2}dx\leq C\left( \int_{[0,b]^c}|u(x,T;u_0)|^{2}dx\right) ^{\gamma}\left( \int_{\mathbb{R}^{+}}e^{\lambda \alpha(x)}|u_0(x)|^{2}dx\right) ^{1-\gamma}.
\end{equation}
\end{itemize}
\end{theorem}
The third one gives another unique continuation inequality at one time point for equation (\ref{equation1.1}) when the initial data have compact support. In addition, in what follows, $a\wedge b:=min\{a,b\}$. 
\begin{theorem}\label{T4}
Let $A=[a_{1}, a_{2}]$, $B=[b_{1},b_{2}]$, $a=a_{2}-a_{1}$, $b=b_{2}-b_{1}$, and $\lambda,T>0$, there exist constants $C=C(\nu)>0$ and $\theta=\theta(\nu)\in(0,1)$ such that for all  $u_0\in C_0^{\infty}(\mathbb{R}^{+})$:
\begin{equation}\label{T4.1}
\begin{split}
\int_{A}|u(x,T;u_0)|^{2}dx&\leq C(a_{2}^{2\nu+2}-a_{1}^{2\nu+2})((\lambda T)\wedge b)^{-(2\nu+2)}\\
 &~~~~\times\left(\int_{B}|u(x,T;u_0)|^{2}dx\right)^{\theta^{p}}\left( \int_{\mathbb{R}^{+}}e^{\lambda x}|u_0(x)|^{2}dx\right)^{1-\theta^{p}},
\end{split}
\end{equation}
with  
\begin{equation*}
p:=1+\frac{|x_0-x_1|+\frac{a}{2}+\frac{b}{2}}{(\lambda T)\wedge \frac{b}{2}}, 
\end{equation*} 
where $x_{0}$, $x_{1}$ be the center of $A$, $B$ respectively.
\end{theorem}
We make the following remarks related to the above three theorems:
\begin{itemize}
\item[($\textbf{a}_{1}$)] Theorem \ref{T1} can be explained in the following two perspectives. From the unique continuation perspective, Theorem \ref{T1} is a unique continuation inequality at two time points for equation (\ref{equation1.1}). From (\ref{T1.1}), we find that
\begin{equation*}
u(x,S;u_{0})=0~on~A^{c},~~u(x,T;u_{0})=0~on~ B^{c}~~\Longrightarrow~u(x,t;u_{0})=0~on ~\mathbb{R}^{+}\times[0,\infty).
\end{equation*}
From the observability perspective, Theorem \ref{T1} is an observability inequality at two time points for equation (\ref{equation1.1}). Observing a solution outside finite measurable sets at two different times, one can recover the solution at any time.\\
In addition, we elaborate on the sharpness of Theorem \ref{T1} in the following sense: First, we can't expect to recover the solution by observing it at two different points in time, one point outside a bounded interval while the other inside a bounded interval (see Theorem \ref{T6}(i)). Second, we can't expect to recover the solution by observing it at one point in time outside a bounded interval and the other in a time interval $[0,T]$ inside a bounded interval (see Theorem \ref{T6}(ii)). Third, we can't expect to recover the solution at one point in time, one point over a subset $A\subset \mathbb{R}^{+}$  with $m(A^{c})>0$ (see Theorem \ref{T6}(iii)).
\item[($\textbf{a}_{2}$)] Theorem \ref{T3} is a unique continuation inequality at one time point for equation (\ref{equation1.1}).\\
From (\ref{T3.1}), we find that
\begin{equation*}
e^{\lambda x/2}u_{0}(x)\in L^{2}(\mathbb{R}^{+}),~~u(x,T;u_{0})=0~on~ [0,b]^{c}~~\Longrightarrow~u(x,t;u_{0})=0~on ~\mathbb{R}^{+}\times[0,\infty).
\end{equation*} 
From (\ref{T3.2}), we find that when $\beta>1$,
\begin{equation*}
e^{\lambda x^{\beta}/2}u_{0}(x)\in L^{2}(\mathbb{R}^{+}),~~u(x,T;u_{0})=0~on~ [0,b]^{c}~~\Longrightarrow~ u(x,t;u_{0})=0~on ~\mathbb{R}^{+}\times[0,\infty).
\end{equation*} 
In addition, we elaborate on the sharpness of Theorem \ref{T3} in the following sense: First, when $\beta\in (0,1)$, there is no constant $C$ such that the inequality (\ref{T3.2}) holds (see (\ref{T3.30})). Hence, when we expect by observing solutions at one time point and outside a bounded interval, we have requirements for the decay rate of the initial data. Second, when $[0,b]^{c}$ is replaced by $[0,b]$, (\ref{T3.1}) does not hold (see Theorem \ref{T7}(i)). So even though the initial data have exponential decay at infinity, we can't expect to recover the solution by observing it at one time point and inside a bounded interval.
\item[($\textbf{a}_{3}$)] Theorem \ref{T4} is a unique continuation inequality at one time point for equation (\ref{equation1.1}).\\
From (\ref{T4.1}), we find that
\begin{equation*}
e^{\lambda x/2}u_{0}(x)\in L^{2}(\mathbb{R}^{+}),~~u(x,T;u_{0})=0~on~B=[b_{1},b_{2}]~~\Longrightarrow~ u(x,t;u_{0})=0~on ~\mathbb{R}^{+}\times[0,\infty).
\end{equation*}
For unique continuation properties of Schr\"{o}dinger equations, we refer the readers to \cite{EKPV1,EKPV2,IK,SI} and the references therein.
\end{itemize}

We next present three consequences of the above main theorems.
\begin{theorem}\label{T5}
Let $b$, $T$, $N>0$, there exists a constant $C=C(\nu)>0$ such that for all $u_{0}\in L^{2}(\mathbb{R}^{+})$ with supp $u_{0}\subset [0,N]$:
	\begin{equation}\label{T5.1}
	\int_{\mathbb{R}^{+}}|u_{0}(x)|^{2}dx\leq e^{C(1+\frac{bN}{T})}\left( \int_{[0,b]^{c}}|u(x,T;u_0)|^{2}dx\right).
	\end{equation}
\end{theorem}
\begin{theorem}\label{T8}
 Let $B=[b_{1},b_{2}]\subset \mathbb{R}^{+}$, $b=b_{2}-b_{1}$, $\lambda_{1}$, $\lambda_{2}$, $T>0$, there exists a constant $C=C(\nu)>0$ such that for all $u_{0}\in C_{0}^{\infty}(\mathbb{R}^{+})$ and $\varepsilon\in (0,1)$:\\
 \\$\int_{\mathbb{R}^{+}}e^{-\lambda_{2}x}|u(x,T;u_0)|^{2}dx$
 \begin{equation}\label{T8.1}
 \leq
 C(x_{0},b,\lambda_{1},\lambda_{2},T)\left( \varepsilon\int_{\mathbb{R}^{+}}e^{\lambda_{1}x}|u_{0}(x)|^{2}dx+\varepsilon e^{\varepsilon^{-1-\frac{C\lambda_{2}^{-1}}{(\lambda_{1}T)\wedge \frac{b}{2}}}}\int_{B}|u(x,T;u_0)|^{2}dx\right) ,
 \end{equation}
 where $C(x_{0},b,\lambda_{1},\lambda_{2},T):=exp\left\lbrace C\left(  1+\frac{x_{0}+\frac{b}{2}+\lambda_{2}^{-1}}{(\lambda_{1}T)\wedge \frac{b}{2}}\right)  \right\rbrace $, $x_{0}$ be the center of $B$.
 \end{theorem}
 \begin{theorem}\label{T9}
  Let $B=[b_{1},b_{2}]\subset \mathbb{R}^{+}$, $b=b_{2}-b_{1}$, $\lambda$, $T>0$, there exists a constant $C=C(\nu)>0$ such that for all $u_{0}\in C_{0}^{\infty}(\mathbb{R}^{+})$ and $\varepsilon\in (0,1)$: \\
  \\$\int_{\mathbb{R}^{+}}|u_0(x)|^{2}dx$
  \begin{equation}\label{T9.1}
  \begin{split}
  &\leq
  C(x_{0},b,\lambda,T)\left( \varepsilon\left( \int_{\mathbb{R}^{+}}e^{\lambda x}|u_{0}(x)|^{2}dx+\lVert u_{0}\rVert^{2}_{H^{4([\nu]+3)}(\mathbb{R}^{+})}+\int_{0}^{\infty}\frac{1}{x^{4([\nu]+3)}} |u_{0}|^{2}dx\right)\right. \\
  &~~~~~~~~~~~~~~~~~~~~~~~~~~~~~~~~~~~~~~~~~~~~~~~~~~~~~~~~~~~~~~~~~~~~~ \left. +\varepsilon e^{\varepsilon^{-2}}\int_{B}|u(x,T;u_0)|^{2}dx\right) ,
   \end{split}
  \end{equation}
  where $C(x_{0},b,\lambda,T):=\left( T+\frac{1}{T} \right) ^{[\nu]+3}(1+T)^{4([\nu]+3)}e^{C^{1+ \frac{x_{0}+\frac{b}{2}+1}{(\lambda T)\wedge \frac{b}{2}}}} $, $x_{0}$ be the center of $B$, $[\nu]$ stands for  the integral part of $\nu$ .
  \end{theorem}
  Finally, we should point out that, although the proofs in this article are inspired by their approach in \cite{WWZ} for the free Schrödinger equation, but technically, we need to shift from discussing Fourier transform to discussing Hankel transform which will create some new difficulties. Such as, to prove observability inequalities at two time points similar to \eqref{equation1.4}, we need to prove a Nazarov type uncertainty principle (see Lemma \ref{L4} below) for the Hankel transform $F_\nu$. To prove Theorems \ref{T3}\textendash\ref{T4}, we need to establish the interpolation Lemma \ref{L7}. To prove Theorem \ref{T9}, we need a regularity estimate Lemma \ref{L.S.2} for equation \eqref{equation1.1}, etc. And due to the complexity of proving these corollaries \ref{T5}\textendash\ref{T9}, for the completeness of this article, we will provide all the details of their proof.
 \par {\bf Plan of the paper.} The rest of the paper is organized as follows: Section 2 provides some preliminaries, these preliminaries will appear in the proofs of our theorems here and there. In Section 3, we give the proofs of Theorems \ref{T1}\textendash\ref{T4}. In Section 4, we show the proofs of Theorems \ref{T5}\textendash\ref{T9}.  Section 5 provides the sharpness results of Theorems \ref{T1} and \ref{T3}. Section 6 is devoted to applications to controllability for the Schr\"{o}dinger equation.

\section{\bf Preliminaries}
In this section, we give the exact meaning of equation \eqref{equation1.1} and give the proof of the identity \eqref{equation1.7}. Our equation \eqref{equation1.1} involve Schr\"{o}dinger operators:
\begin{equation}\label{Sec2.3.1}
 \displaystyle{H_\alpha=-\partial_{x}^{2}+\frac{\alpha}{x^2}},\,\,\,\alpha \geq -\frac{1}{4},
 \end{equation}
where $\alpha=\nu^{2}-\frac{1}{4}$. The well-known classical Hardy inequality (see \cite{H}) says, for every $u\in C_0^{\infty}(\mathbb{R}^{+})$:
\begin{equation}\label{Sec2.3.2}
\frac{1}{4}\int_0^{\infty}\frac{|u(x)|^{2}}{x^{2}}dx\leq\int_{0}^{\infty}|u'(x)|^{2}dx,
\end{equation}
and $\frac{1}{4}$ is the best constant. 
\par By this inequality we get that $-\partial_{x}^{2}+\frac{\alpha}{x^2}$ is form-bounded from below on $C_0^{\infty}(\mathbb{R}^{+})$ if and only if $\alpha \geq -\frac{1}{4}$. So there exists Friedrichs extension of this operator. And for $\nu\geq1$, the differential operator $L_\nu=-\partial_{x}^{2}+(\nu^{2}-\frac{1}{4})\frac{1}{x^2}$ with
domain $C_0^{\infty}(\mathbb{R}^{+})$ is essentially self-adjoint, and we denote by $H_\nu$ its closure, which exactly the Friedrichs extension. We note that if $0\leq\nu<1$, the operator $L_\nu$ is not essentially self-adjoint. And if $0<\nu<1$, this operator has exactly two distinct homogeneous extensions which are precisely the operators $H_\nu$ and $H_{-\nu}$: they are the Friedrichs and Krein extension of $L_\nu$ respectively. Throughout this paper we only consider the Friedrichs extension $H_\nu$ of $-\partial_{x}^{2}+(\nu^{2}-\frac{1}{4})\frac{1}{x^2}$ on $C_0^{\infty}(\mathbb{R}^{+})$, it is equivalent to the self-adjoint extension of $-\partial_{x}^{2}+(\nu^{2}-\frac{1}{4})\frac{1}{x^2}$ in $L^{2}(\mathbb{R}^{+})$ with Dirichlet boundary condition at $x=0$. We refer the readers to \cite[Section X.3]{RS1}
for the general theory of such extensions. We also refer the readers to \cite{LJV,JS} for extensions for more general parameter $\nu\in\mathbb{C}$ and some properties of that family of the operators $H_v$.
\begin{lemma}\label{L1}
$F_\nu$ is a unitary involution on $L^{2}(0,\infty)$ diagonalizing  $H_\nu$, precisely,
\begin{equation}\label{L1.1}
F_\nu H_\nu F_\nu^{-1}=Q^{2},
\end{equation}
where $Q^{2}f(x)=x^{2}f(x)$ and $\nu\geq0$.
\end{lemma}
With Lemma \ref{L1} in hand, we can get the integral representation of the unitary group $e^{-itH_\nu}$. The proof of the following result can be seen in the process of the proof of Theorem 2.4 in \cite{KT}, which deals with the dispersive estimate of \eqref{equation1.1}.
\begin{lemma}\label{L3}
For every $f\in L^{2}(\mathbb{R}^{+})$, we have
\begin{equation}\label{L3.1}
\begin{aligned}
(e^{-itH_\nu}f)(x)&=\frac{1}{2it}\int_0^{\infty}\sqrt{xy}J_\nu\left( \frac{xy}{2t}\right) e^{-\frac{x^{2}+y^{2}}{4it}}e^{\frac{-i\nu\pi}{2}}f(y)dy\\
&=(2t)^{-\frac{1}{2}}e^{-\frac{i(\nu+1)\pi}{2}}e^{-\frac{x^{2}}{4it}}F_\nu(e^{-\frac{y^{2}}{4it}}f(y))(x/2t).
\end{aligned}
\end{equation}
\end{lemma}
\begin{proof}
Let $f\in C_0^{\infty}(\mathbb{R}^{+})$. By the spectral theorem in \cite[Thm.3.1]{G}
\begin{equation}
(e^{-itH_\nu}f)(x)=\lim_{\varepsilon\rightarrow 0+}(e^{-(\varepsilon+it)H_\nu}f)(x).
\end{equation}
In view of (\ref{L1.1}), we thus get
\begin{equation}\label{L3.2}
\begin{split}
\lim_{\varepsilon\rightarrow 0+}(e^{-(\varepsilon+it)H_\nu}f)(x)&=\lim_{\varepsilon\rightarrow 0+}(F_\nu e^{-(\varepsilon+it)p^{2}}F_\nu^{-1}f)(x)\\
&=\lim_{\varepsilon\rightarrow 0+}\int_0^{\infty}\sqrt{xp}J_\nu(xp)e^{-(\varepsilon+it)p^{2}}F_\nu(f)(p)dp\\
&=\lim_{\varepsilon\rightarrow 0+}\int_0^{\infty}\sqrt{xp}J_\nu(xp)e^{-(\varepsilon+it)p^{2}}\int_0^{\infty}\sqrt{py}J_\nu(py)f(y)dydp\\
&=\lim_{\varepsilon\rightarrow 0+}\int_0^{\infty}\sqrt{x}\int_0^{\infty}p^{\frac{1}{2}}J_\nu(xp)e^{-(\varepsilon+it)p^{2}}\sqrt{yp}J_\nu(yp)dpf(y)dy\\
&=\lim_{\varepsilon\rightarrow 0+}\frac{1}{2(\varepsilon+it)}\int_0^{\infty}\sqrt{xy}I_\nu\left( \frac{xy}{2(\varepsilon+it)}\right) e^{-\frac{x^{2}+y^{2}}{4(\varepsilon+it)}}f(y)dy.
\end{split}
\end{equation}
The last equality follows from the following equation in \cite{A4},
\begin{equation}
F_\nu(\sqrt{x}e^{-\beta x^{2}}J_\nu(ax))(y)=\frac{\sqrt{y}}{2\beta}exp\left( -\frac{a^{2}+y^{2}}{4\beta}\right) I_\nu\left( \frac{ay}{2\beta}\right),~~for~Re \beta>0,~Re \nu>-1.
\end{equation}
Moreover, from \cite{AS}, it follows that the function $ I_\nu\left( \frac{xy}{2(\varepsilon+it)}\right) e^{-\frac{x^{2}+y^{2}}{4(\varepsilon+it)}}$ is bounded on every compact interval uniformly with respect to $\varepsilon>0$. Since the support of $f$ is compact, we can use the dominated convergence theorem and interchange the limit and integration in (\ref{L3.2}). Taking the limit $\varepsilon\rightarrow 0$ and using the identity $I_\nu(iz)=e^{-i\nu\pi/2}J_\nu(z)$, and then using the definition of Hankel transform $F_\nu$, we obtain (\ref{L3.1}).
\end{proof}
To proceed, we also need the following modified definition of Hankel transform $H_\nu$, which is defined by formula
\begin{equation*}
H_\nu(f)(x):=\int_{0}^{\infty}(xy)^{-\nu}J_\nu(xy)f(y)d\mu_\nu(y), \ x\in \mathbb{R}^{+},
\end{equation*}
where $d\mu_\nu(y)=y^{2\nu+1}dy$. It is well known that $H_\nu$ extends to an isometric isomorphism of $L^{2}(\mathbb{R}^{+}, x^{2\nu+1} dx)$ onto itself, with symmetric inverse: $H_\nu^{-1}=H_\nu$. It is easily verified by the definitions of $F_\nu$ and $H_\nu$ that
\begin{equation}\label{Sec2.1}
F_\nu(f)(x)=x^{\nu+\frac{1}{2}}H_\nu(y^{-\nu-\frac{1}{2}}f),\ \forall y\in \mathbb{R}^{+}.
\end{equation}

\section{Proofs of Theorems \ref{T1}\textendash Theorem \ref{T4}}
In this section, we will prove Theorem \ref{T1}\textendash\ref{T4}.
\subsection{Proof of Theorem \ref{T1}}
 This subsection is devoted to proving Theorem \ref{T1}. First, we start with introducing the uncertainty principle of the modified Hankel transform $H_{\nu}$. Second, by the relationship \eqref{Sec2.1}, then we get the uncertainty principle for Hankel transform $F_{\nu}$ in Lemma \ref{L4}. Then we show the equivalence between the uncertainty principle and the observability at two points in time in Lemma \ref{L5}. Finally, we give the proof of Theorem \ref{T1}. In addition, we also show in Corollary \ref{C3} the standard observability inequality \eqref{equation1.3} with the observable set $\Omega=\{x\in \mathbb{R}^{+}:\,x\geq r\}$.
\begin{theorem}\label{T2}
[Theorem 4.3 in \cite{GJ}] Let $A$, $B$ be a pair of measurable subsets of $\mathbb{R}^{+}$ with $0<|A|,|B|<\infty$. Then there is a positive constant $C=C(\nu,A,B)$ such that for each $f\in L^{2}_\nu(\mathbb{R}^{+})$,
\begin{equation}\label{UP}
\lVert f\rVert_{L^{2}_{\nu}(\mathbb{R}^{+})}\leq C(\lVert f\rVert_{L^{2}_{\nu}(A^{c})}+\lVert H_\nu(f)\rVert_{L^{2}_{\nu}(B^{c})}),
\end{equation}
where $A^{c}=\mathbb{R}^{+}\backslash A$, $|A|$ is the Lebesgue measure of set $A$, and $\lVert f\rVert_{L^{2}_{\nu}(\mathbb{R}^{+})}=(\int_{0}^{\infty}|f(x)|^{2}x^{2\nu+1}dx)^{\frac{1}{2}}$.
\end{theorem}
We give two remarks on the constant $C=C(\nu,A,B)$.
 \begin{remark}\label{R1}
 We point out that if $\nu=\frac{n}{2}-1$, $n=2,3,4 ...$, the constant has the following explicit form$:$
 \begin{equation*} C(\nu,S,\varSigma)=Ce^{C\mu_{\nu}(A)\mu_{\nu}(B)}. \end{equation*}
 In fact, recall the normalized Fourier transform
 is defined for $f\in L^{1}(\mathbb{R}^{n})$ by
 \begin{equation}
\mathcal{F}(f)(\xi)=\int_{\mathbb{R}^{n}}f(x)e^{-2\pi i\langle x,\xi\rangle}dx
 \end{equation}
and extended to $L^{2}(\mathbb{R}^{n})$ in the usual way. Thus, if $f(x)=g(|x|)$ is a radial function on $\mathbb{R}^{n}$, then $\mathcal{F}(f)(\xi)=H_{\frac{n}{2}-1}(g)(|\xi|)$. While for the Fourier transform, we already have \eqref{equation1.6} with constant $C(n,A_{n},B_{n})=C_{n}e^{C_{n}|A_{n}||B_{n}|}$. Now if we define $A$ and $B$ as
 \begin{equation*}
 A_{n}=\{x\in \mathbb{R}^{n}:|x|\in A\}~~ and~~ B_{n}=\{x\in \mathbb{R}^{n}:|x|\in B\},
 \end{equation*}
 then there exists a constant $C$ such that
 \begin{equation*}
 \lVert f\rVert_{L^{2}_{n/2-1}(\mathbb{R}^{+})}\leq Ce^{C\mu_{n/2-1}(A)\mu_{n/2-1}(B)}\left( \lVert f\rVert_{L^{2}_{n/2-1}(A^{c})}+\lVert H_{n/2-1}(f)\rVert_{L^{2}_{n/2-1}(B^{c})}\right).
 \end{equation*}
 \end{remark}
\begin{remark}\label{R2}
 Unfortunately, for the general case where $\nu\geq 0$, we can't have an explicit constant in Theorem \ref{T2} for general finite measure sets $A$ and $B$. But when the Lebesgue measure of $A$ and $B$ are small enough, precisely, by Lemma 4.2 in \cite{GJ}, if $|A||B|<k_\nu^{-2}$, the constant has the following explicit form:
\begin{equation*} C(\nu,A,B)=\left( 1+\frac{1}{1-k_\nu\sqrt{2\pi|A||B|}}\right) ,
\end{equation*}
where $k_\nu $ is the constant such that the following decay estimate is true:
\begin{equation}\label{R2.1}
\left| \frac{J_\nu(x)}{x^{\nu}}\right| \leq k_\nu x^{-\nu-\frac{1}{2}}.
\end{equation}
\end{remark}
We claim Theorem \ref{T2} is true for transform $F_\nu$, too.
\begin{lemma}\label{L4}
Let $A$, $B$ be a pair of measurable subsets of $\mathbb{R}^{+}$ with $0<|A|,|B|<\infty$. Then there is a positive constant $C=C(\nu,A,B)$ such that for each $f\in L^{2}(\mathbb{R}^{+})$,
\begin{equation}\label{L4.1}
\lVert f\rVert_{L^{2}(\mathbb{R}^{+})}\leq C\left( \lVert f\rVert_{L^{2}(A^{c})}+\lVert F_\nu(f)\rVert_{L^{2}(B^{c})}\right) ,
\end{equation}
where $A^{c}=\mathbb{R}^{+}\backslash A$, $|A|$ is the Lebesgue measure of set $A$.
\end{lemma}
\begin{proof}
 It is easy to see that the operator $M_\nu f(x)=x^{-\nu-\frac{1}{2}}f(x)$ has the property that:\\
\begin{equation*}
M_\nu:L^{2}(\mathbb{R}^{+})\rightarrow L^{2}_{\nu}(\mathbb{R}^{+})~~~~~is~~unitary.
\end{equation*}
Then $\forall f(x)\in L^{2}(\mathbb{R}^{+})$, let $g(x)=M_\nu f(x)=x^{-\nu-\frac{1}{2}}f(x)\in L^{2}_{\nu}(\mathbb{R}^{+})$.\\
By (\ref{UP}), we have
\begin{equation}\label{L4.2}
	\lVert g\rVert_{L^{2}_{\nu}(\mathbb{R}^{+})}\leq C(\lVert g\rVert_{L^{2}_{\nu}(A^{c})}+\lVert F_\nu(g)\rVert_{L^{2}_{\nu}(B^{c})}),
\end{equation}
and notice that:
\begin{equation*}
\lVert g\rVert_{L^{2}_{\nu}}=\lVert f\rVert_{L^{2}(\mathbb{R}^{+})},~~
\lVert g\rVert_{L^{2}_{\nu}(A^{c})}=\lVert f\rVert_{L^{2}(A^{c})}.
\end{equation*}
By (\ref{Sec2.1}),
\begin{equation*}
\lVert H_\nu(g)\rVert_{L^{2}_{\nu}(B^{c})}=\lVert x^{-\nu-\frac{1}{2}}F_\nu(f)(x)\rVert_{L^{2}_{\nu}(B^{c})}=\lVert F_\nu(f)(x)\rVert_{L^{2}(B^{c})}.
\end{equation*}
Substitute all the terms in (\ref{L4.2}) by above equality, we get (\ref{L4.1}).
\end{proof}

\begin{lemma}\label{L5}
Let $A$ and $B$ be measurable subsets of $\mathbb{R}^{+}$. Then the following statements are equivalent:\\
(i) There exists a positive constant $C_1(\nu,A,B)$ such that for each $f\in L^{2}(\mathbb{R}^{+})$,
\begin{equation}\label{L5.1}
\int_{\mathbb{R}^{+}}|f(x)|^{2}dx\leq C_1(\nu,A,B)\left( \int_A |f(x)|^{2}dx+\int_B |F_\nu(f)(x)|^{2}dx\right) .
\end{equation}  
(ii) There exists a positive constant $C_2(\nu,A,B)$ such that for each $T>0$ and each $u_0\in L^{2}(\mathbb{R}^{+})$,
\begin{equation}\label{L5.2}
\int_{\mathbb{R}^{+}}|u_0(x)|^{2}dx\leq C_2(\nu,A,B)\left( \int_A |u_0 (x)|^{2}dx+\int_{2TB} |u(x,T;u_0)|^{2}dx\right) .
\end{equation}
Furthermore, when one of the above two statements holds, the constants $C_1(\nu,A,B)$ and $C_2(\nu,A,B)$ can be chosen to be the same.
\end{lemma}
\begin{proof}
The proof of this lemma is just the same as the Lemma 2.3 in \cite{WWZ}. We just mention that one  needs to replace formula (2.6) in \cite{WWZ} with formula (\ref{L3.1}) in this article.
\end{proof}
We are now in a position to prove Theorem \ref{T1}.
\begin{proof}
[\textbf{Proof of Theorem \ref{T1}.}]  Let $T>S\geq 0$. $A^{c}, B^{c}$ be two measurable sets in $\mathbb{R}^{+}$ with finite measure. By Lemma \ref{L4}, We have (\ref{L5.1}) with $(A,B)$ replaced by $\left( A,\frac{B}{2(T-S)}\right) $ and $C_1(\nu,A,B)$ replaced by $C\left( \nu,A^{c},\frac{B^{c}}{2(T-S)}\right) $, where  $C\left( \nu,A^{c},\frac{B^{c}}{2(T-S)}\right) $ is given in (\ref{L4.1}). Thus, we can apply Lemma \ref{L5} to get (\ref{L5.2}) with $(A,B)$ replaced by $\left( A,\frac{B}{2(T-S)}\right) $ and $C_2(\nu,A,B)$ replaced by $C\left( \nu,A^{c},\frac{B^{c}}{2(T-S)}\right) $. So we can have
\begin{equation}\label{T1.2}
\int_{\mathbb{R}^{+}}|u_0(x)|^{2}dx\leq C\left( \nu,A^{c},\frac{B^{c}}{2(T-S)}\right) \left( \int_{A}|u_0(x)|^{2}dx+\int_{B}|u(x,T-S;u_0)|^{2}dx\right) .
\end{equation}
Finally, by (\ref{T1.2}) we get
\begin{equation*}
\int_{\mathbb{R}^{+}}|u(x,S;u_0)|^{2}dx\leq C\left( \nu,A^{c},\frac{B^{c}}{2(T-S)}\right) \left( \int_{A}|u(x,S;u_0)|^{2}dx+\int_{B}|u(x,T;u_0)|^{2}dx\right) .
\end{equation*}
By the conservation law for the Schr\"{o}dinger equation, we get the inequality (\ref{T1.1}) in Theorem \ref{T1}.
\par The conclusions (i)-(ii) of the theorem are direct consequences of Lemma \ref{L4}, Lemma \ref{L5} combined with Remark \ref{R1} and Remark \ref{R2} respectively. For (ii), we only need to give an estimate of the constant $k_\nu$ appearing in inequality \eqref{R2.1}. In fact, it is well-known that the Bessel function has the following asymptotic behavior (see \cite[Appendix B.7]{GL}). When $\nu>1/2$, we have
\begin{equation*}
|J_{\nu}(x)|\leq 2\frac{(x/2)^{\nu}}{\varGamma(\nu+\frac{1}{2})\varGamma(\frac{1}{2})}2^{\nu-\frac{3}{2}}\left( \frac{\varGamma(2\nu)}{x^{2\nu}}+2^{\nu}\frac{\varGamma(\nu+\frac{1}{2})}{x^{\nu+\frac{1}{2}}}\right). 
\end{equation*}
When $1/2\geq \nu \geq 0$, we have
\begin{equation*}
|J_{\nu}(x)|\leq 2\frac{(x/2)^{\nu}}{x^{\nu+\frac{1}{2}}\varGamma(\frac{1}{2})}.
\end{equation*}
These estimates yield that for $\nu\geq0$ and $x\geq1$ we get
\begin{equation}\label{T1.P1}
|J_{\nu}(x)|\leq 2\left( \frac{\varGamma(2\nu)}{\varGamma(\nu+\frac{1}{2})}+2^{\nu}\right)  x^{-\frac{1}{2}}.
\end{equation}
For $\nu\geq0$ and $0\leq x<1$, by (\ref{L2.19}) below, we know that$$|J_{\nu}(x)|\leq x^{-\frac{1}{2}}.$$
This, along with (\ref{T1.P1}) indicates that we can take $k_\nu=2\left( \frac{\varGamma(2\nu)}{\varGamma(\nu+\frac{1}{2})}+2^{\nu}\right)$.
\par Finally, we prove the conclusion (iii) of the theorem. Let $U(t)=e^{-itH_{\nu}}$ be the unitary group given by the equation (\ref{equation1.1}). In other words, the solution of (\ref{equation1.1}) can be written as
\begin{equation*}
u(x,t)=U(t)u_{0}(x),~~x\in \mathbb{R}^{+},~t\in [0,\infty).
\end{equation*}
Let $\mathbbm{1}_{\leq a}$ and $\mathbbm{1}_{>a}$ be the characteristic functions of the sets $\{x\in\mathbb{R}^{+}:x\leq a\}$ and $\{x\in\mathbb{R}^{+}:x> a\}$, respectively. Then
\begin{equation}\label{T1(i).1}
\int_{\mathbb{R}^{+}} |u(x,T-S)|^{2}dx\leq 2\int_{\mathbb{R}^{+}}|U(T-S)(\mathbbm{1}_{\leq a}u_{0})|^{2}dx+2\int_{\mathbb{R}^{+}}|U(T-S)(\mathbbm{1}_{> a}u_{0})|^{2}dx.
\end{equation} 
For the first term, by Theorem \ref{T5} (the theorem will be proved later), we get 
\begin{equation}\label{T1(i).2}
\begin{split}
\int_{\mathbb{R}^{+}}|U(T-S)(\mathbbm{1}_{\leq a}u_{0})|^{2}dx
&\leq e^{C(1+\frac{ab}{T-S})}(\int_{[0,b]^{c}}|U(T-S)(\mathbbm{1}_{\leq a}u_{0})|^{2}dx)\\
&\leq 2e^{C(1+\frac{ab}{T-S})}\left( \int_{[0,b]^{c}}|u(x,T-S)|^{2}dx+\int_{[0,b]^{c}}|U(T-S)(\mathbbm{1}_{> a}u_{0})|^{2}dx\right),
\end{split}
\end{equation}
where $C=C(\nu)>0$ is an absolute constant.\\
Inserting (\ref{T1(i).2}) into (\ref{T1(i).1}), we have 
\begin{equation}\label{T1(i).3}
\int_{\mathbb{R}^{+}} |u(x,T-S)|^{2}dx\leq Ce^{C(1+\frac{ab}{T-S})}\left( \int_{[0,b]^{c}}|u(x,T-S)|^{2}dx+\int_{\mathbb{R}^{+}}|U(T-S)(\mathbbm{1}_{> a}u_{0})|^{2}dx\right) .
\end{equation}
By (\ref{T1(i).3}) and the conservation law, we have
\begin{equation*}
\int_{\mathbb{R}^{+}} |u(x,T)|^{2}dx\leq Ce^{C(1+\frac{ab}{T-S})}\left( \int_{[0,b]^{c}}|u(x,T)|^{2}dx+\int_{[0,a]^{c}}|u(x,S)|^{2}dx\right) .
\end{equation*}
Thus, the above leads to (\ref{T1.01}) and ends the proof of (iii).
\end{proof}
With the explicit constant in hand, we can get a sufficient condition on subset $\Omega$ such that the standard observability inequality \eqref{equation1.3} holds for $u(x,t)$ solving equation \eqref{equation1.1}, we take the case (iii) for example.
\begin{corollary}\label{C3}
Given $r, T>0$, there is $C=C(\nu)>0$ so that for all $u_{0}\in L^{2}(\mathbb{R}^{+})$,
\begin{equation}\label{C3.1}
\int_{\mathbb{R}^{+}} |u_{0}|^{2}dx\leq C(1+\frac{1}{T})e^{Cr^{2}(1+\frac{1}{T})}\int_{0}^{T}\int_{[0,r]^{c}}|u(x,t;u_{0})|^{2}dxdt.
\end{equation}
Thus, the exterior domain $[0,r]^{c}$ is an observable set at any time for (\ref{equation1.1}).
\end{corollary}
\begin{proof}
Fix $r,T>0$. According to (iii) of Theorem \ref{T1}(with $a=b=r$), there exists $C=C(\nu)$ so that when $0\leq s<t\leq T$,
\begin{equation}\label{C3.2}
\int_{\mathbb{R}^{+}}|u_0(x)|^{2}dx\leq Ce^{C(1+\frac{r^{2}}{t-s})}\left( \int_{[0,r]^{c}}|u(x,s;u_0)|^{2}dx+\int_{[0,r]^{c}}|u(x,t;u_0)|^{2}dx\right).
\end{equation}
Since $0\leq s<t\leq T$, and if $(s,t)\in[0,T/3]\times[2T/3,T]$, we have 
\begin{equation}\label{C3.3}
(t-s)\geq T/3.
\end{equation}
Integrating (\ref{C3.2}) with $s$ over $s\in [0,T/3]$ and $t$ over $t\in [2T/3,T]$, using (\ref{C3.3}), we obtain that\\
\\$\left(\frac{T}{3} \right)^{2}\int_{\mathbb{R}^{+}}|u_0(x)|^{2}dx$
\begin{equation*}
\begin{split}
&\leq C\int_{0}^{\frac{T}{3}}\int_{\frac{2T}{3}}^{T}e^{C(1+\frac{r^{2}}{T/3})}\left( \int_{[0,r]^{c}}|u(x,s;u_0)|^{2}dx+\int_{[0,r]^{c}}|u(x,t;u_0)|^{2}dx\right)dtds\\
&\leq \frac{CT}{3}e^{C(1+\frac{3r^{2}}{T})}\left( \int_{0}^{\frac{T}{3}} \int_{[0,r]^{c}}|u(x,s;u_0)|^{2}dxds+\int_{\frac{2T}{3}}^{T}\int_{[0,r]^{c}}|u(x,t;u_0)|^{2}dxdt\right)\\
&\leq \frac{CT}{3}e^{C(1+\frac{3r^{2}}{T})}\int_{0}^{T} \int_{[0,r]^{c}}|u(x,t;u_0)|^{2}dxdt.
\end{split}
\end{equation*}
From the above, we obtain 
\begin{equation*}
\int_{\mathbb{R}^{+}}|u_0(x)|^{2}dx\leq\frac{3C}{T}e^{C(1+\frac{3r^{2}}{T})}\int_{0}^{T} \int_{[0,r]^{c}}|u(x,t;u_0)|^{2}dxdt,
\end{equation*}
which leads to (\ref{C3.1}). This ends the proof of Corollary \ref{C3}.
\end{proof}

\subsection{Proofs of Theorems \ref{T3} and \ref{T4}}
The proofs of Theorems \ref{T3} and \ref{T4} are mainly based on an interpolation inequality (see Lemma \ref{L7} below). To proceed it, we need the following two lemmas. Lemma \ref{L6} below is a modified version of Theorem 1.3 in \cite{AE}.
\begin{lemma}\label{L6}
Let $f$ be analytic in $[a,b]$, $b>a\geq 0$, $E$ be a subinterval in $[a,b]$ and assume there are positive constants $M$ and $\rho$  such that
\begin{equation*}
|f^{k}(x)|\leq Mk!(\rho(b-a))^{-k},\,\,for\,\,k\geq 0,\,\,x\in[a,b],
\end{equation*}
then there are constants $N=N(\rho,\nu,\frac{|E|}{b-a})$ and $\gamma=\gamma(\rho,\nu,\frac{|E|}{b-a})$ such that
\begin{equation}\label{L6.1}
	\lVert f\rVert_{L^{\infty}([a,b])}\leq N\left( \mu_\nu(E)^{-1}\int_E |f|d\mu_\nu(x)\right) ^{\gamma} M^{1-\gamma},
\end{equation}
where measure $\mu_\nu(x)=x^{2\nu+1}dx$ and $|E|$ denotes the Lebesgue measure of $E$.
\end{lemma}
\begin{proof}
Define linear transform $Ax=(b-a)x+a$. Let $G(x)=f\circ A(x)$, $E_0=A^{-1}E$, then $G(x)$ is analytic in $[0,1]$, $|E_0|=\frac{|E|}{b-a}$. Then by Lemma 3.3 in \cite{AE}, there are $N=N(\rho,\frac{|E|}{b-a})$, and $\gamma=\gamma(\rho,\frac{|E|}{b-a})$ such that
\begin{equation}\label{L6.2}
	\lVert f\rVert_{L^{\infty}([a,b])}\leq N(\lVert f\rVert_{L^{\infty}(E)})^{\gamma} M^{1-\gamma}.
\end{equation}
Define 
\begin{equation*}
E_1=\{x\in E:\frac{|f(x)|}{2}\leq \mu_\nu(E)^{-1}\int_E |f|d\mu_\nu(x)\},
\end{equation*}
then we have
\begin{equation}\label{L6.3}
|E_1|\geq \frac{|E|}{4(\nu+1)},\,\,\lVert f\rVert_{L^{\infty}(E_1)}\leq 2\mu_\nu(E)^{-1}\int_E |f|d\mu_\nu(x).
\end{equation}
In fact,
\begin{equation*}
\int_{E\backslash E_1} |f|d\mu_\nu(x)\geq \int_{E\backslash E_1}\left( \frac{2}{\mu_\nu(E)}\int_E |f|d\mu_\nu(x)\right) d\mu_\nu(x)=\frac{2\mu_\nu(E\backslash E_1)}{\mu_\nu(E)}\int_E |f|d\mu_\nu(x). 
\end{equation*}
Then we have $\mu_\nu(E\backslash E_1)\leq \frac{\mu_\nu(E)}{2}$. It implies that 
$\mu_\nu(E_1)\geq \frac{\mu_\nu(E)}{2}$. Since $E\subset[a,b]$ is an interval, denote $E=[a_1,b_1]$, we thus get
\begin{equation}\label{eL}
\int_{E_1}x^{2\nu+1}dx\geq \frac{1}{2}\int_{[a_1,b_1]}x^{2\nu+1}dx.
\end{equation}
Since $x^{2\nu+1}$ is monotonically increasing with respect to $x$ and $E_1\subset[a_1,b_1]$, so the left hand of \eqref{eL} is controlled by $\int_{E_1}b_1^{2\nu+1}dt$, i.e.
\begin{equation}\label{eL1} 
	\int_{E_1}b_1^{2\nu+1}dx\geq \int_{E_1}x^{2\nu+1}dx.
\end{equation}
The right hand of \eqref{eL} is equal to
\begin{equation}\label{eL2}
	 \frac{b_1^{2\nu+2}-a_1^{2\nu+2}}{4(\nu+1)}.
\end{equation}
So by \eqref{eL1} and \eqref{eL2}, we have
\begin{equation*}
\int_{E_1}dx\geq \frac{1}{4(\nu+1)}\frac{b_1^{2\nu+2}-a_1^{2\nu+2}}{b_1^{2\nu+1}}.
\end{equation*}
By 
$$\frac{a_1^{2\nu+2}}{b_1^{2\nu+1}}\leq \frac{a_1^{2\nu+2}}{a_1^{2\nu+1}},$$
we have
\begin{equation*}\int_{E_1}dx\geq \frac{1}{4(\nu+1)}(b_1-a_1)
	\Leftrightarrow |E_1|\geq \frac{1}{4(\nu+1)}|E|.
\end{equation*}
We thus prove \eqref{L6.3}, finally, we use (\ref{L6.2}) for $E_1$ and then by (\ref{L6.3}) we get the assertion of the lemma and formula (\ref{L6.1}).
\end{proof}
\begin{remark}
For fixed $\nu$ and an interval $E=[a,b]\subset\mathbb{R}^{+}$, we make a note on the Lebesgue measure of $E$ and the $\mu_\nu$-measure of $E$. By inequality $a^{p}+b^{p}\leq 2(a+b)^{p},\,\,p\geq0$, we always have
$|E|^{2(\nu+1)}\leq 4(\nu+1)\mu_\nu(E)$, but the right side of the inequality may not necessarily be controlled by the left side except $b\geq a>0$. In fact, in case $b\geq a>0$, we have inequality $(a+b)^{p}\leq2^{p}(a^{p}+b^{p})\leq(2^{p+1}-1)b^{p}+a^{p}$, $p\geq1$, which yields 
\begin{equation*}
\frac{b^{2(\nu+1)}-a^{2(\nu+1)}}{2(\nu+1)}\leq\frac{(2^{2\nu+3}-1)(b-a)^{2(\nu+1)}}{2(\nu+1)},
\end{equation*}
that is,
\begin{equation*}
\mu_\nu(E)\leq\frac{(2^{2\nu+3}-1)|E|^{2(\nu+1)}}{2(\nu+1)}.
\end{equation*}
\end{remark}
\begin{lemma}\label{L}
For every $\nu\geq0$ and every positive integer $k\in \mathbb{N}^{+}$, we have estimate
\begin{equation}\label{L2.19}
\left| \frac{\partial^k}{\partial x^{k}}\left( \frac{J_\nu(xy)}{(xy)^{\nu}}\right) \right| \leq y^{k},\,\,\,\,\,\,\,\,\,\,\,\,x,\,\,y\in \mathbb{R}^{+}.
\end{equation}
\end{lemma}
\begin{proof}
By the Poisson representation formula, we have
\begin{equation*}
\frac{J_\nu(xy)}{(xy)^\nu}=\frac{1}{2^\nu\Gamma(\nu+\frac{1}{2})\Gamma(\frac{1}{2})}\int_{-1}^{+1}e^{ixys}(1-s^2)^{\nu}\frac{ds}{\sqrt{1-s^2}},
\end{equation*}
then for every positive integer $k\in \mathbb{N}^{+}$, we get
\begin{equation*}
\begin{aligned}
\left| \frac{\partial^k}{\partial x^{k}}\left( \frac{J_\nu(xy)}{(xy)^{\nu}}\right) \right| &=\left| \frac{(y)^{k}}{2^\nu\Gamma(\nu+\frac{1}{2})\Gamma(\frac{1}{2})}\int_{-1}^{+1}(is)^{k}e^{ixys}(1-s^2)^{\nu}\frac{ds}{\sqrt{1-s^2}}\right| \\
&\leq\frac{y^{k}}{2^\nu\Gamma(\nu+\frac{1}{2})\Gamma(\frac{1}{2})}\left| \int_{-1}^{+1}\frac{ds}{\sqrt{1-s^2}}\right| =\frac{\pi y^{k}}{2^\nu\Gamma(\nu+\frac{1}{2})\Gamma(\frac{1}{2})}\\
&\leq\frac{\pi y^{k}}{2^\nu\Gamma(\frac{1}{2})\Gamma(\frac{1}{2})}=\frac{ y^{k}}{2^\nu}\leq y^{k},\,\,for\,\,every\,\,x,\,y\in \mathbb{R}^{+},\,and\,\,\nu\geq 0.
\end{aligned}
\end{equation*}
\end{proof}
Next, we show an interpolation inequality for a class of $L^{2}$-functions whose Hankel transform have compact supports.
\begin{lemma}\label{L7}
Given any interval $A=[a_{1}, a_{2}]$, $B=[b_{1},b_{2}]$, $a=a_{2}-a_{1}>0$, $b=b_{2}-b_{1}>0$  and any constant $\lambda>0$. Then for each $f\in L^{2}(\mathbb{R}^{+})$ with $F_\nu(f)\in C_0^{\infty}(\mathbb{R}^{+})$, there exist absolute constants $C=C(\nu)>0$ and $\theta=\theta(\nu)\in(0,1)$ such that 
\begin{equation}\label{L7.1}
\begin{split}
\int_A|f(x)|^{2}dx&\leq C(a_{2}^{2(\nu+1)}-a_{1}^{2(\nu+1)})(\lambda^{-2(\nu+1)}+b^{-2(\nu+1)})\\ 
&~~~~~~~\times\left( \int_B|f(x)|^{2}dx\right) ^{\theta^{p}}\left( \int_{\mathbb{R}^{+}}|F_\nu(f)(y)|^{2}e^{\lambda y}dy\right) ^{1-\theta^{p}},
\end{split}
\end{equation}
where $p:=1+\frac{|x_0-x_1|+\frac{a}{2}+\frac{b}{2}}{\lambda\wedge \frac{b}{2}}$,  $x_0$, $x_1$ be the center of $A$ and $B$ respectively.
\end{lemma}
\begin{proof}
We organize the proof by two steps.\\
\textbf{Step 1.}
We show that (\ref{L7.1}) holds for $\lambda=1$.\\
For $F_\nu(f)\in C_0^{\infty}(\mathbb{R}^{+})$, we have
\begin{equation*}
 f(x)=\int_{\mathbb{R}^{+}}\sqrt{xy}J_\nu(xy)F_\nu(f)(y)dy,
\end{equation*}
then
\begin{equation*}
x^{-\nu-1/2}f(x)=\int_{\mathbb{R}^{+}}\frac{J_\nu(xy)}{(xy)^\nu}F_\nu(f)(y)y^{\nu+1/2}dy.
\end{equation*}
Now let $g(x)=x^{-\nu-1/2}f(x)$ and we know that $g(x)$ is an analytic even function, and for each $k\in \mathbb{N}^{+}$,
\begin{equation*}
\frac{\partial^{k}g(x) }{\partial x^{k}}=\int_{\mathbb{R}^{+}}\frac{\partial^k}{\partial x^{k}}\left( \frac{J_\nu(xy)}{(xy)^{\nu}}\right) F_\nu(f)(y)y^{\nu+1/2}dy.
\end{equation*}
Hence, by the H\"{o}lder inequality, for each $k\in \mathbb{N}^{+}$,
\begin{equation*}
\begin{aligned}
\left\|  \frac{\partial^{k}g}{\partial x^{k}}\right\| _{L^{\infty}(\mathbb{R}^{+})}&\leq \sqrt{\int_{\mathbb{R}^{+}}\left| \frac{\partial^{k}}{\partial x^{k}}\left( \frac{J_\nu(xy)}{(xy)^{\nu}}\right) \right| ^{2}e^{-y}y^{2\nu+1}dy}\sqrt{\int_{\mathbb{R}^{+}}|F_\nu(f)|^{2}e^{y}dy}\\
&\leq
\sqrt{\sqrt{\int_{\mathbb{R}^{+}}\left| \frac{\partial^{k}}{\partial x^{k}}\left( \frac{J_\nu(xy)}{(xy)^{\nu}}\right) \right| ^{4}e^{-y}dy}\sqrt{\int_{\mathbb{R}^{+}}e^{-y}y^{4\nu+2}dy}}\sqrt{\int_{\mathbb{R}^{+}}|F_\nu(f)|^{2}e^{y}dy},
\end{aligned}
\end{equation*}
then by \eqref{L2.19}, we have
\begin{equation*}
\begin{aligned}
\left\|  \frac{\partial^{k}g}{\partial x}\right\| _{L^{\infty}(\mathbb{R}^{+})}&\leq
\sqrt{\sqrt{\int_{\mathbb{R}^{+}}y^{4k}e^{-y}dy}\sqrt{\int_{\mathbb{R}^{+}}e^{-y}y^{4\nu+2}dy}}\sqrt{\int_{\mathbb{R}^{+}}|F_\nu(f)|^{2}e^{y}dy}\\ &=\Gamma(4\nu+3)^{\frac{1}{4}}(4k!)^{\frac{1}{4}}\sqrt{\int_{\mathbb{R}^{+}}|F_\nu(f)|^{2}e^{y}dy}.
\end{aligned}
\end{equation*}

We next claim that there is an absolute constant $C>1$ such that
\begin{equation*}
(4k!)^{\frac{1}{4}}\leq k!C^{k}, \,\,\,for\,\,all\,\,k\in \mathbb{N}^{+}.
\end{equation*}
In fact, using Stirling's approximation for factorials
\begin{equation*}
\ln(m!)=m\ln m-m+O(\ln m),\,\,\forall m\in \mathbb{N}^{+},
\end{equation*}
we see that for all $k\in \mathbb{N}^{+}$,
\begin{equation*}
\ln((4k!)^{\frac{1}{4}})=\frac{1}{4}(4k\ln(4k)-4k+O(\ln(4k)))=\ln k!+k\ln 4+O(\ln k).
\end{equation*}
Thus, there exists an absolute constant $C>1$ such that
\begin{equation*}
(4k!)^{\frac{1}{4}}\leq e^{[\ln k!+k\ln C]}=k!C^{k},\,\,\,for\,\,all\,\,k\in \mathbb{N}^{+}.
\end{equation*}
By this, we get
\begin{equation*}
\left\|  \frac{\partial^{k}g}{\partial x^{k}}\right\| _{L^{\infty}(\mathbb{R}^{+})}\leq \Gamma(4\nu+3)^{\frac{1}{4}}k!C^{k}\sqrt{\int_{\mathbb{R}^{+}}|F_\nu(f)|^{2}e^{y}dy}.
\end{equation*}
Let $A=[a_1,a_2],~B=[b_1,b_2]$, $x_0$, $x_1$ be the center of $A$ and $B$, and $a$, $b$ denote the length of $A,B$ respectively. $B_{r}(x_0)$ denotes the interval in $\mathbb{R}^{+}$ with center $x_0 $, length $2r$. Let 
\begin{equation}\label{L7.2}
M=\Gamma(4\nu+3)^{\frac{1}{4}}\sqrt{\int_{\mathbb{R}^{+}}|F_\nu(f)|^{2}e^{y}dy},~~~r_0=\frac{C^{-1}\wedge \frac{b}{2}}{5}<1,
\end{equation}
we get
\begin{equation*}
\left| \frac{\partial^{k}g}{\partial x^{k}}\right| \leq M\frac{k!}{(5r_0)^k}, ~~~~~~~x\in B_{2r_0}(x_1).
\end{equation*}
Then by Lemma \ref{L6}, there exist  constants $C_1=C_{1}(\nu)>0$ and $\theta=\theta(\nu)\in(0,1)$ such that
\begin{equation*}
\lVert g\rVert_{L^{\infty}(B_{2r_0}(x_1))}\leq C_1M^{1-\theta}(\mu_\nu(B_{r_0}(x_1))^{-1}\lVert g\rVert_{L_\nu^{1}(B_{r_0}(x_1))})^\theta.
\end{equation*}
The H\"{o}lder inequality yields
\begin{equation*}
\lVert g\rVert_{L^{\infty}(B_{2r_0}(x_1))}\leq C_1M^{1-\theta}(\mu_\nu(B_{r_0}(x_1))^{-\frac{1}{2}}\lVert g\rVert_{L_\nu^{2}(B_{r_0}(x_1))})^\theta.
\end{equation*}
Since $B_{r_0}(x_1)\subset B$ and $\lVert g\rVert_{L_\nu^{2}(B_{r_0}(x_1))}=\lVert f\rVert_{L^{2}(B_{r_0}(x_1))}$, we have
\begin{equation*}
\lVert g\rVert_{L^{\infty}(B_{2r_0}(x_1))}\leq C_1M^{1-\theta}(\mu_\nu(B_{r_0}(x_1))^{-\frac{1}{2}}\lVert f\rVert_{L^{2}(B)})^\theta.
\end{equation*}
Just as the proof in \cite{WWZ}, denote $D_r(z)$ for the closed disk in the complex plane, centered at $z$ and of radius $r$. It is clear that $D_{r_0}((k+1)r_0)\subset D_{2r_0}(kr_0)$, $k=1,2,...$
Define $G(s)=\frac{1}{M}g(x_1+s)$, $s\in \mathbb{R}^{+}$.
Then $G$ can be extended to an analytic function on
\begin{equation*}
\Omega_{r_{0}}:=\{x+iy\in \mathbb{C}:x,y\in \mathbb{R},|y|<5r_0\}.
\end{equation*}
In addition, $G$ has the property that
$\lVert G\rVert_{L^{\infty}(\Omega_{r_0})}\leq 1$. The function $G(4r_{0}z)$ is analytic on $D_1(0)$ and 
$\sup\limits_{z\in D_1(0)}|G(4r_0z)|\leq1$. Then apply Lemma 3.2 in \cite{AE} to find that there are constants $C_2>0$ and $\theta_1\in (0,1)$ such that
\begin{equation*}
\sup\limits_{z\in D_{1/2}(0)}|G(4r_0z)|\leq C_2\sup\limits_{x\in \mathbb{R},|x|\leq 1/5}|G(4r_0x)|^{\theta_{1}}.
\end{equation*}
We obtain
\begin{equation*}
\lVert G\rVert_{L^{\infty}(D_{2r_0}(0))}\leq C_2\left( \frac{1}{M}\lVert g\rVert_{L^{\infty}(B_{2r_0}(x_1))}\right) ^{\theta_{1}},
\end{equation*}
then we have
\begin{equation}\label{L7.3}
\lVert G\rVert_{L^{\infty}(D_{2r_0}(0))}\leq C_{2}C_{1}^{\theta_1}\mu_\nu(B_{r_0}(x_1))^{-\frac{\theta\theta_1}{2}}\left( \frac{1}{M}\lVert f\rVert_{L^2(B)}\right) ^{\theta\theta_1}.
\end{equation}

Meanwhile, we can apply the Hadamard three-circle theorem to deduce that for each $k=1,2,...$,
\begin{equation*}
\lVert G\rVert_{L^{\infty}(D_{2r_0}(kr_0))}\leq \lVert G\rVert^{1/2}_{L^{\infty}(D_{r_0}(kr_0))}\lVert G\rVert^{1/2}_{L^{\infty}(D_{4r_0}(kr_0))}\leq \lVert G\rVert^{1/2}_{L^{\infty}(D_{r_0}(kr_0))}.
\end{equation*}
We see that for each $k=1,2,...$,
\begin{equation*}
\lVert G\rVert_{L^{\infty}(D_{r_0}((k+1)r_0))}\leq \lVert G\rVert_{L^{\infty}(D_{2r_0}(kr_0))}\leq \lVert G\rVert^{1/2}_{L^{\infty}(D_{r_0}(kr_0))},
\end{equation*}
which implies that for each $k=1,2,...$,
\begin{equation*}
\lVert G\rVert_{L^{\infty}(D_{r_0}((k+1)r_0))}\leq \lVert G\rVert^{1/2}_{L^{\infty}(D_{r_0}(kr_0))}\leq ...\leq \lVert G\rVert^{(1/2)^{k}}_{L^{\infty}(D_{r_0}(r_0))}.
\end{equation*}
This yields
\begin{equation}\label{L7.4}
\begin{split}
\lVert G\rVert_{L^{\infty}(\cup_{1\leq k\leq n}D_{r_0}(kr_0))}&=\sup\limits_{1\leq k\leq n}\lVert G\rVert_{L^{\infty}(D_{r_0}(kr_0))}\leq \sup\limits_{1\leq k\leq n}\lVert G\rVert^{(1/2)^{k-1}}_{L^{\infty}(D_{r_0}(r_0))}\\
&\leq \sup\limits_{1\leq k\leq n}\lVert G\rVert^{(1/2)^{n-1}}_{L^{\infty}(D_{r_0}(r_0))}\leq  \lVert G\rVert^{(1/2)^{K}}_{L^{\infty}(D_{r_0}(r_0))}, 
\end{split}
\end{equation}
where $n$ is the integer such that
\begin{equation*}
nr_0\geq |x_0-x_1|+\frac{a}{2}+\frac{b}{2}>(n-1)r_0,
\end{equation*} 
and 
 \begin{equation*}
 K=\frac{|x_0-x_1|+\frac{a}{2}+\frac{b}{2}}{r_0},
 \end{equation*}
it follows that
\begin{equation*}
\left[ 0,|x_0-x_1|+\frac{a}{2}+\frac{b}{2}\right] \subset\bigcup_{1\leq k\leq n}D_{r_0}(kr_0))~~and~~ D_{r_0}(r_0)\subset D_{2r_0}(0).
\end{equation*}
We see from (\ref{L7.4}) that for all $s\in [0,|x_0-x_1|+\frac{a}{2}+\frac{b}{2}]$,
\begin{equation}\label{L7.5}
|G(s)|\leq \lVert G\rVert_{L^{\infty}(\cup_{1\leq k\leq n}D_{r_0}(kr_0))}\leq \lVert G\rVert^{(1/2)^{K}}_{L^{\infty}(D_{r_0}(r_0))}\leq \lVert G\rVert^{(1/2)^{K}}_{L^{\infty}(D_{2r_0}(0))}.
\end{equation}
From (\ref{L7.3}) and (\ref{L7.5}), we find that for all $s\in [0,|x_0-x_1|+\frac{a}{2}+\frac{b}{2}]$,
\begin{equation*}
	\begin{split}
|g(x_1+s)|&=M|G(s)|\leq M\lVert G\rVert^{(1/2)^{K}}_{L^{\infty}(D_{2r_0}(0))}\\
&\leq M\left( C_{2}C_{1}^{\theta_1}\mu_\nu(B_{r_0}(x_1))^{-\frac{\theta\theta_1}{2}}\left( \frac{1}{M}\lVert f\rVert_{L^2(B)}\right) ^{\theta\theta_1}\right)  ^{(1/2)^{K}}\\
&= \left( C_{2}C_{1}^{\theta_1}\mu_\nu(B_{r_0}(x_1))^{-\frac{\theta\theta_1}{2}}\right) ^{2^{-K}}M^{1-\frac{\theta\theta_1}{2^{K}}}\lVert f\rVert_{L^2(B)}^{\frac{\theta\theta_1}{2^{K}}}.
\end{split}
\end{equation*}
One can easily find that above inequality holds for $g(x_1-s)$, for all $s\in [0,|x_0-x_1|+\frac{a}{2}+\frac{b}{2}]$, too. We see that,
\begin{equation*}
\sup\limits_{|x-x_1|\leq |x_0-x_1|+\frac{a}{2}+\frac{b}{2}}|g(x)|\leq  \left( C_{2}C_{1}^{\theta_1}\mu_\nu(B_{r_0}(x_1))^{-\frac{\theta\theta_1}{2}}\right) ^{2^{-K}}M^{1-\frac{\theta\theta_1}{2^{K}}}\lVert f\rVert_{L^2(B)}^{\frac{\theta\theta_1}{2^{K}}}.
\end{equation*}
Since $A\subset \{x:|x-x_1|\leq |x_0-x_1|+\frac{a}{2}+\frac{b}{2}\}\cap \mathbb{R}^{+}$ and
\begin{equation*}
\sup\limits_{\{x:|x-x_1|\leq |x_0-x_1|+\frac{a}{2}+\frac{b}{2}\}\cap \mathbb{R}^{+}}|g(x)|\leq \sup\limits_{|x-x_1|\leq |x_0-x_1|+\frac{a}{2}+\frac{b}{2}}|g(x)|,
\end{equation*}
the above estimate yields
\begin{equation*}
\begin{split}
\int_A|f(x)|^{2}dx&=\int_A|g(x)|^{2}d\mu_\nu(x)\leq \mu_\nu(A)\sup\limits_{\{x:|x-x_1|\leq |x_0-x_1|+\frac{a}{2}+\frac{b}{2}\}\cap \mathbb{R}^{+}}|g(x)|^{2}\\
&\leq \mu_\nu(A)\left( C_{2}C_{1}^{\theta_1}\mu_\nu(B_{r_0}(x_1))^{-\frac{\theta\theta_1}{2}}\right) ^{2^{-(K-1)}}M^{2(1-\frac{\theta\theta_1}{2^{K}})}\lVert f\rVert_{L^2(B)}^{\frac{2\theta\theta_1}{2^{K}}}.
\end{split}
\end{equation*}
We observe that for $I=[a,b]$,
\begin{equation*}
\mu_\nu([a,b])=\int_a^{b}x^{2\nu+1}dx=\frac{b^{2(\nu+1)}-a^{2(\nu+1)}}{2(\nu+1)}.
\end{equation*}
By inequality $(a^{p}+b^{p})\leq (a+b)^{p},$ $ b>0,$ $a>0,$ $ p\geq 1$, we get
\begin{equation*}
\frac{b^{2(\nu+1)}-a^{2(\nu+1)}}{2(\nu+1)} \geq \frac{(b-a)^{2(\nu+1)}}{2(\nu+1)}
 ~~~~~i.e.~~~~ \mu_\nu(I)\geq\frac{|I|^{2\nu+2}}{2(\nu+1)},
\end{equation*} 
where $|I|$ is the length of $I$, for any interval $I\subset \mathbb{R}^{+}$. We get 
\begin{equation*}
	\begin{split}
	\int_A|f(x)|^{2}dx&\leq \mu_\nu(A)\left( C_{2}C_{1}^{\theta_1}\left( \frac{(2r_0)^{2\nu+2}}{2(\nu+1)}\right) ^{-\frac{\theta\theta_1}{2}}\right) ^{2^{-(K-1)}}M^{2(1-\theta\theta_1/2^{K})}\lVert f\rVert_{L^2(B)}^{2\theta\theta_1/2^{K}}\\
		 &\leq\mu_\nu(A)(C_{2}C_{1}^{\theta_1}r_0^{-(\nu+1)}+1)^{2}M^{2(1-\theta\theta_1/2^{K})}\lVert f\rVert_{L^2(B)}^{2\theta\theta_1/2^{K}}\\
		&= \frac{(a_{2}^{2(\nu+1)}-a_{1}^{2(\nu+1)})}{2(\nu+1)}(C_{2}C_{1}^{\theta_1}r_0^{-(\nu+1)}+1)^{2}M^{2(1-\theta\theta_1/2^{K})}\lVert f\rVert_{L^2(B)}^{2\theta\theta_1/2^{K}}.	
	\end{split}
\end{equation*}

Finally by $M\geq\lVert f\rVert_{L^{2}(B)}$ and  $r_0\geq \frac{C^{-1}(1\wedge \frac{b}{2})}{5}$, we have
\begin{equation*}
	\begin{split}
	\int_A|f(x)|^{2}dx&\leq\frac{(a_{2}^{2(\nu+1)}-a_{1}^{2(\nu+1)})}{2(\nu+1)}(1+C_3) ^{2}(5C)^{2\nu+2}\left( \left( 1\wedge \frac{b}{2}\right) ^{-(\nu+1)}+1\right) ^2M^{2}\left( \frac{\lVert f\rVert^{2}_{L^{2}(B)}}{M^{2}}\right) ^{\alpha_1}\\
	&\leq 4\frac{(a_{2}^{2(\nu+1)}-a_{1}^{2(\nu+1)})}{2(\nu+1)}(1+C_3) ^{2}(5C)^{2\nu+2}2^{2\nu+2}(b^{-2(\nu+1)}+1)M^{2}\left( \frac{\lVert f\rVert^{2}_{L^{2}(B)}}{M^{2}}\right) ^{\alpha_1}\\
	&\leq C_4^{2(\nu+1)}(a_{2}^{2(\nu+1)}-a_{1}^{2(\nu+1)})(b^{-2(\nu+1)}+1)M^{2}\left( \frac{\lVert f\rVert^{2}_{L^{2}(B)}}{M^{2}}\right) ^{\alpha_2},
	\end{split}
\end{equation*}
where $\alpha_{1}:=\theta_2(\frac{1}{2})^{\frac{|x_0-x_1|+\frac{a}{2}+\frac{b}{2}}{r_0}}$ and $\alpha_2:=min\left\lbrace \theta_2,(\frac{1}{2})^{5C}\right\rbrace ^{1+\frac{|x_0-x_1|+\frac{a}{2}+\frac{b}{2}}{1\wedge\frac{b}{2}}}$.
\\
\\ \textbf{Step 2.} We show that (\ref{L7.1}) holds for $\lambda>0$.\\
Define $h(x)=\lambda^{1/2}f(\lambda x),$ $ \lambda>0,$ $ x \in \mathbb{R}^{+}$. It is clear that
\begin{equation*}
h\in L^{2}(\mathbb{R}^{+}) ~~and~~ F_\nu(h)(x)=\lambda^{-\frac{1}{2}}F_\nu(f)(x/\lambda).
\end{equation*}
Since $F_\nu(f)\in C_0^{\infty}(\mathbb{R}^{+})$, the above implies $F_\nu(h)(x)\in C_0^{\infty}(\mathbb{R}^{+})$. Thus, we have
\begin{equation}\label{L7.6}
\begin{split}
\int_{A/\lambda}|h(x)|^{2}dx&\leq C\left( \left( \frac{a_{2}}{\lambda}\right) ^{2(\nu+1)}-\left( \frac{a_{1}}{\lambda}\right) ^{2(\nu+1)}\right) \left( \left( \frac{b}{\lambda}\right) ^{-2(\nu+1)}+1 \right)\\
&~~~~~\times \left( \int_{B/\lambda}|h(x)|^{2}dx\right) ^{\theta^{p_1}}\left( \int_0^{\infty}|F_\nu(h)(y)|^{2}e^ydy\right) ^{1-\theta^{p_1}},
\end{split}
\end{equation}
where 
\begin{equation*}
p_1=1+\frac{|\frac{x_0}{\lambda}-\frac{x_1}{\lambda}|+\frac{a}{2\lambda}+\frac{b}{2\lambda}}{1\wedge\frac{b}{2\lambda}}=1+\frac{|x_0-x_1|+\frac{a}{2}+\frac{b}{2}}{\lambda\wedge \frac{b}{2}}.
\end{equation*}
From (\ref{L7.6}), we find that\\
\\$\int_A|f(x)|^{2}dx=\int_{A/\lambda}|h(x)|^{2}dx$
\begin{equation*}
\leq C(\nu)(a_{2}^{2(\nu+1)}-a_{1}^{2(\nu+1)})(\lambda^{-2(\nu+1)}+b^{-2(\nu+1)})
 \left( \int_B|f(x)|^{2}dx\right) ^{\theta^{p_1}}\left( \int_{\mathbb{R}^{+}}|F_\nu(f)(y)|^{2}e^{\lambda y}dy\right) ^{1-\theta^{p_1}}.
\end{equation*}
Thus, we complete the proof of the lemma.
\end{proof}
By the above lemma, we omit the proofs and give the following two corollaries.
\begin{corollary}\label{C1}
There exist constants $C=C(\nu)>0$ and $\theta=\theta(\nu)\in(0,1)$ such that for any $b,$ $\lambda>0$, 
\begin{equation}\label{C1.1}
\begin{split}
\int_0^{\infty}|f(x)|^{2}dx\leq C\left( 1+\frac{b^{2\nu+2}}{\lambda^{2\nu+2}}\right) \left( \int_{[0,b]^{c}}|f(x)|^{2}dx\right) ^{\theta^{1+\frac{b}{\lambda}}}\left( \int_{\mathbb{R}^{+}}|F_\nu(f)(y)|^{2}e^{\lambda y}dy\right) ^{1-\theta^{1+\frac{b}{\lambda}}}
\end{split}
\end{equation}
for each $f\in L^{2}(\mathbb{R}^{+})$ with $F_\nu(f)\in C_0^{\infty}(\mathbb{R}^{+})$.
\end{corollary}
\begin{corollary}\label{C2}
There exists a positive constant $C=C(\nu)$ such that for each $b>0$ and $N\geq0$ and all $f\in L^{2}(\mathbb{R}^{+})$ with supp $F_\nu f\subset [0,N]$.
\begin{equation}\label{C2.1}
\int_0^{\infty}|f(x)|^{2}dx\leq e^{C(1+bN)} \int_{[0,b]^{c}}|f(x)|^{2}dx
\end{equation}

\end{corollary}

With the above lemmas and corollaries in hand, we now can show Theorem \ref{T3} and \ref{T4}.
\begin{proof}
[\textbf{Proof of Theorem \ref{T3}.}]
For fixed  $\lambda,$ $b,$ $T>0$ and $u_0\in C_0^{\infty}(\mathbb{R}^{+})$.
Define 
\begin{equation}\label{T3.3}
	\begin{split}
		f(x):=e^{-\frac{ix^{2}}{4T}}u(x,T;u_0), ~~~~~x\in \mathbb{R}^{+}.
	\end{split}
\end{equation}
From (\ref{L3.1}), we know that
\begin{equation}\label{T3.0}
(2T)^{\frac{1}{2}}e^{\frac{i(\nu+1)\pi}{2}}f(x)=F_\nu(e^{\frac{iy^{2}}{4T}}u_0(y))(x/2T),~~~~~x\in \mathbb{R}^{+}.
\end{equation} 
This implies that for a.e. $x\in \mathbb{R}^{+}$,
\begin{equation}\label{T3.4}
\begin{aligned}
F_\nu(f)(x)&=\int_0^{\infty}\sqrt{xy}J_\nu(xy)f(y)dy\\
&=(2T)^{-\frac{1}{2}}e^{-\frac{i(\nu+1)\pi}{2}}\int_0^{\infty}\sqrt{xy}J_\nu(xy)(2T)^{\frac{1}{2}}e^{\frac{i(\nu+1)\pi}{2}}f(y)dy\\
&=(2T)^{\frac{1}{2}}e^{-\frac{i(\nu+1)\pi}{2}}\int_0^{\infty}\sqrt{x2Tz}J_\nu(x2Tz)(2T)^{\frac{1}{2}}e^{\frac{i(\nu+1)\pi}{2}}f(2Tz)dz\\
&=(2T)^{\frac{1}{2}}e^{-\frac{i(\nu+1)\pi}{2}}\int_0^{\infty}\sqrt{x2Tz}J_\nu(x2Tz)F_\nu(e^{\frac{iy^{2}}{4T}}u_0(y))(z)dz\\
&=(2T)^{\frac{1}{2}}e^{-\frac{i(\nu+1)\pi}{2}}e^{\frac{iy^{2}}{4T}}u_0(y)|_{y=2Tx}\\
&=(2T)^{\frac{1}{2}}e^{-\frac{i(\nu+1)\pi}{2}}e^{iTx^{2}}u_0(2Tx).
\end{aligned}
\end{equation}

We are going to prove the conclusions (i) and (ii) of the theorem one by one. The proof of (iii) is given in section 5. (See the proof of Theorem \ref{T7}(ii).)\par 
(i) By Corollary \ref{C1} with $\lambda$ replaced by $2T\lambda$, we obtain for some absolute constants $C=C(\nu)>0$ and $\theta=\theta(\nu)\in(0,1)$,\\
\\$\int_0^{\infty}|u(x,T;u_0)|^{2}dx
=\int_0^{\infty}|f(x)|^{2}dx$\\
\begin{equation*}
\begin{aligned}
&\leq C\left( 1+\frac{b^{2\nu+2}}{(2\lambda T)^{2\nu+2}}\right) \left( \int_{[0,b]^{c}}|f(x)|^{2}dx\right) ^{\theta^{1+b/(2\lambda T)}}\left( \int_{\mathbb{R}^{+}}e^{2T\lambda x}|F_\nu(f)(x)|^{2}dx\right) ^{1-\theta^{1+b/(2\lambda T)}}\\
&=C\left( 1+\frac{b^{2\nu+2}}{(2\lambda T)^{2\nu+2}}\right) \left( \frac{\int_{[0,b]^{c}}|f(x)|^{2}dx}{\int_{\mathbb{R}^{+}}e^{2T\lambda x}|F_\nu(f)(x)|^{2}dx}\right) ^{\theta^{1+b/(2\lambda T)}}\left( \int_{\mathbb{R}^{+}}e^{2T\lambda x}|F_\nu(f)(x)|^{2}dx\right) \\
&\leq C\left( 1+\frac{b^{2\nu+2}}{(\lambda T)^{2\nu+2}}\right) \left( \int_{[0,b]^{c}}|f(x)|^{2}dx\right) ^{\theta^{1+b/(\lambda T)}}
\left( \int_{\mathbb{R}^{+}}e^{2T\lambda x}|F_\nu(f)(x)|^{2}dx\right) ^{1-\theta^{1+b/(\lambda T)}}.
\end{aligned}
\end{equation*}
 By (\ref{T3.3}), (\ref{T3.4}), and after some computations, we find that\\
\\$\int_0^{\infty}|u(x,T;u_0)|^{2}dx$
\begin{equation*}
\begin{aligned}
\leq C\left( 1+\frac{b^{2\nu+2}}{(\lambda T)^{2\nu+2}}\right) \left( \int_{[0,b]^{c}}|u(x,T;u_0)|^{2}dx\right) ^{\theta^{1+b/(\lambda T)}} \left( \int_{\mathbb{R}^{+}}e^{\lambda x}|u_0(x)|^{2}dx\right) ^{1-\theta^{1+b/(\lambda T)}}.
\end{aligned}
\end{equation*}
The above inequality, along with the conservation law for the Schr\"{o}dinger equation, leads to (\ref{T3.1}). Hence (i) is true.
\\

(ii) Fix $\beta>1$ and $\gamma\in(0,1)$. For above $f(x)$, we claim there exists $C=C(\nu)$ such that
\begin{equation}\label{T3.5}
\int_0^{\infty}|f(x)|^{2}dx\leq Ce^{\left( \frac{C^{\beta}b^{\beta}}{\lambda(1-\gamma)T^{\beta}}\right) ^{1/(\beta-1)}}\left( \int_{[0,b]^{c}}|f(x)|^{2}dx\right) ^{\gamma}\left( \int_{\mathbb{R}^{+}}e^{\lambda (2Tx)^{\beta}}|F_\nu(f)|^{2}dx\right) ^{1-\gamma}.
\end{equation}
In fact, for any fixed $N\geq 0$, we make the following decomposition: $f=g_1+g_2$ in $L^{2}(\mathbb{R}^{+})$ where
\begin{equation*}
F_\nu(g_1):=\chi_{[0,N]}F_\nu(f), ~~F_\nu(g_2):=\chi_{[0,N]^{c}}F_\nu(f).
\end{equation*}
On the one hand, by applying Corollary \ref{C2} to $g_1$, we find that
\begin{equation}\label{T3.6}
	\begin{aligned}
\int_0^{\infty}|f(x)|^{2}dx
&\leq2\int_0^{\infty}|g_1(x)|^{2}dx+2\int_0^{\infty}|g_2(x)|^{2}dx\\
&\leq 2e^{C(1+bN)}\int_{[0,b]^{c}}|g_1(x)|^{2}dx+2\int_0^{\infty}|g_2(x)|^{2}dx\\
&\leq 4e^{C(1+bN)}\int_{[0,b]^{c}}(|f(x)|^{2}+|g_2(x)|^{2})dx+2\int_0^{\infty}|g_2(x)|^{2}dx\\
&\leq 4e^{C(1+bN)}\int_{[0,b]^{c}}|f(x)|^{2}dx+6e^{C(1+bN)}\int_0^{\infty}|g_2(x)|^{2}dx,   
\end{aligned}
\end{equation}	
where $C>0$ depending only on $\nu$. On the other hand, since the Hankel transform $F_\nu$ is an isometry, we get
\begin{equation*}
\begin{split}
\int_0^{\infty}|g_2(x)|^{2}dx&=\int_0^{\infty}|F_\nu(g_2)|^{2}dx=\int_0^{\infty}|\chi_{[0,N]^{c}}F_\nu(f)|^{2}dx\\
&=e^{-\lambda(2TN)^{\beta}}\int_0^{\infty}|\chi_{[0,N]^{c}}F_\nu(f)|^{2}e^{\lambda(2TN)^{\beta}}dx.
\end{split}
\end{equation*}
This, together with (\ref{T3.6}), yields\\
\begin{equation}\label{T3.7}
\int_0^{\infty}|f(x)|^{2}dx\leq 4e^{C(1+bN)}\int_{[0,b]^{c}}|f(x)|^{2}dx+6e^{C(1+bN)-\lambda(2TN)^{\beta}}\int_0^{\infty}|F_\nu(f)(x)|^{2}e^{\lambda(2Tx)^{\beta}}dx.
\end{equation}
Since it follows from the Young inequality that
\begin{equation*}
	\begin{aligned}
CbN&=[Cb((1-\gamma)\lambda(2T)^{\beta})^{-1/\beta}][((1-\gamma)\lambda(2T)^{\beta})^{1/\beta}N]\\
&\leq(1-\frac{1}{\beta})[Cb((1-\gamma)\lambda(2T)^{\beta})^{-1/\beta}]^{\frac{\beta}{\beta-1}}+\frac{1}{\beta}[((1-\gamma)\lambda(2T)^{\beta})^{1/\beta}N]^{\beta}\\
&\leq[(Cb)^{\beta}/((1-\gamma)\lambda(2T)^{\beta})]^{\frac{1}{\beta-1}}+(1-\gamma)\lambda(2TN)^{\beta}.
\end{aligned}
\end{equation*}	
We deduce from (\ref{T3.7}) that\\
\\$\int_0^{\infty}|f(x)|^{2}dx\leq 6e^{C+(\frac{(Cb)^{\beta}}{(1-\gamma)\lambda(2T)^{\beta}})^{\frac{1}{\beta-1}}}$
\begin{equation*}
~~~~~~~~~\times\left(e^{(1-\gamma)\lambda(2TN)^{\beta}}\int_{[0,b]^{c}}|f(x)|^{2}dx+e^{-\gamma\lambda(2TN)^{\beta}}\int_0^{\infty}|F_\nu(f)(x)|^{2}e^{\lambda(2Tx)|^{\beta}}dx\right).
\end{equation*}
As $N$ was arbitrary from $[0,\infty)$, the above indicates that for all $\varepsilon\in(0,1)$,
\begin{equation*}
\int_0^{\infty}|f(x)|^{2}dx\leq 6e^{C+\left( \frac{(Cb)^{\beta}}{(1-\gamma)\lambda(2T)^{\beta}}\right) ^{\frac{1}{\beta-1}}}\left(\varepsilon^{-(1-\gamma)}\int_{[0,b]^{c}}|f(x)|^{2}dx+\varepsilon^{\gamma}\int_0^{\infty}|F_\nu(f)(x)|^{2}e^{\lambda(2Tx)^{\beta}}dx\right).
\end{equation*}
It is easy to check that the above inequality holds in fact for all $\varepsilon>0$. Minimizing it with respect to $\varepsilon>0$ leads to (\ref{T3.5}). Here, we use the inequality
\begin{equation}
\inf\limits_{\varepsilon>0}(\varepsilon^{-(1-\gamma)}A+\varepsilon^{\gamma}B)\leq 2A^{\gamma}B^{1-\gamma} ~~~~~~~ for~ all~ A,B\geq 0. 
\end{equation}
This proves (\ref{T3.5}).
\\Finally from (\ref{T3.3}), (\ref{T3.4}) and (\ref{T3.5}), after some computations, we know that (ii) is true.
\end{proof}

\begin{proof}
[\textbf{Proof of Theorem \ref{T4}.}]
By (\ref{T3.3}), (\ref{T3.4}) and Lemma \ref{L7} (with $\lambda$ replaced by $2\lambda T$), we find that\\
\\$\int_{A}|u(x,T;u_0)|^{2}dx=\int_{A}|f(x)|^{2}dx$
\begin{equation*}
\begin{aligned}
&\leq C\mu_\nu(A)((2\lambda T)^{-2(\nu+1)}+b^{-2(\nu+1)})\left( \int_B|f(x)|^{2}dx\right) ^{\theta^{\alpha_1}}\left( \int_{\mathbb{R}^{+}}|F_\nu(f)(y)|^{2}e^{2\lambda T y}dy\right) ^{1-\theta^{\alpha_1}}\\
&\leq C\mu_\nu(A)((\lambda T)^{-2(\nu+1)}+b^{-2(\nu+1)})\left( \int_{B}|u(x,T;u_0)|^{2}dx\right) ^{\theta^{\alpha_1}}\left( \int_{\mathbb{R}^{+}}e^{\lambda x}|u_0(x)|^{2}dx\right) ^{1-\theta^{\alpha_1}}\\
&\leq C\mu_\nu(A)((\lambda T)^{-1}+b^{-1})^{2\nu+2}\int_{\mathbb{R}^{+}}e^{\lambda x}|u_0(x)|^{2}dx\left( \frac{\int_{B}|u(x,T;u_0)|^{2}dx}{\int_{\mathbb{R}^{+}}e^{\lambda x}|u_0(x)|^{2}dx}\right) ^{\theta^{\alpha_1}}\\
&\leq 2C\mu_\nu(A)((\lambda T)\wedge b)^{-(2\nu+2)}\int_{\mathbb{R}^{+}}e^{\lambda x}|u_0(x)|^{2}dx\left( \frac{\int_{B}|u(x,T;u_0)|^{2}dx}{\int_{\mathbb{R}^{+}}e^{\lambda x}|u_0(x)|^{2}dx}\right) ^{\theta^{\alpha_2}}
\end{aligned}
\end{equation*}
for some absolute constant $C=C(\nu)>0$ and $\theta=\theta(\nu) \in(0,1)$, where
\begin{equation*}
\alpha_1=1+\frac{|x_0-x_1|+\frac{a}{2}+\frac{b}{2}}{(2\lambda T)\wedge \frac{b}{2}}, ~~~ \alpha_2=1+\frac{|x_0-x_1|+\frac{a}{2}+\frac{b}{2}}{(\lambda T)\wedge \frac{b}{2}}.
\end{equation*} 
%Since
%\begin{equation*}
%(\lambda T)^{-1}+b^{-1}\leq 2((\lambda T)\wedge b)^{-1}, ~~~~ (\lambda T)\wedge b\leq (2\lambda T)\wedge b,~~~ \theta\in(0,1).
%\end{equation*}
%Then we deduce that\\
%\\$\int_{A}|u(x,T;u_0)|^{2}dx$
%\begin{equation*}
%\begin{aligned}
%&\leq C\mu_\nu(A)((\lambda T)^{-1}+b^{-1})^{2\nu+2}\int_{\mathbb{R}^{+}}e^{\lambda x}|u_0(x)|^{2}dx\left( \frac{\int_{B}|u(x,T;u_0)|^{2}dx}{\int_{\mathbb{R}^{+}}e^{\lambda x}|u_0(x)|^{2}dx}\right) ^{\theta^{\alpha_1}}\\
%&\leq C\mu_\nu(A)((\lambda T)\wedge b)^{-(2\nu+2)}\int_{\mathbb{R}^{+}}e^{\lambda x}|u_0(x)|^{2}dx\left( \frac{\int_{B}|u(x,T;u_0)|^{2}dx}{\int_{\mathbb{R}^{+}}e^{\lambda x}|u_0(x)|^{2}dx}\right) ^{\theta^{\alpha_2}}
%\end{aligned}
%\end{equation*}
%with
%\begin{equation*}
%\alpha_1=1+\frac{|x_0-x_1|+\frac{a}{2}+\frac{b}{2}}{(2\lambda T)\wedge \frac{b}{2}}, ~~~ \alpha_2=1+\frac{|x_0-x_1|+\frac{a}{2}+\frac{b}{2}}{(\lambda T)\wedge \frac{b}{2}}.
%\end{equation*} 
This proves our theorem.
\end{proof}

\section{Proofs of Theorems \ref{T5}-\ref{T9}}
Theorem \ref{T5} is mainly based on Theorem \ref{T3}. Theorem \ref{T8} is a consequence of Theorem \ref{T4}. Theorem \ref{T9} is based on Theorem \ref{T4}, as well as the other property for the Schr\"{o}dinger equation (presented in Lemma \ref{L.S.2} of this paper). 
\begin{proof}
[\textbf{Proof of Theorem \ref{T5}.}]
Given $u_{0}\in L^{2}(\mathbb{R}^{+})$ with supp $u_{0}\subset [0,N]$. By a standard density argument, we can apply Theorem \ref{T3}(i) (with $\lambda=b/T$) to get that for some $C=C(\nu)>0$ and $\theta=\theta(\nu)\in(0,1)$,
\begin{equation}\label{T5.2}
\begin{split}
\int_{\mathbb{R}^{+}}|u_0(x)|^{2}dx
&\leq 2C\left( \int_{[0,b]^{c}}|u(x,T;u_0)|^{2}dx\right) ^{\theta^{2}}\left( \int_{\mathbb{R}^{+}}e^{bx/T}|u_0(x)|^{2}dx\right) ^{1-\theta^{2}}\\
&\leq 2Ce^{bN(1-{\theta}^{2})/T}\left( \int_{[0,b]^{c}}|u(x,T;u_0)|^{2}dx\right) ^{\theta^{2}}\left( \int_{\mathbb{R}^{+}}|u_0(x)|^{2}dx\right) ^{1-\theta^{2}}.
\end{split}
\end{equation}
%Meanwhile, since supp $u_{0}\subset [0,N]$, we have
%\begin{equation*}
%\int_{\mathbb{R}^{+}}e^{bx/T}|u_0(x)|^{2}dx\leq e^{bN/T} \int_{\mathbb{R}^{+}}|u_0(x)|^{2}dx.
%\end{equation*}
This implies that
\begin{equation*}
\int_{\mathbb{R}^{+}}|u_0(x)|^{2}dx\leq (2C)^{\frac{1}{{\theta}^{2}}}e^{bN(1-{\theta}^{2})/T{\theta}^{2}}\int_{[0,b]^{c}}|u(x,T;u_0)|^{2}dx.
\end{equation*}
Hence, we end the proof of the theorem.
\end{proof}
We recall Lemma 3.1 in \cite{WWZ} that will be used in the proofs of Theorems \ref{T8} and \ref{T9}.
\begin{lemma}\label{L.R}
Let $x$, $\theta\in(0,1)$ .\\
(i) For each $a>0$,
\begin{equation}\label{L.R.1}
\sum_{k=1}^{\infty}x^{\theta^{k}}e^{-ak}\leq\frac{e^{a}}{|ln\theta|}\varGamma\left(\frac{a}{|ln\theta|} \right)|lnx|^{-a/|ln\theta|}.
\end{equation}
(ii) For each $\varepsilon>0$ and $\alpha>0$,
\begin{equation}\label{L.R.2}
\sum_{k=1}^{\infty}x^{\theta^{k}}k^{-1-\varepsilon}\leq\frac{4}{\varepsilon}\alpha^{\varepsilon}e^{\varepsilon ln\varepsilon+\varepsilon +e\alpha^{-1}\theta^{-1}}(ln(\alpha|lnx|+e))^{-\varepsilon}.
\end{equation}
\end{lemma}
 \begin{proof}
[\textbf{Proof of Theorem \ref{T8}.}]
 %Let $x_{0}\in {\mathbb{R}^{+}}$, $b$, $\lambda_{1}$, $\lambda_{ 2}$, $T>0$.
 When $u_{0}=0$, (\ref{T8.1}) holds clearly for all $\varepsilon \in (0,1)$. We now fix $u_{0}\in C_{0}^{\infty}(\mathbb{R}^{+})\backslash\{0\}$. For convenience, we define 
 \begin{equation*}
 A_{1}:=\int_{\mathbb{R}^{+}}e^{\lambda_{1}x}|u_{0}(x)|^{2}dx,~~~B_{1}:=\int_{B}|u(x,T;u_0)|^{2}dx,
 \end{equation*}
 $~~~~~~~~~~~~~~~~~~~~~~~~~~~$ $R_{\lambda_{2}}:=\int_{\mathbb{R}^{+}}e^{-\lambda_{2}x}|u(x,T;u_0)|^{2}dx.$\\
 \\The proof of (\ref{T8.1}) is divided into several steps.\\
 \\ \textbf{Step 1.} There exist positive constants $C_{1}(\nu)$ and $C_{2}(\nu)$ such that
 \begin{equation}\label{T8.2}
 R_{\lambda_{2}}\leq C_{3}(x_{0},b,\lambda_{1},\lambda_{2},T)\left( ln\frac{A_{1}}{B_{1}}\right) ^{-C_{2}\lambda_{2}((\lambda_{1}T)\wedge \frac{b}{2})}A_{1},
 \end{equation}
 where 
 \begin{equation}\label{T8.3}
 	C_{3}(x_{0},b,\lambda_{1},\lambda_{2},T):=1+C_{1}\varGamma\left( C_{2}\lambda_{2}\left( (\lambda_{1}T)\wedge \frac{b}{2}\right) \right) \times exp\left( \lambda_{2}^{-1}\left( (\lambda_{1}T)\wedge \frac{b}{2}\right) ^{-1}+\lambda_{2}\left( x_{0}+\frac{b}{2}\right) \right) .
 \end{equation}
 %\begin{equation}\label{T8.4}
 %g(\eta):=(ln\eta)^{-C_{2}\lambda_{2}((\lambda_{1}T)\wedge \frac{b}{2})},~~\eta>1.
 %\end{equation}
 According to Theorem \ref{T4} (with ($x_{0}, x_{1}, \frac{a}{2}, \frac{b}{2}$) replaced by ($2(k-1)\lambda_{2}^{-1}+\lambda_{2}^{-1}, x_{0}, \lambda_{2}^{-1}, \frac{b}{2}$) with $k\in {\mathbb{N}^{+}}$), we notice that \\
\\$\int_{\mathbb{R}^{+}}e^{-\lambda_{2}x}|u(x,T;u_0)|^{2}dx\leq
\sum_{k=1}^{\infty}\int_{2(k-1)\lambda_{2}^{-1}<x<2k\lambda_{2}^{-1}}e^{-2(k-1)}|u(x,T;u_0)|^{2}dx\\$
\begin{equation}\label{T8.5}
\begin{split}
&\leq
C((\lambda_{1} T)\wedge b)^{-(2\nu+2)}\left( \sum_{k=1}^{\infty}e^{-2k+2}\left( (2k\lambda_{2}^{-1})^{2\nu+2}-(2(k-1)\lambda_{2}^{-1})^{2\nu+2}\right) \left( \frac{B_{1}}{A_{1}}\right) 
^{\theta^{1+\frac{|x_{0}-(2k-1)\lambda_{2}^{-1}|+\lambda_{2}^{-1}+\frac{b}{2}}{(\lambda_{1}T)\wedge \frac{b}{2}}}}\right)A_{1}\\
&\leq
C((\lambda_{1} T)\wedge b)^{-(2\nu+2)}\left( \sum_{k=1}^{\infty}e^{-2k+2}(2k\lambda_{2}^{-1})^{2\nu+2}\left( \frac{B_{1}}{A_{1}}\right) 
^{\theta^{1+\frac{|x_{0}-(2k-1)\lambda_{2}^{-1}|+\lambda_{2}^{-1}+\frac{b}{2}}{(\lambda_{1}T)\wedge \frac{b}{2}}}}\right)A_{1}\\
&\leq C(2\lambda_{2}^{-1})^{2\nu+2}(2\nu+2)^{2\nu+2}((\lambda_{1} T)\wedge b)^{-(2\nu+2)}e^{2}\left( \sum_{k=1}^{\infty}e^{-k}\left( \frac{B_{1}}{A_{1}}\right) 
^{\theta^{1+\frac{|x_{0}-(2k-1)\lambda_{2}^{-1}|+\lambda_{2}^{-1}+\frac{b}{2}}{(\lambda_{1}T)\wedge \frac{b}{2}}}}\right) A_{1}\\
&\leq C(2\lambda_{2}^{-1})^{2\nu+2}(2\nu+2)^{2\nu+2}((\lambda_{1} T)\wedge b)^{-(2\nu+2)}e^{2}\left( \sum_{k=1}^{\infty}e^{-k}\left( \frac{B_{1}}{A_{1}}\right) 
^{\theta^{1+\frac{x_{0}+2k\lambda_{2}^{-1}+\frac{b}{2}}{(\lambda_{1}T)\wedge \frac{b}{2}}}}\right) A_{1}
\end{split}
\end{equation}
for some $\theta \in (0,1)$ and $C>0$ depending only on $\nu$. In the fourth inequality, we used  the fact that $k\leq ne^{k/n}$ for all $k\in{\mathbb{N}^{+}}$.
The last inequality is due to $B_{1}<A_{1}$ (which follows from the definitions of $A_{1}$ and $B_{1}$, the conservation law for the Schr\"{o}dinger equation and the fact that $u_{0}\neq0$).\\
Now, we apply (\ref{L.R.1}) in Lemma \ref{L.R} with
\begin{equation*}
(a,x,\theta)=(1,({B_{1}}/{A_{1}}) 
^{\theta^{1+\frac{x_{0}+\frac{b}{2}}{(\lambda_{1}T)\wedge \frac{b}{2}}}},\theta^{\frac{2}{\lambda_{2}((\lambda_{1}T)\wedge \frac{b}{2})}})
\end{equation*}
to get 
$$\sum_{k=1}^{\infty}e^{-k}\left( \frac{B_{1}}{A_{1}}\right) 
^{\theta^{1+\frac{x_{0}+2k\lambda_{2}^{-1}+\frac{b}{2}}{(\lambda_{1}T)\wedge \frac{b}{2}}}}
\leq \frac{e\lambda_{2}((\lambda_{1}T)\wedge \frac{b}{2})}{2|ln\theta|}\varGamma\left( \frac{\lambda_{2}((\lambda_{1}T)\wedge \frac{b}{2})}{2|ln\theta|}\right) \left(  \theta^{1+\frac{x_{0}+\frac{b}{2}}{(\lambda_{1}T)\wedge \frac{b}{2}}}\left| ln\frac{B_{1}}{A_{1}}\right| \right)  ^\frac{-\lambda_{2}((\lambda_{1}T)\wedge \frac{b}{2})}{2|ln\theta|}.$$
This, along with (\ref{T8.5}) and the facts that $x^{a}\leq ([a]+1)!e^{x}$ for all $x>0$, $a>0$ where $[a]$ is the integral part of $a$,
%and that $(\lambda_{1}T)\wedge \frac{b}{2}\leq \frac{b}{2}$,
 imply that\\
\\$\int_{\mathbb{R}^{+}}e^{-\lambda_{2}x}|u(x,T;u_0)|^{2}dx$
\begin{equation*}
\begin{aligned}
&\leq
C(2\lambda_{2}^{-1})^{2\nu+2}\left( \frac{2\nu+2}{(\lambda_{1} T)\wedge b}\right)^{2\nu+2}\\ 
&~~~~\times e^{3}\frac{\lambda_{2}((\lambda_{1}T)\wedge \frac{b}{2})}{|ln\theta|}\varGamma\left( \frac{\lambda_{2}\left( (\lambda_{1}T)\wedge \frac{b}{2}\right) }{2|ln\theta|}\right) \left( \theta^{1+\frac{x_{0}+\frac{b}{2}}{(\lambda_{1}T)\wedge \frac{b}{2}}}\left| ln\frac{B_{1}}{A_{1}}\right|\right) ^\frac{-\lambda_{2}((\lambda_{1}T)\wedge \frac{b}{2})}{2|ln\theta|}A_{1}\\ &=\frac{C2^{2\nu+2}e^{3}}{|ln\theta|}(2\nu+2)^{2\nu+2}(\lambda_{2}(\lambda_{1} T)\wedge \frac{b}{2})^{-(2\nu+2)+1}\\
&~~~~\times e^{\frac{\lambda_{2}}{2}((\lambda_{1}T)\wedge \frac{b}{2}+x_{0}+\frac{b}{2})} \varGamma\left( \frac{\lambda_{2}((\lambda_{1}T)\wedge \frac{b}{2})}{2|ln\theta|}\right) \left(  ln\frac{A_{1}}{B_{1}} \right) ^\frac{-\lambda_{2}((\lambda_{1}T)\wedge \frac{b}{2})}{2|ln\theta|}A_{1}\\
&\leq\frac{C2^{2\nu+2}e^{3}}{|ln\theta|}(2\nu+2)^{2\nu+2}([2\nu+1]+1)!e^{\lambda_{2}^{-1}((\lambda_{1}T)\wedge \frac{b}{2})^{-1}+\lambda_{2} (x_{0}+\frac{b}{2})}
 \varGamma\left( \frac{\lambda_{2}((\lambda_{1}T)\wedge \frac{b}{2})}{2|ln\theta|}\right) \left( ln\frac{A_{1}}{B_{1}}\right) ^\frac{-\lambda_{2}((\lambda_{1}T)\wedge \frac{b}{2})}{2|ln\theta|}A_{1}.
\end{aligned}
\end{equation*}
This leads to (\ref{T8.2}).\\
 \\ \textbf{Step 2.} 
(\ref{T8.1}) holds if $\lambda_{2}\leq \frac{1}{C_{2}((\lambda_{1}T)\wedge \frac{b}{2})}$.\\
\\First, we claim that for each $\varepsilon \in (0,1)$,
\begin{equation}\label{T8.7}
 R_{\lambda_{2}}\leq C_{3}(\varepsilon A_{1}+\varepsilon e^{\varepsilon^{-\frac{1}{C_{2}\lambda_{2}((\lambda_{1}T)\wedge \frac{b}{2})}}}B_{1}),
\end{equation}
where $C_{3}=C_{3}(x_{0},b,\lambda_{1},\lambda_{2},T)$ is given by (\ref{T8.3}). Indeed, if $R_{\lambda_{2}}\leq C_{3}\varepsilon A_{1}$, (\ref{T8.7}) is obvious. So we only consider the case: $R_{\lambda_{2}}> C_{3}\varepsilon A_{1}$.
% for any fixed $\varepsilon >0$, there are only two cases: either $R_{\lambda_{2}}\leq C_{3}\varepsilon A_{1}$ or $R_{\lambda_{2}}> C_{3}\varepsilon A_{1}$. In the first case, (\ref{T8.7}) is obvious. 
In this case, we have the following observation:
\begin{equation}\label{T8.8}
0<\varepsilon<\frac{R_{\lambda_{2}}}{C_{3}A_{1}}<1.
\end{equation}
%In fact, the first and second inequalities in (\ref{T8.8}) are clear. To prove the last inequality in (\ref{T8.8}), two facts deserve to be mentioned: First, we observe from (\ref{T8.3}) that $C_{3}>1$. Second, by the definitions of $A_{1}$ and $R_{\lambda_{2}}$, using the conservation law for the Schr\"{o}dinger equation, we find that
%\begin{equation*}
%\begin{split}
%R_{\lambda_{2}}&=\int_{\mathbb{R}^{+}}e^{-\lambda_{2}x}|u(x,T;u_0)|^{2}dx\leq\int_{\mathbb{R}^{+}}|u(x,T;u_0)|^{2}dx=\int_{\mathbb{R}^{+}}|u_0(x)|^{2}dx\\
%&\leq\int_{\mathbb{R}^{+}}e^{\lambda_{1}x}|u_(x))|^{2}dx=A_{1}.
%\end{split}
%\end{equation*}
%These two facts lead to the last inequality in (\ref{T8.8}) at once. \\
Besides, two facts are given in order. First, since $\lambda_{2}\leq \frac{1}{C_{2}((\lambda_{1}T)\wedge \frac{b}{2})}$, the function $x\rightarrow xe^{x^{-\frac{1}{C_{2}\lambda_{2}((\lambda_{1}T)\wedge \frac{b}{2})}}}$ is decreasing on $(0,1)$. This, along with (\ref{T8.8}), indicates that
\begin{equation}\label{T8.9}
\frac{R_{\lambda_{2}}}{C_{3}A_{1}}e^{\left( \frac{R_{\lambda_{2}}}{C_{3}A_{1}}\right) ^{-\frac{1}{C_{2}\lambda_{2}((\lambda_{1}T)\wedge \frac{b}{2})}}}\leq \varepsilon e^{\varepsilon^{-\frac{1}{C_{2}\lambda_{2}((\lambda_{1}T)\wedge \frac{b}{2})}}}.
\end{equation}
Second, since the function $f(x)=e^{x^{-\frac{1}{C_{2}\lambda_{2}((\lambda_{1}T)\wedge \frac{b}{2})}}}$ is decreasing on $(0,\infty)$ and its inverse is the function $g(x):=(lnx)^{-C_{2}\lambda_{2}((\lambda_{1}T)\wedge \frac{b}{2})}$, we deduce from (\ref{T8.2}) that
\begin{equation}\label{T8.10}
\frac{A_{1}}{B_{1}}=f\left( g\left( \frac{A_{1}}{B_{1}}\right) \right) \leq f\left( \frac{R_{\lambda_{2}}}{C_{3}A_{1}}\right) =e^{\left( \frac{R_{\lambda_{2}}}{C_{3}A_{1}}\right) ^{-\frac{1}{C_{2}\lambda_{2}((\lambda_{1}T)\wedge \frac{b}{2})}}}.
\end{equation}
According to (\ref{T8.9}) and (\ref{T8.10}), we have
\begin{equation*}
R_{\lambda_{2}}=C_{3}\frac{R_{\lambda_{2}}}{C_{3}A_{1}}\frac{A_{1}}{B_{1}}B_{1}\leq C_{3}\left(  \frac{R_{\lambda_{2}}}{C_{3}A_{1}}e^{\left( \frac{R_{\lambda_{2}}}{C_{3}A_{1}}\right) ^{-\frac{1}{C_{2}\lambda_{2}((\lambda_{1}T)\wedge \frac{b}{2})}}}\right)  B_{1}\leq C_{3}\varepsilon e^{\varepsilon^{-\frac{1}{C_{2}\lambda_{2}((\lambda_{1}T)\wedge \frac{b}{2})}}}B_{1}.
\end{equation*}
Since $\varepsilon \in (0,1)$ was arbitrary, the above leads to (\ref{T8.7}) for $R_{\lambda_{2}}> C_{3}\varepsilon A_{1}$. Hence, (\ref{T8.7}) is true.
Next, we claim that
\begin{equation}\label{T8.11}
C_{3}(x_{0},b,\lambda_{1},\lambda_{2},T)\leq exp\left\lbrace 2(C_{1}+C_{2}^{-1}+1)\left( 1+\frac{\lambda_{2}^{-1}+x_{0}+\frac{b}{2}}{(\lambda_{1}T)\wedge \frac{b}{2}}\right) \right\rbrace .
\end{equation}
In fact, by some computations, we first observe that for each $s \in (0,1]$,
\begin{equation}\label{T8.12}
\Gamma(s)\leq e^{2s^{-1}}.
\end{equation}
Second, we have $\lambda_{2}\leq \frac{1}{C_{2}((\lambda_{1}T)\wedge \frac{b}{2})}$. So according to (\ref{T8.3}) and (\ref{T8.12}) with $s=C_{2}\lambda_{2}((\lambda_{1}T)\wedge \frac{b}{2})$, we get
\begin{equation*}
\begin{aligned}
C_{3}(x_{0},b,\lambda_{1},\lambda_{2},T)&\leq 1+e^{C_{1}}e^{2C_{2}^{-1}\lambda_{2}^{-1}((\lambda_{1}T)\wedge \frac{b}{2})^{-1}}\times exp(\lambda_{2}^{-1}((\lambda_{1}T)\wedge \frac{b}{2})^{-1}+\lambda_{2}(x_{0}+\frac{b}{2}))\\
&\leq e\cdotp exp\left(  C_{1}+(2C_{2}^{-1}+1)\lambda_{2}^{-1}((\lambda_{1}T)\wedge \frac{b}{2})^{-1}+C_{2}^{-1}\frac{x_{0}+\frac{b}{2}}{(\lambda_{1}T)\wedge \frac{b}{2}}\right)  .
\end{aligned}
\end{equation*}
This leads to (\ref{T8.11}). Now, by (\ref{T8.7}) and (\ref{T8.11}), we obtain
\begin{equation}\label{T8.6}
  R_{\lambda_{2}}\leq C_{4}(x_{0},b,\lambda_{1},\lambda_{2},T)(\varepsilon A_{1}+\varepsilon e^{\varepsilon^{-1-\frac{1}{C_{2}\lambda_{2}((\lambda_{1}T)\wedge \frac{b}{2})}}}B_{1}),
  \end{equation}
  where 
  $$C_{4}(x_{0},b,\lambda_{1},\lambda_{2},T):= C_{1}exp\left\lbrace 2(C_{1}+C_{2}^{-1}+1)(C_{2}+1)\left(  1+\frac{\lambda_{2}^{-1}+x_{0}+\frac{b}{2}}{(\lambda_{1}T)\wedge \frac{b}{2}}\right)  \right\rbrace .$$
  Since
   \begin{equation}\label{T8.13}
   \varepsilon e^{\varepsilon^{-1-\frac{\alpha}{\lambda_{2}((\lambda_{1}T)\wedge \frac{b}{2})}}}\leq e^{\varepsilon^{-1-\frac{\beta}{\lambda_{2}((\lambda_{1}T)\wedge \frac{b}{2})}}}~when~0<\alpha<\beta~and~\varepsilon \in (0,1),
   \end{equation}
this, along with (\ref{T8.6}), leads to (\ref{T8.1}) for $\lambda_{2} \leq \frac{1}{C_{2}((\lambda_{1}T)\wedge \frac{b}{2})}$.  
\\ \textbf{Step 3.} (\ref{T8.1}) holds for $\lambda_{2}> \frac{1}{C_{2}((\lambda_{1}T)\wedge \frac{b}{2})}$.\\
\\First, by the definition of $R_{\lambda_{2}}$, we obtain that $R_{\lambda_{2}}\leq R_{\frac{1}{C_{2}((\lambda_{1}T)\wedge \frac{b}{2})}}$. Then combining (\ref{T8.7}) and (\ref{T8.11}) (with $\lambda_{2}$ replaced by $\frac{1}{C_{2}((\lambda_{1}T)\wedge \frac{b}{2})}$), we notice that for each $\varepsilon \in (0,1)$,
\begin{equation*}
\begin{split}
R_{\lambda_{2}}&\leq exp\left\lbrace 2(C_{1}+C_{2}^{-1}+1)\left(  1+\frac{C_{2}(\lambda_{1}T)\wedge \frac{b}{2}+x_{0}+\frac{b}{2}}{(\lambda_{1}T)\wedge \frac{b}{2}}\right)  \right\rbrace (\varepsilon A_{1}+\varepsilon e^{\varepsilon^{-1}}B_{1})\\
&\leq exp\left\lbrace 2(C_{1}+C_{2}^{-1}+1)\left(  1+C_{2}+\frac{\lambda_{2}^{-1}+x_{0}+\frac{b}{2}}{(\lambda_{1}T)\wedge \frac{b}{2}}\right) \right\rbrace (\varepsilon A_{1}+\varepsilon e^{\varepsilon^{-1-\frac{1}{C_{2}\lambda_{2}((\lambda_{1}T)\wedge \frac{b}{2})}}}B_{1})
\end{split}
\end{equation*}
which, together with (\ref{T8.13}), yields (\ref{T8.1}) for $\lambda_{2}> \frac{1}{C_{2}((\lambda_{1}T)\wedge \frac{b}{2})}$.\\
\\Hence, this ends the proof of Theorem \ref{T8}.
\end{proof}

We first prove a lemma about regularity propagation property for the Schr\"{o}dinger equation, then we use this result to prove Theorem \ref{T9}.
\begin{lemma}\label{L.S.2}
Given $k\in \mathbb{N}^{+}$, there exists a constant $C(k,\nu)$ such that for any $T>0$ and $u_{0}\in C_{0}^{\infty}(\mathbb{R}^{+})$,
\begin{equation}\label{L.S.2.1}
\int_{0}^{\infty}x^{2k}|u(x,T;u_{0})|^{2}dx\leq C(k,\nu)\left( T+\frac{1}{T}\right) ^{2k}\left( \lVert u_{0}\rVert^{2}_{H^{4k}(\mathbb{R}^{+})}+\int_{0}^{\infty}x^{8k}|u_{0}|^{2}dx+\int_{0}^{\infty}\frac{1}{x^{4k}} |u_{0}|^{2}dx\right) .
\end{equation}
\end{lemma}
\begin{proof}
Fix $k\in \mathbb{N}^{+}$, $T>0$ and $u_{0}\in C_{0}^{\infty}(\mathbb{R}^{+})$.
Because of the identity (\ref{equation1.7}) and the unitary property of Hankel transform, we find that
\begin{equation}\label{L.S.2.2}
\begin{split}
 \lVert x^{k}u(x,T;u_{0})\rVert_{L^{2}(\mathbb{R}^{+})}^{2}&= \lVert x^{k}(2T)^{-\frac{1}{2}}F_{\nu}(f)(x/2T)\rVert_{L^{2}(\mathbb{R}^{+})}^{2}\\
 &=\int_{0}^{\infty}x^{2k}(2T)^{-1}|F_{\nu}(f)(x/2T)|^{2}dx\\
 &=(2T)^{2k}\int_{0}^{\infty}x^{2k}|F_{\nu}(f)(x)|^{2}dx\\
 &=(2T)^{2k}\lVert F_{\nu}x^{k}F_{\nu}(f)\rVert_{L^{2}(\mathbb{R}^{+})}^{2}\\
 &=(2T)^{2k}\left\| H_{\nu} ^{\frac{k}{2}}f\right\| _{L^{2}(\mathbb{R}^{+})}^{2},
\end{split}
\end{equation}
where $f:=e^{\frac{ix^{2}}{4T}}u_{0}$.\\
Since the operator $ H_{\nu} ^{k}$ with domain $C_{0}^{\infty}(\mathbb{R}^{+})$ is a polynomial in $\frac{1}{x}$ and $\partial_{x}$, of degree $2k$, and after some computations, we can find that the polynomial $ H_{\nu} ^{k}$ is a linear combination of the monomials
\begin{equation*}
\left\lbrace  \frac{1}{x^{r}}\partial_{x}^{s}: r+s=2k,~r,~s\in\mathbb{N},~r\neq1 \right\rbrace.
\end{equation*}
From this, we see that 
\begin{equation}\label{L.S.2.3}
\begin{split}
\int_{0}^{\infty}\left|  H_{\nu} ^{\frac{k}{2}}f\right| ^{2}dx&=\int_{0}^{\infty}\left\langle  H_{\nu} ^{k}f,f\right\rangle _{\mathbb{C}} dx \\
&\leq C(k,\nu) {\Huge \Sigma_{r+s=2k,r\neq1}}\int_{0}^{\infty}\left| \left\langle \partial_{x}^{s}f, \frac{1}{x^{r}}f\right\rangle _{\mathbb{C}}\right| dx,
\end{split}
\end{equation}
here and throughout the proof, $C(k,\nu)$ stands for a positive constant (depending only on $k$, $\nu$), which may vary in different contexts.
Then by some computations, we find that
\begin{equation}\label{L.S.2.4}
\begin{split}
\|\partial_{x}^{s}f\|_{L^{2}(\mathbb{R}^{+})}^{2}&=\int_{0}^{\infty}\left| \left\langle  \partial_{x}^{2s}f,f\right\rangle  _{\mathbb{C}}\right| dx\leq C_{1}(s)max\{T^{-2s},1\}{\Huge \Sigma_{0\leq m+n\leq 2s}}\int_{0}^{\infty}\left| \left\langle \partial_{x}^{m}u_{0}, x^{n}u_{0}\right\rangle _{\mathbb{C}}\right|dx\\
&\leq C_{2}(s)max\{T^{-2s},1\}\left( \lVert u_{0}\rVert^{2}_{H^{2s}(\mathbb{R}^{+})}+\int_{0}^{\infty}x^{4s}|u_{0}|^{2}dx\right).
\end{split}
\end{equation}
%where $\mathcal{F}$ is Fourier transform and $\tilde{u}(x,T;u_{0})$ is the solution of the free Schr\"{o}dinger equation with the initial condition $\tilde{u}(x,0)=u_{0}$ in $\mathbb{R}$.\\
From (\ref{L.S.2.2}), (\ref{L.S.2.3}) and (\ref{L.S.2.4}), we get that
\begin{align}
\int_{0}^{\infty}x^{2k}|u(x,T;u_{0})|^{2}dx&=(2T)^{2k}\int_{0}^{\infty}\left|  H_{\nu} ^{\frac{k}{2}}f\right| ^{2}dx \notag\\
&\leq C(k,\nu)(2T)^{2k} {\Huge \Sigma_{r+s=2k,r\neq1}}\int_{0}^{\infty}\left| \left\langle \partial_{x}^{s}f, \frac{1}{x^{r}}f\right\rangle _{\mathbb{C}}\right| dx \notag\\
&\leq C(k,\nu)(2T)^{2k}{\Huge \Sigma_{r+s=2k,r\neq1}}\left(\int_{0}^{\infty}|\partial_{x}^{s}f|^{2}dx+\int_{0}^{\infty}\left| \frac{1}{x^{r}}f\right| ^{2}dx \right)\label{L.S.2.5}\\
&\leq C(k,\nu)\left( T+\frac{1}{T}\right) ^{2k}\left( \lVert u_{0}\rVert^{2}_{H^{4k}(\mathbb{R}^{+})}+\int_{0}^{\infty}x^{8k}|u_{0}|^{2}dx+\int_{0}^{\infty}\frac{1}{x^{4k}} |u_{0}|^{2}dx\right).
\end{align}
So we complete the proof.
\end{proof}
\begin{remark}
 Lemma \ref{L.S.2} gives a quantitative property for solutions of (\ref{equation1.1}) for general $\nu\geq0$. For $\nu=\frac{1}{2}$, which can be seen as free Schr\"{o}dinger equation on half line, based on the above proof process (combining \eqref{L.S.2.2} and \eqref{L.S.2.4}), the quantitative estimate takes the following form: Given $k\in \mathbb{N}^{+}$, for $u_{0}\in C_{0}^{\infty}(\mathbb{R}^{+})$, and $u(t,x)$ solves (\ref{equation1.1}) with $\nu=\frac{1}{2}$, then there exists a constant $C(k)$ such that for any $T>0$
\begin{equation*}
\int_{0}^{\infty}x^{2k}|u(x,T;u_{0})|^{2}dx\leq C(k)(1+T)^{2k}\left( \lVert u_{0}\rVert^{2}_{H^{2k}(\mathbb{R}^{+})}+\int_{0}^{\infty}x^{4k}|u_{0}|^{2}dx\right).
\end{equation*}
This estimate is consistent with the quantitative estimate given by Lemma 3.2 in \cite{WWZ} for the free Schr\"{o}dinger equation in $\mathbb{R}^{n}$. While for general $\nu\geq0$, the additional requirements for the regularity and decay of function $u_0$ come from inequality (\ref{L.S.2.5}). It is natural to wonder if the additional requirements for the regularity and decay of $u_0$, even the last term
$\int_{0}^{\infty}\frac{1}{x^{4k}} |u_{0}|^{2}dx$ can be removed. We point out that for general $\nu\geq0$, if $k=1$, it is true, i.e., we have the estimate that
\begin{equation*}
\int_{0}^{\infty}x^{2}|u(x,T;u_{0})|^{2}dx\leq C(\nu)(1+T)^{2}\left( \lVert u_{0}\rVert^{2}_{H^{2}(\mathbb{R}^{+})}+\int_{0}^{\infty}x^{4}|u_{0}|^{2}dx\right).
\end{equation*}
This dues to the fact that the equivalence of the Sobolev norms $\left\|   H_{\nu} ^{\frac{1}{2}}f\right\| _{L^{2}(\mathbb{R}^{+})}^{2} $ and $\left\|   -\Delta ^{\frac{1}{2}}f\right\| _{L^{2}(\mathbb{R}^{+})}^{2}$, but for $k\geq2$, the equivalence of these two norms is still unknown, for this topic, we refer readers to \cite{RKH, RMJZZ, TP, KM} for more knowledge.
\end{remark}
\begin{remark}
 The proof of this Lemma borrows some ideas from the proof of Lemma 3.2 in \cite{WWZ}, but there are some differences between the proof of this Lemma and the proof of Lemma 3.2. Their proof relies on the commutativity of the operators $(x_j+2i(t-T)\partial_{x_{j}})^{k}$ and $i\partial_t+\Delta$, while for the case in presence of a potential term, we don't know how to construct similar commutators to get an estimate similar to \eqref{L.S.2.1}. Thanks to the identity \eqref{equation1.7}, we can deduce the estimate \eqref{L.S.2.1}, and this approach fits to prove Lemma 3.2 in \cite{WWZ}. It may be interesting to know if this 
quantitative estimate still holds for general potential $V$, which may be not an easy question since both of the approach may be fail, and to our best knowledge, we haven't found such quantitative estimates for general potential $V$.
\end{remark}

 \begin{proof}
 [\textbf{Proof of Theorem \ref{T9}.}]
 %Let $x_{0}\in {\mathbb{R}^{+}}$, $b$, $\lambda$, $T>0$. 
 When $u_{0}=0$, (\ref{T9.1}) holds clearly for all $\varepsilon \in (0,1)$. We now fix $u_{0}\in C_{0}^{\infty}(\mathbb{R}^{+})\backslash\{0\}$. For convenience, we define 
  \begin{equation*}
  A_{2}:=\int_{\mathbb{R}^{+}}e^{\lambda x}|u_{0}(x)|^{2}dx+\lVert u_{0}\rVert^{2}_{H^{4([\nu]+3)}(\mathbb{R}^{+})}+\int_{0}^{\infty}\frac{1}{x^{4([\nu]+3)}} |u_{0}|^{2}dx,~~~B_{2}:=\int_{B}|u(x,T;u_0)|^{2}dx,
  \end{equation*}
  $~$ $A_{3}:=\int_{\mathbb{R}^{+}}e^{\lambda x}|u(x,T;u_0)|^{2}dx.$\\
  \\The proof of (\ref{T9.1}) is divided into several steps.\\
  \\ \textbf{Step 1.} There exists $C_{0}=C_{0}(\nu)>1$ such that
  \begin{equation}\label{T9.9}
  \int_{0}^{\infty}|u_0(x)|^{2}dx\leq C_{1}(x_{0},b,\lambda,T)\frac{A_{2}}{\sqrt{ln\left( \left| ln\frac{B_{2}}{A_{2}}\right| +e\right)}}, 
  \end{equation}
  where 
  \begin{equation}\label{T9.10}
  C_{1}(x_{0},b,\lambda,T):=\left( T+\frac{1}{T}\right) ^{[\nu]+3}(1+T)^{4([\nu]+3)}e^{C_{0}^{1+ \frac{x_{0}+\frac{b}{2}+1}{(\lambda T)\wedge \frac{b}{2}}}}.
  \end{equation}
  By the conservation law for the Schr\"{o}dinger equation, and the H\"{o}lder inequality, we have\\
  \\$\int_{0}^{\infty}|u_0(x)|^{2}dx=\int_{0}^{\infty}|u(x,T;u_0)|^{2}dx$
  \begin{equation}\label{T9.0}
  \begin{split}
  &\leq\left(  \int_{0}^{\infty}(1+x)^{2\nu+2+2}|u(x,T;u_0)|^{2}dx\right) ^{1/2}\left(  \int_{0}^{\infty}(1+x)^{-(2\nu+2)-2}|u(x,T;u_0)|^{2}dx\right) ^{1/2}\\
  &\leq \left(  \int_{0}^{\infty}(1+x)^{2([\nu]+3)}|u(x,T;u_0)|^{2}dx\right) ^{1/2}\left(  \int_{0}^{\infty}(1+x)^{-(2\nu+2)-2}|u(x,T;u_0)|^{2}dx\right) ^{1/2}.
  \end{split}
  \end{equation}
  Next, the proof of (\ref{T9.9}) is organized in two parts.
  \\ \textbf{Part 1.1.} There exists positive constant $C_{2}(\nu)>1$ such that
   \begin{equation}\label{T9.2}
   \int_{0}^{\infty}(1+x)^{-(2\nu+2)-2}|u(x,T;u_0)|^{2}dx\leq C_{3}(x_{0},b,\lambda,T)\frac{1}{ln(ln\frac{A_{2}}{B_{2}}+e)}A_{2},
   %\tilde{g}\left( \frac{A_{2}}{B_{2}}\right),
   \end{equation}
   where 
   \begin{equation}\label{T9.3}
   	C_{3}(x_{0},b,\lambda,T):=e^{C_{2}^{1+\frac{x_{0}+\frac{b}{2}+1}{(\lambda T)\wedge\frac{b}{2}}}}.
   \end{equation}
 %  \begin{equation}\label{T9.4}
 %  \tilde{g}(\eta):=\frac{1}{ln(ln\eta+e)},~~\eta>1.
  % \end{equation}
 %By the definitions of $A_{2}$ and $A_{3}$, we see that $A_{3}\leq A_{2}$. 
  By Theorem \ref{T4}(with ($x_{0}$, $x_{1}$, $\frac{a}{2}$, $\frac{b}{2}$) replaced by ($k-\frac{1}{2}$, $x_{0}$, $\frac{1}{2}$, $\frac{b}{2}$) for $k\in\mathbb{N}^{+}$), we find that 
 %for $k\in\mathbb{N}^{+}$,
% \begin{equation*}
 %\begin{split}
 %\int_{[k-1,k]}|u(x,T;u_0)|^{2}dx&\leq C k^{2\nu+2}((\lambda T)\wedge b)^{-(2\nu+2)}B_{2}^{\theta^{1+\frac{x_0+k+\frac{b}{2}}{(\lambda T)\wedge \frac{b}{2}}}}A_{3}^{1-\theta^{1+\frac{x_0+k+\frac{b}{2}}{(\lambda T)\wedge \frac{b}{2}}}}\\
 %&\leq C k^{2\nu+2}((\lambda T)\wedge b)^{-(2\nu+2)}B_{2}^{\theta^{1+\frac{x_0+k+\frac{b}{2}}{(\lambda T)\wedge \frac{b}{2}}}}A_{2}^{1-\theta^{1+\frac{x_0+k+\frac{b}{2}}{(\lambda T)\wedge \frac{b}{2}}}}
 %\end{split}
 %\end{equation*}
 %for some $C>0$ and $\theta\in(0,1)$ depending only on $\nu$. Hence,
 \begin{equation}\label{T9.5}
 \begin{split}
  \int_{0}^{\infty}(1+x)^{-(2\nu+2)-2}|u(x,T;u_0)|^{2}dx
  &\leq\sum_{k=1}^{\infty}\int_{k-1\leq x<k}k^{-(2\nu+2)-2}|u(x,T;u_0)|^{2}dx\\
  &\leq C ((\lambda T)\wedge b)^{-(2\nu+2)}\sum_{k=1}^{\infty}k^{-2}B_{2}^{\theta^{1+\frac{x_0+k+\frac{b}{2}}{(\lambda T)\wedge \frac{b}{2}}}}A_{3}^{1-\theta^{1+\frac{x_0+k+\frac{b}{2}}{(\lambda T)\wedge \frac{b}{2}}}}\\
  &\leq C((\lambda T)\wedge b)^{-(2\nu+2)}\left( \sum_{k=1}^{\infty}k^{-2}\left( \frac{B_{2}}{A_{2}}\right) ^{\theta^{1+\frac{x_0+k+\frac{b}{2}}{(\lambda T)\wedge \frac{b}{2}}}}\right) A_{2}
  \end{split}
  \end{equation}
 for some $C>0$ and $\theta\in(0,1)$ depending only on $\nu$.
 Since $u_{0}\neq0$, by the definitions of $A_{2}$ and $B_{2}$, and by the conservation law for the Schr\"{o}dinger equation we obtain $B_{2}<A_{2}$. Then by (\ref{L.R.2}) in Lemma \ref{L.R} with 
 \begin{equation*}
 (x, \theta, \varepsilon, \alpha)=(({B_{2}}/{A_{2}}) 
 ^{\theta^{1+\frac{x_{0}+\frac{b}{2}}{(\lambda T)\wedge \frac{b}{2}}}},\theta^{\frac{1}{(\lambda_{1}T)\wedge \frac{b}{2}}}, 1, \theta^{-1-\frac{x_{0}+\frac{b}{2}}{(\lambda T)\wedge \frac{b}{2}}}),
 \end{equation*}
 to get 
 \begin{equation}\label{T9.7}
 \sum_{k=1}^{\infty}k^{-2}\left( \frac{B_{2}}{A_{2}}\right) 
 ^{\theta^{1+\frac{x_{0}+k+\frac{b}{2}}{(\lambda T)\wedge \frac{b}{2}}}}
 \leq 4 \theta^{-1-\frac{x_{0}+\frac{b}{2}}{(\lambda T)\wedge \frac{b}{2}}} e^{1+e\theta^{1+\frac{x_{0}+\frac{b}{2}-1}{(\lambda T)\wedge \frac{b}{2}}}}\frac{1}{ln\left( \left| ln\frac{B_{2}}{A_{2}}\right| +e\right) }.
 \end{equation}
 Therefore,\\
 \\$\int_{0}^{\infty}(1+x)^{-(2\nu+2)-2}|u(x,T;u_0)|^{2}dx$
 \begin{equation}\label{T9.8}
 \begin{split}
&\leq 4C((\lambda T)\wedge\frac{b}{2} )^{-(2\nu+2)}\theta^{-1-\frac{x_{0}+\frac{b}{2}}{(\lambda T)\wedge \frac{b}{2}}} e^{1+e\theta^{1+\frac{x_{0}+\frac{b}{2}-1}{(\lambda T)\wedge \frac{b}{2}}}}\frac{A_{2}}{ln\left( \left| ln\frac{B_{2}}{A_{2}}\right| +e\right) }\\
&\leq 4C([2\nu+2]+1)!e^\frac{1}{(\lambda T)\wedge \frac{b}{2}}e^{\theta^{-1-\frac{x_{0}+\frac{b}{2}}{(\lambda T)\wedge \frac{b}{2}}}}e^{1+e\theta^{1+\frac{x_{0}+\frac{b}{2}-1}{(\lambda T)\wedge \frac{b}{2}}}}\frac{A_{2}}{ln\left( \left| ln\frac{B_{2}}{A_{2}}\right| +e\right) }\\
&\leq 4C([2\nu+2]+1)!e\cdot e^{(\theta^{-1}+e+1)\theta^{-2\frac{x_{0}+\frac{b}{2}+1}{(\lambda T)\wedge \frac{b}{2}}}}\frac{A_{2}}{ln\left( \left| ln\frac{B_{2}}{A_{2}}\right| +e\right) }.
\end{split}
 \end{equation}
 In the first inequality of (\ref{T9.8}), we have used (\ref{T9.5}) and (\ref{T9.7}); in the last two inequalities of (\ref{T9.8}), we have used the facts that
 \begin{equation*}
 \theta\in (0,1)~and~\left( (\lambda T)\wedge\frac{b}{2} \right) ^{-(2\nu+2)}\leq([2\nu+2]+1)!e^\frac{1}{(\lambda T)\wedge \frac{b}{2}}\leq ([2\nu+2]+1)!e^{\theta^{-2\frac{1}{(\lambda T)\wedge \frac{b}{2}}}}.
 \end{equation*}
 Since $\theta\in (0,1)$, (\ref{T9.2}) follows from (\ref{T9.8}), as also do (\ref{T9.3}), so we get Part 1.1.\\
 \\ \textbf{Part 1.2.}
  There exists a positive constant $C_{4}(\nu)$ such that 
  \begin{equation}\label{T9.01}
  \begin{split}
 \int_{0}^{\infty}(1+x)^{2([\nu]+3)}|u(x,T;u_0)|^{2}dx 
  \leq C_{4}\left( T+\frac{1}{T}\right) ^{2([\nu]+3)}(1+T)^{8([\nu]+3)}\left( 1+\left( (\lambda T)\wedge \frac{b}{2}\right) ^{-1}\right) ^{8([\nu]+3)} A_{2}.
  \end{split}
  \end{equation}
By Lemma \ref{L.S.2} (with $k=[\nu]+3$), we find that\\ 
\\$\int_{0}^{\infty}x^{2([\nu]+3)}|u(x,T;u_0)|^{2}dx$
\begin{equation*}
\leq C_{41}\left( T+\frac{1}{T}\right) ^{2([\nu]+3)}\left(\lVert u_{0}\rVert^{2}_{H^{4([\nu]+3)}(\mathbb{R}^{+})}+\int_{0}^{\infty}x^{8([\nu]+3)}|u_{0}|^{2}dx+\int_{0}^{\infty}\frac{1}{x^{4([\nu]+3)}} |u_{0}|^{2}dx \right) 
\end{equation*}
for some $C_{41}>0$ depending only on $\nu$. It follows that\\ 
\\$\int_{0}^{\infty}(1+x)^{2([\nu]+3)}|u(x,T;u_0)|^{2}dx\leq \int_{0}^{\infty}2^{2([\nu]+3)}(1+x^{2([\nu]+3)})|u(x,T;u_0)|^{2}dx$\\
\\$\leq C_{42}\left( T+\frac{1}{T}\right) ^{2([\nu]+3)}$\\
\begin{equation}\label{T9.11}
 \times\left(\int_{0}^{\infty}|u(x,T;u_0)|^{2}dx+\lVert u_{0}\rVert^{2}_{H^{4([\nu]+3)}(\mathbb{R}^{+})}+\int_{0}^{\infty}x^{8([\nu]+3)}|u_{0}|^{2}dx+\int_{0}^{\infty}\frac{1}{x^{4([\nu]+3)}} |u_{0}|^{2}dx \right)
\end{equation} 
for some $C_{42}>0$ depending only on $\nu$. Since
\begin{equation*}
(\lambda x)^{8([\nu]+3)}\leq (8( [\nu]+3)) !e^{\lambda x},~~x\in \mathbb{R}^{+},
\end{equation*}
and 
\begin{equation*}
\begin{split}
max\{1,\lambda^{-8([\nu]+3)}\}&=max\{1,(\lambda T)^{-8([\nu]+3)}T^{8([\nu]+3)}\}\\
&\leq(1+T)^{8([\nu]+3)}max\left\lbrace 1,\left( (\lambda T)\wedge \frac{b}{2}\right) ^{-8([\nu]+3)}\right\rbrace \\
&\leq (1+T)^{8([\nu]+3)}\left( 1+\left( (\lambda T)\wedge \frac{b}{2}\right) ^{-1}\right) ^{8([\nu]+3)},
\end{split}
\end{equation*}
this, along with (\ref{T9.11}), yields that\\
\\$\int_{0}^{\infty}(1+x)^{2([\nu]+3)}|u(x,T;u_0)|^{2}dx$
\begin{equation}\label{T9.12}
\begin{split}
&\leq C_{43}\left( T+\frac{1}{T}\right) ^{2([\nu]+3)} \left(\lVert u_{0}\rVert^{2}_{H^{4([\nu]+3)}(\mathbb{R}^{+})}+\int_{0}^{\infty}\lambda^{-8([\nu]+3)}e^{\lambda x}|u_{0}|^{2}dx+\int_{0}^{\infty}\frac{1}{x^{4([\nu]+3)}} |u_{0}|^{2}dx \right)\\
&\leq C_{43}\left( T+\frac{1}{T}\right) ^{2([\nu]+3)} max\{1,\lambda^{-8([\nu]+3)}\}A_{2}\\
&\leq C_{43}\left( T+\frac{1}{T}\right) ^{2([\nu]+3)}(1+T)^{8([\nu]+3)}\left( 1+\left( (\lambda T)\wedge \frac{b}{2}\right) ^{-1}\right) ^{8([\nu]+3)} A_{2}
\end{split}
\end{equation}
for some $C_{43}>0$ depending only on $\nu$.
Hence, we get Part 1.2.\\

Now, by  (\ref{T9.2}) and (\ref{T9.01}), we get\\
\\$\int_{0}^{\infty}|u_0(x)|^{2}dx
\leq \left(  \int_{0}^{\infty}(1+x)^{2([\nu]+3)}|u(x,T;u_0)|^{2}dx\right) ^{1/2}\left(  \int_{0}^{\infty}(1+x)^{-(2\nu+2)-2}|u(x,T;u_0)|^{2}dx\right) ^{1/2}$
\begin{equation}\label{T9.13}
\begin{split}
&\leq  \sqrt{C_{4}}\left( T+\frac{1}{T}\right) ^{[\nu]+3}(1+T)^{4([\nu]+3)}\left( 1+\left( (\lambda T)\wedge \frac{b}{2}\right) ^{-1}\right) ^{4([\nu]+3)}\sqrt{C_{3}(x_{0},b,\lambda,T)} \frac{A_{2}}{\sqrt{ln\left( ln\frac{A_{2}}{B_{2}} +e\right)}}\\
&\leq  \sqrt{C_{4}}\left( T+\frac{1}{T}\right) ^{[\nu]+3}(1+T)^{4([\nu]+3)}\left(4([\nu]+3)\right) !e^{1+((\lambda T)\wedge \frac{b}{2})^{-1}}\sqrt{C_{3}(x_{0},b,\lambda,T)} \frac{A_{2}}{\sqrt{ln\left( ln\frac{A_{2}}{B_{2}} +e\right)}}.
\end{split}
\end{equation}
(In the last inequality in (\ref{T9.13}), we have used $x^{4([\nu]+3)}\leq(4([\nu]+3))!e^{x}$ for all $x>0$.) Now, (\ref{T9.9}) follows from (\ref{T9.13}) and (\ref{T9.3}) at once.\\
 \\ \textbf{Step 2.} (\ref{T9.1}) holds for the above-mentioned $u_{0}$ and each $\varepsilon \in (0,1)$.\\
 It suffices to show that for each $\varepsilon \in (0,1)$,
 \begin{equation}\label{T9.14}
 A:=\int_{0}^{\infty}|u_0(x)|^{2}dx\leq C_{1}(\varepsilon A_{2}+\varepsilon e^{e^{\varepsilon^{-2}}}B_{2}),
 \end{equation}
 where $C_{1}=C_{1}(x_{0},b,\lambda,T)$ is given by (\ref{T9.10}). In fact, if $A\leq C_{1}\varepsilon A_{2}$, (\ref{T9.14}) is obvious. So we only consider the case: $A> C_{1}\varepsilon A_{2}$.
% for any fixed $\varepsilon>0$, there are two possibilities: either $S\leq C_{1}\varepsilon A_{2}$ or $S> C_{1}\varepsilon A_{2}$. In the first case,  
 In this case, two observations are given in order: First, since $C_{1}>1$(see (\ref{T9.10})), we deduce from the definitions of $A$ and $A_{2}$ that
\begin{equation}\label{T9.15}
 0< \varepsilon < \frac{A}{C_{1}A_{2}}<1.
\end{equation}
Since the function $x\mapsto xe^{ex^{-2}}$ is decreasing on (0,1), we see from (\ref{T9.15}) that 
\begin{equation}\label{T9.16}
\frac{A}{C_{1}A_{2}}e^{e^{\left( \frac{A}{C_{1}A_{2}}\right) ^{-2}}}\leq \varepsilon e^{e^{\varepsilon^{-2}}}.
\end{equation}
Second, since $x\mapsto e^{-e}e^{ex^{-2}}$ is decreasing on (0,1) and because its inverse is: $x\mapsto \frac{1}{\sqrt{ln\left( lnx +e\right)}}$ on $(1,\infty)$, it follows from (\ref{T9.9}) that
\begin{equation}\label{T9.17}
\frac{A_{2}}{B_{2}}\leq e^{-e}e^{e^{\left( \frac{A}{C_{1}A_{2}}\right) ^{-2}}}.
\end{equation}
Now we infer from (\ref{T9.16}) and (\ref{T9.17}) that
\begin{equation*}
A=C_{1}\frac{A}{C_{1}A_{2}}\frac{A_{2}}{B_{2}}B_{2}\leq C_{1}\left(\frac{A}{C_{1}A_{2}}e^{-e}e^{e^{(\frac{A}{C_{1}A_{2}})^{-2}}} \right)B_{2}\leq C_{1}\varepsilon e^{-e}e^{e^{\varepsilon^{-2}}}B_{2} \leq C_{1}\varepsilon e^{e^{\varepsilon^{-2}}}B_{2}.
\end{equation*}
As $\varepsilon \in (0,1)$ was arbitrary, we obtain (\ref{T9.14}). This ends the proof of (\ref{T9.1}).
\end{proof}

\section{\bf The sharpness of the main results}
The purpose of this section is to show the optimality of the inequalities established in Theorem \ref{T1} and Theorem \ref{T3}.
\subsection{The sharpness of Theorem \ref{T1}}
To show the sharpness of Theorem \ref{T1}, we establish the following theorem.
\begin{theorem}\label{T6}
\begin{itemize}
\item[(i)] Let $A=[a_{1},a_{2}]$, $B=[b_{1},b_{2}]$, $a=a_{2}-a_{1}$, $b=b_{2}-b_{1}$, and $a,b,T>0$. Then one can find a sequence $\{u_{k}\}_{k\in\mathbb{N}^{+}}\subset L^{2}(\mathbb{R}^{+})$ with 
\begin{equation}\label{T6.5}
	\int_{\mathbb{R}^{+}}|u_k(x)|^{2}dx=1
\end{equation}
such that 
\begin{equation}\label{T6.6}
	\lim_{k\rightarrow\infty}\int_{A^{c}}|u_k(x)|^{2}dx=\lim_{k\rightarrow\infty}\int_{B}|u(x,T;u_k)|^{2}dx=0.
\end{equation}
\end{itemize}
\begin{itemize}
\item[(ii)] Let $A=[a_{1},a_{2}]$, $B=[b_{1},b_{2}]$, $a=a_{2}-a_{1}$, $b=b_{2}-b_{1}$, and $S_{1}, S_{2}>0$. Then one can find a sequence $\{u_{k}\}_{k\in\mathbb{N}^{+}}\subset L^{2}(\mathbb{R}^{+})$ with 
\begin{equation}\label{T6.7}
	\int_{\mathbb{R}^{+}}|u_k(x)|^{2}dx=1
\end{equation}
such that
\begin{equation}\label{T6.8}
	\lim_{k\rightarrow\infty}\int_{A^{c}}|u(x,S_{1};u_k)|^{2}dx=\lim_{k\rightarrow\infty}\int_{0}^{S_{2}}\int_{B}|u(x,t;u_k)|^{2}dxdt=0.
\end{equation}
\item[(iii)] For each $A\subset \mathbb{R}^{+}$ with $m(A^{c})>0$, and each $T>0$, there is no constant $C>0$ such that for all $u_0\in L^{2}(\mathbb{R}^{+})$,
\begin{equation}\label{T6.1}
\int_{\mathbb{R}^{+}}|u_0(x)|^{2}dx\leq C\int_{A}|u(x,T;u_0)|^{2}dx.
\end{equation}
\end{itemize}
\end{theorem}
\begin{proof}
%For each $s\in\mathbb{R}\backslash\{0\}$ and $f\in L^{2}(\mathbb{R}^{+})$, we define 
%\begin{equation}\label{T6.2}
%	u_{s,f}(x):=e^{-ix^{2}/4s}f(x),~~~~x\in \mathbb{R}^{+}.
%\end{equation}
%By (\ref{L3.1}) and (\ref{T6.2}), we see that for all $s\in\mathbb{R}\backslash\{0\}$ and $f\in L^{2}(\mathbb{R}^{+})$,\\
%\\$(2s)^{\frac{1}{2}}e^{\frac{i(\nu+1)\pi}{2}}e^{-ix^{2}/4s}u(x,s;u_{s,f})=F_\nu(e^{\frac{iy^{2}}{4t}}u_{s,f}(y))(x/2s)=F_\nu(f)(x/2s),~x\in\mathbb{R}^{+}.$\\
%\\Thus, for all  $s\in\mathbb{R}\backslash\{0\}$ and $f\in L^{2}(\mathbb{R}^{+})$,
%\begin{equation}\label{T6.3}
%u(x,s;u_{s,f})=(2s)^{-\frac{1}{2}}e^{-\frac{i(\nu+1)\pi}{2}}e^{ix^{2}/4s}F_\nu(f)(x/2s),~~x\in\mathbb{R}^{+}.
%\end{equation}
%Now, we proof the conclusions (i)-(iii) one by one.

(i) Let $x_{0}$, $x_{1}$ be the center of $A$ and $B$. Choose
\begin{equation}\label{T6.9}
	g\in C_{0}^{\infty}({\mathbb{R}^{+}})~~such~ that~~\|g\|_{L^{2}(\mathbb{R}^{+})}=1.
\end{equation}
For each $k\in \mathbb{N}^{+}$, we set
\begin{equation}\label{T6.10}
	g_{k}(x):= k^{\frac{1}{2}}g(k(x-x_{0})),~ supp~g_{k}\subset[x_{0},\infty).
\end{equation}
We define a sequence $\{u_{k}\}\subset L^{2}(\mathbb{R}^{+})$ as follows:
\begin{equation}\label{T6.11}
	u_{k}(x):=e^{-ix^{2}/4T}g_{k}(x),~~x\in \mathbb{R}^{+},~k\in\mathbb{N}^{+}.
\end{equation}
Two observations are in order: First, by (\ref{T6.10}) and (\ref{T6.11}),
\begin{equation*}
	\begin{split}
\lim_{k\rightarrow \infty}\int_{A^{c}}|u_{k}(x)|^{2}dx	
=\lim_{k\rightarrow \infty}\int_{\frac{ak}{2}}^{\infty}|g(x)|^{2}dx=0;
	\end{split}
\end{equation*}
second, from (\ref{T6.9})-(\ref{T6.11}), we see that
\begin{equation*}
	\begin{split}
	\int_{\mathbb{R}^{+}}|u_{k}(x)|^{2}dx	
		=\int_{\mathbb{R}^{+}}|g_{k}(x)|^{2}dx==\int_{\mathbb{R}^{+}}|g(x)|^{2}dx=1,~~for~all~k\in \mathbb{N}^{+}.
	\end{split}
\end{equation*}
Next, it suffices to prove $\int_{B}|u(x,T;u_k)|^{2}dx$ goes to zero as $k\rightarrow \infty$.
%By (\ref{T6.2}) and (\ref{T6.11}), we have
%\begin{equation*}
%	u_{T,g_{k}}(x):=u_{k}(x),~~for~all~k\in\mathbb{N}^{+}.
%\end{equation*}
%From this, (\ref{T6.3}) and (\ref{T6.10}),
By (\ref{L3.1}), we see that for each $k\in\mathbb{N}^{+}$,
\begin{equation}\label{T6.12}
	u(x,T;u_{k}):=(2T)^{-\frac{1}{2}}e^{-\frac{i}{2}(\nu+1)\pi}e^{ix^{2}/4T}F_{\nu}(g_{k})(x/2T),~~x\in \mathbb{R}^{+}.
\end{equation}
Meanwhile, from (\ref{T6.10}), it follows that a.e. $x\in \mathbb{R}^{+}$,
\begin{equation*}
\begin{split}
	F_{\nu}(g_{k}(y))(x)&=\int_{0}^{\infty}\sqrt{xy}J_{\nu}(xy)g_{k}(y)dy\\
	&=\int_{x_{0}}^{\infty}\sqrt{xy}J_{\nu}(xy)k^{\frac{1}{2}}g(k(y-x_{0}))dy\\
	&=\int_{0}^{\infty}\sqrt{x(x_{0}+y')}J_{\nu}(x(x_{0}+y'))k^{\frac{1}{2}}g(ky')dy'\\
	&=k^{-\frac{1}{2}}\int_{0}^{\infty}\sqrt{x\left( x_{0}+\frac{y}{k}\right) }J_{\nu}\left( x\left( x_{0}+\frac{y}{k}\right) \right) g(y)dy.\\
\end{split}
\end{equation*}
This, along with (\ref{T6.12}), (\ref{T6.9}) and when $x\in \mathbb{R}^{+}$, $|J_{\nu}(x)|\leq C_{\nu} $ where $C_{\nu}$ is a constant that only depends on $\nu$, show that
\begin{equation*}
	\begin{split}
		\int_{B}|u(x,T;u_k)|^{2}dx	
		&\leq |B| sup_{x\in B}|u(x,T;u_k)|^{2}\\
		&\leq|B|C_{\nu}^{2}\left( (2T)^{-\frac{1}{2}}k^{-\frac{1}{2}}\int_{0}^{\infty}\sqrt{\frac{x_{1}+\frac{b}{2}}{2T}\left( x_{0}+\frac{y}{k}\right) }|g(y)|dy\right) ^{2}\\
		&\leq|B|C_{\nu}^{2}\left( (2T)^{-\frac{1}{2}}k^{-\frac{1}{2}}\left( \int_{0}^{\infty}\sqrt{\frac{x_{1}+\frac{b}{2}}{2T}x_{0}}|g(y)|dy+\int_{0}^{\infty}\sqrt{\frac{x_{1}+\frac{b}{2}}{2T}\frac{y}{k}}|g(y)|dy\right) \right) ^{2}\\
		&\leq 2|B|C_{\nu}^{2}(2T)^{-1}k^{-1}\left( \int_{0}^{\infty}\sqrt{\frac{x_{1}+\frac{b}{2}}{2T}x_{0}}|g(y)|dy\right)^{2}\\	
		&~~~+2|B|C_{\nu}^{2}(2T)^{-1}k^{-2}\left( \int_{0}^{\infty}\sqrt{\frac{x_{1}+\frac{b}{2}}{2T}y}|g(y)|dy\right)^{2},
	\end{split}
\end{equation*}
which implies that 
\begin{equation*}
\lim_{k\rightarrow \infty}\int_{B}|u(x,T;u_{k})|^{2}dx=0.
\end{equation*}
Now, from the above proof, we get (\ref{T6.5}) and (\ref{T6.6}).

(ii) Let $g$ and $g_{k}$, with $k\in\mathbb{N}^{+}$, satisfy (\ref{T6.9}) and (\ref{T6.10}), respectively. Since the Schr\"{o}dinger equation is time-reversible, we can find a sequence $\{u_{k}\}\subset L^{2}(\mathbb{R}^{+})$ such that
\begin{equation}\label{T6.13}
v_{k}(x):=u(x,S_{1};u_{k})=g_{k}(x),~~~x\in\mathbb{R}^{+},~k\in\mathbb{N}^{+}.
\end{equation}
By (\ref{T6.13}), (\ref{T6.9}) and (\ref{T6.10}), we have the following two observations: First, we notice that
\begin{equation}\label{T6.14}
	\lim_{k\rightarrow \infty}\int_{A^{c}}|v_{k}(x)|^{2}dx	
	=\lim_{k\rightarrow \infty}\int_{[0,\frac{ak}{2}]^{c}}|g(x)|^{2}dx=0.
\end{equation}
Second, we have
\begin{equation}\label{T6.15}
\int_{\mathbb{R}^{+}}|v_{k}(x)|^{2}dx	
=\int_{\mathbb{R}^{+}}|g_{k}(x)|^{2}dx=1~~for~all~k\in\mathbb{N}^{+}.
\end{equation}
%Next, by (\ref{T6.13}) and (\ref{T6.2}),
%\begin{equation*}
%v_{k}=u_{s,f}~with~(s,f)=(t,e^{ix^{2}/4t}g_{k}(x)).
%\end{equation*}
Then by (\ref{L3.1}) and (\ref{T6.13}), for each $k\in\mathbb{N}^{+}$,
\begin{equation}\label{T6.16}
	u(x,t;v_{k})=(2t)^{-\frac{1}{2}}e^{-\frac{i(\nu+1)\pi}{2}}e^{ix^{2}/4t}F_\nu(e^{iy^{2}/4t}g_{k}(y))(x/2t),~~(x,t)\in\mathbb{R}^{+}\times (\mathbb{R}\backslash\{0\}).
\end{equation}
In a similar way to (i), we can get 
\begin{equation}\label{T6.17}
	\lim_{k\rightarrow \infty}\int_{B}|u(x,t;v_{k})|^{2}dx	
	=0.
\end{equation}
At the same time, by the conservation law for the Schr\"{o}dinger equation and (\ref{T6.15}), we find that for all $k\in\mathbb{N}^{+}$ and $t\in \mathbb{R}\backslash\{0\}$,
\begin{equation*}
	\int_{B}|u(x,t;v_{k})|^{2}dx\leq \int_{\mathbb{R}^{+}}|u(x,t;v_{k})|^{2}dx=\int_{\mathbb{R}^{+}}|v_{k}(x)|^{2}dx=1.
\end{equation*}
Hence by (\ref{T6.17}), we can apply the Lebesgue dominated convergence theorem to get 
\begin{equation}\label{T6.18}
	\lim_{k\rightarrow \infty}\int_{-S_{1}}^{S_{2}-S_{1}}\int_{B}|u(x,t;v_{k})|^{2}dxdt=0.
\end{equation}
Since $v_{k}(x)=u(x,S_{1};u_{k})$, $x\in \mathbb{R}^{+}$(see (\ref{T6.13})), by (\ref{T6.14}) and (\ref{T6.15}) and (\ref{T6.18}), one can directly check that the above-mentioned sequence $\{u_{k}\}$ satisfies (\ref{T6.7}) and (\ref{T6.8}). This ends the proof of (ii).

(iii) For a contradiction, suppose that (iii) is not true. Then there exist $A_{0}\subset \mathbb{R}^{+}$ with $m(A_{0}^{c})>0$ and $C_{1}$, $T>0$ such that 
\begin{equation}\label{T6.4}
\int_{\mathbb{R}^{+}}|u_0(x)|^{2}dx\leq C_{1}\int_{A_{0}}|u(x,T;u_0)|^{2}dx~~~~for~all~u_0\in L^{2}(\mathbb{R}^{+}).
\end{equation}
we define 
\begin{equation}\label{T6.2}
	u_{T,f}(x):=e^{-ix^{2}/4T}f(x),~~~~x\in \mathbb{R}^{+}.
\end{equation}
By (\ref{T6.2}), (\ref{T6.4}) and (\ref{L3.1}), we notice that for each $f\in L^{2}(\mathbb{R}^{+})$,
\begin{equation*}
\begin{split}
\int_{\mathbb{R}^{+}}|F_\nu(f)|^{2}dx &=\int_{\mathbb{R}^{+}}|f|^{2}dx=\int_{\mathbb{R}^{+}}|u_{T,f}(x)|^{2}dx\\
&\leq C_{1}\int_{A_{0}}|u(x,T;u_{T,f})|^{2}dx=C_{1}\int_{A_{0}\text/2T}|F_\nu(f)|^{2}dx.
\end{split}
\end{equation*}
Since $|A_{0}^{c}|>0$, by taking $f\in L^{2}(\mathbb{R}^{+})\backslash \{0\}$ with supp $F_\nu(f) \subset A_{0}^{c}/2T$ in the above inequality, we are led to a contradiction.\par
\end{proof}

\subsection{The sharpness of Theorem \ref{T3}}
Next, we are going to establish Theorem \ref{T7} below, which shows the sharpness of Theorem \ref{T3}. Before, we need a lemma. In the proof of this lemma, we borrow some ideas from \cite{LM}.
\begin{lemma}\label{L.S.1}
Let $A \subset \mathbb{R}^{+}$ be measurable set such that $0\notin A$ and $dist(0,A)>0$. For each $\nu \geq 0$, there exist $C_{0}>0$ and $N_{0}>0$ such that $\forall N\geq N_{0}$, $\exists f\in L^{2}(\mathbb{R}^{+}) $ with supp$F_{\nu}f\subset[0,N]$ and
\begin{equation}\label{L.S.1.1}
N^{\frac{m}{2}}\lVert f\rVert_{L^{2}(\mathbb{R}^{+})}\geq e^{C_{0}N}\lVert f\rVert_{L^{2}(A)},
\end{equation}
where $m$ is the smallest integer for $\nu-m\leq0$.
\end{lemma}
\begin{proof}
%At several places we shall use the following simple estimate
%\begin{equation}\label{L.S.1.2}
%	\int_{\lvert x\rvert \geq \alpha}e^{-\lvert x\rvert^{2}}dx\leq C_{d}e^{-\alpha^{2}/2},~\alpha\geq1.
%\end{equation}
%The above estimate follows from the proof of Proposition 3.4 in []. We only use the estimate in $d=1$.\\
Let 
\begin{equation}\label{L.S.1.6}
	d_{0}:=dist(0,A)>0.
\end{equation}
 We consider $\phi_{s}=x^{\nu+\frac{1}{2}}e^{-sx^{2}}$, for $s>0$ and $x\in\mathbb{R}^{+}$, whose Hankel transform is given by 
\begin{equation}\label{L.S.1.3}
F_{\nu}(\phi_{s})(y)=\frac{y^{\nu+\frac{1}{2}}}{(2s)^{\nu+1}}e^{-\frac{y^{2}}{4s}},~~s>0,~y\in\mathbb{R}^{+}.
\end{equation}
We define $f\in L^{2}(\mathbb{R}^{+})$ by its Hankel transform as follows
\begin{equation*}
	F_{\nu}(f)(y)=\frac{y^{\nu+\frac{1}{2}}}{\left( \frac{N}{2}\right) ^{\nu+1}}e^{-\frac{y^{2}}{N}}1_{\{y\leq N\}}(y),~~y\in\mathbb{R}^{+},~N>0.
\end{equation*}
We first give an estimate of the $L^{2}$-norm of $f$ over the whole domain $\mathbb{R}^{+}$; we have
\begin{equation}\label{L.S.1.4}
	\lVert f\rVert_{L^{2}(\mathbb{R}^{+})}\gtrsim N^{-\frac{m+1}{2}},
\end{equation}
where $m$ is the smallest integer for $\nu-m\leq0$.\\
In fact, with the unitary of $F_{\nu}$ and the formula of integration by parts, we write
\begin{equation*}
\begin{split}
	\lVert f\rVert_{L^{2}(\mathbb{R}^{+})}^{2}&=\lVert F_{\nu}(f)\rVert_{L^{2}(\mathbb{R}^{+})}^{2}=\int_{0}^{N}\left( \frac{y^{\nu+\frac{1}{2}}}{\left( \frac{N}{2}\right) ^{\nu+1}}\right) ^{2}e^{-\frac{2y^{2}}{N}}dy\\
	&=\frac{2^{2\nu+1}}{N^{2\nu+2}}\left( -\frac{N}{2}e^{-2N}N^{2\nu}-\left( \frac{N}{2}\right) ^{2}\nu e^{-2N}N^{2(\nu-1)}-...\right.\\
	&-\left( \frac{N}{2}\right) ^{m}\nu(\nu-1)...(\nu-(m-2)) e^{-2N}N^{2(\nu-m+1)}\\
	&\left.+\left( \frac{N}{2}\right) ^{m}\nu(\nu-1)...(\nu-(m-1)) \int_{0}^{N^{2}}e^{-\frac{2x}{N}}x^{\nu-m}dx\right) .
\end{split}
\end{equation*}
Then, with the property of $m$ we obtain
\begin{equation*}
\begin{split}
		\lVert f\rVert_{L^{2}(\mathbb{R}^{+})}^{2}
		&\geq\frac{2^{2\nu+1}}{N^{2\nu+2}}\left( -\frac{N}{2}e^{-2N}N^{2\nu}-\left( \frac{N}{2}\right) ^{2}\nu e^{-2N}N^{2(\nu-1)}-...\right.\\
		&-\left( \frac{N}{2}\right) ^{m}\nu(\nu-1)...(\nu-(m-2)) e^{-2N}N^{2(\nu-m+1)}\\
		&\left.+\left( \frac{N}{2}\right) ^{m}\nu(\nu-1)...(\nu-(m-1))N^{2(\nu-m)} \int_{0}^{N^{2}}e^{-\frac{2x}{N}}dx\right) \\
		&=\frac{2^{2\nu+1}}{N^{2\nu+2}}\left( \left( \frac{N}{2}\right) ^{m}\nu(\nu-1)...(\nu-(m-1))N^{2(\nu-m)}\frac{N}{2}\right.\\
		&-\frac{N}{2}e^{-2N}N^{2\nu}-\left( \frac{N}{2}\right) ^{2}\nu e^{-2N}N^{2(\nu-1)}-...\\
		&-\left( \frac{N}{2}\right) ^{m}\nu(\nu-1)...(\nu-(m-2)) e^{-2N}N^{2(\nu-m+1)}\\
		&\left.-\left( \frac{N}{2}\right) ^{m}\nu(\nu-1)...(\nu-(m-1))N^{2(\nu-m)}\frac{N}{2}e^{-2N}\right) \\
		&\gtrsim N^{-m-1},
\end{split}
\end{equation*}
by choosing $N$ sufficiently large.
We now wish to estimate the $L^{2}$-norm of $f$ over the subset $A$. The inverse Hankel transformation gives 
\begin{equation}\label{L.S.1.5}
	\begin{split}
		f(x)&=\int_{0}^{\infty}\sqrt{xy}J_{\nu}(xy)F_{\nu}f(y)dy\\
		&=\int_{0}^{N}\sqrt{xy}J_{\nu}(xy)\frac{y^{\nu+\frac{1}{2}}}{\left( \frac{N}{2}\right) ^{\nu+1}}e^{-\frac{y^{2}}{N}}dy\\
		&=\int_{0}^{\infty}\sqrt{xy}J_{\nu}(xy)\frac{y^{\nu+\frac{1}{2}}}{\left( \frac{N}{2}\right) ^{\nu+1}}e^{-\frac{y^{2}}{N}}dy-\int_{N}^{\infty}\sqrt{xy}J_{\nu}(xy)\frac{y^{\nu+\frac{1}{2}}}{\left( \frac{N}{2}\right) ^{\nu+1}}e^{-\frac{y^{2}}{N}}dy\\
		&=\phi_{s=\frac{N}{4}}(x)-R(x).
	\end{split}
\end{equation}
For the first term in (\ref{L.S.1.5}), by (\ref{L.S.1.6}) and some computations similar to the estimate of $	\lVert f\rVert_{L^{2}(\mathbb{R}^{+})}^{2}$, we get
\begin{equation*}
\begin{split}
\lVert \phi_{s=\frac{N}{4}}\rVert_{L^{2}(A)}^{2}&\leq \int_{x>d_{0}}\left( \phi_{s=\frac{N}{4}}(x)\right) ^{2}dx=\int_{x>d_{0}}(x^{\nu+\frac{1}{2}}e^{-\frac{N}{4}x^{2}})^{2}dx\\
&\lesssim e^{-\frac{N}{2}d_{0}^{2}}N^{-1}.
\end{split}
\end{equation*}
 For the second term in (\ref{L.S.1.5}),  the unitary property of Hankel transform and some computations give
\begin{equation*}
	\begin{split}
		\lVert R\rVert_{L^{2}(A)}^{2}&\leq 	\lVert R\rVert_{L^{2}(\mathbb{R}^{+})}^{2}=\int_{N}^{\infty}\left( \frac{y^{\nu+\frac{1}{2}}}{\left( \frac{N}{2}\right) ^{\nu+1}}\right) ^{2}e^{-\frac{2y^{2}}{N}}dy\\
		&\lesssim e^{-2N}N^{-1}.
	\end{split}
\end{equation*} 
 Setting $C_{1}=min(1,d_{0}^{2}/4)$, we thus obtain
\begin{equation}\label{L.S.1.7}
\lVert f\rVert_{L^{2}(A)}\lesssim N^{-\frac{1}{2}}e^{-C_{1}N}.
\end{equation}
We conclude the proof with (\ref{L.S.1.4}), (\ref{L.S.1.7}) and by choosing $C_{0}$ such that $0<C_{0}<C_{1}$.
\end{proof}

\begin{theorem}\label{T7}
\begin{itemize}
\item[(i)] For each $A=[a_{1},a_{2}] \subset \mathbb{R}^{+}$,  $a=a_{2}-a_{1}$, and $a, \lambda, T>0$, then one can find a sequence $\{u_{k}\}_{k\in\mathbb{N}^{+}}\subset C_{0}^{\infty}(\mathbb{R}^{+})$ and $M>0$ such that
\begin{equation}\label{T6.4.1}
	\int_{\mathbb{R}^{+}}e^{\lambda x}|u_k(x)|^{2}dx\leq M~and~\int_{\mathbb{R}^{+}}|u_k(x)|^{2}dx=1
\end{equation}
and
\begin{equation}\label{T6.4.2}
\lim_{k\rightarrow\infty}\int_{A}|u(x,T;u_k)|^{2}dx=0.
\end{equation}
\end{itemize}
\begin{itemize}
\item[(ii)] Let $\lambda,b>0$ and $\alpha(s)$, $s\in \mathbb{R}^{+}$, be an increasing function with $\lim_{s\rightarrow\infty}\alpha(s)/s=0$. Then for each $\gamma \in (0,1)$, there is no positive constant $C$ such that for all $u_{0}\in C_{0}^{\infty}(\mathbb{R}^{+})$,
\begin{equation}\label{T7.1}
\int_{\mathbb{R}^{+}}|u_0(x)|^{2}dx\leq C\left( \int_{[0,b]^c}|u(x,T;u_0)|^{2}dx\right) ^{\gamma}\left( \int_{\mathbb{R}^{+}}e^{\lambda \alpha(x)}|u_0(x)|^{2}dx\right) ^{1-\gamma}.
\end{equation}
\end{itemize}
\end{theorem}
\begin{proof}
(i) Suppose $x_{0}$ be the center of $A$. Let $g$ and $g_{k}$, with $k\in\mathbb{N}^{+}$, satisfy (\ref{T6.9}) and (\ref{T6.10}), respectively.
As a same way to the proof of Theorem \ref{T6}(i), we define a sequence $\{u_{k}\}\subset C_{0}^{\infty}(\mathbb{R}^{+})$ as follows:
\begin{equation}\label{T6.4.3}
	u_{k}(x):=e^{-ix^{2}/4T}g_{k}(x),~~x\in \mathbb{R}^{+},~k\in\mathbb{N}^{+}.
\end{equation}
We can get
\begin{equation*}
\int_{\mathbb{R}^{+}}|u_k(x)|^{2}dx=1~~~~for~all~k\in\mathbb{N}^{+}
\end{equation*}
and
\begin{equation*}
	\lim_{k\rightarrow\infty}\int_{A}|u(x,T;u_k)|^{2}dx=0.
\end{equation*} 
Meanwhile, from (\ref{T6.9}), (\ref{T6.10}) and (\ref{T6.4.3}), we find that for each $k\in\mathbb{N}^{+}$,
\begin{equation*}
\begin{split}
	\int_{\mathbb{R}^{+}}e^{\lambda x}|u_k(x)|^{2}dx&=\int_{\mathbb{R}^{+}}e^{\lambda x}|g_k(x)|^{2}dx=\int_{x_{0}}^{\infty}e^{\lambda x}|k^{\frac{1}{2}}g(k(x-x_{0}))|^{2}dx\\
	&=\int_{0}^{\infty}e^{\lambda(\frac{x}{k}+x_{0})}|g(x)|^{2}dx\leq \int_{0}^{\infty}e^{\lambda(x+x_{0})}|g(x)|^{2}dx<\infty.
\end{split}
\end{equation*}
Hence, this ends the proof of (i).\\
(ii) For a contradiction, suppose that the conclusion is not true. Then there exist $\bar{b}$, $\bar{\lambda}$, $\bar{T}$, $\bar{C}>0$, $\bar{\gamma}\in (0,1)$, and an increasing function $\bar{\alpha}(s)$ on $[0,\infty)$ with $\lim_{s\rightarrow\infty}\bar{\alpha}(s)/s=0$ such that for each $v_{0}\in C_{0}^{\infty}(\mathbb{R}^{+}),$
\begin{equation}\label{T7.2}
\int_{\mathbb{R}^{+}}|v_0(x)|^{2}dx\leq \bar{C}\left( \int_{[0,\bar{b}]^c}|u(x,\bar{T};v_0)|^{2}dx\right) ^{\bar{\gamma}}\left( \int_{\mathbb{R}^{+}}e^{\bar{\lambda}\bar{\alpha} (x)}|v_0(x)|^{2}dx\right) ^{1-\bar{\gamma}}.
\end{equation} 
For any fixed $g\in L^{2}(\mathbb{R}^{+})$ with $F_{\nu}g\in C_{0}^{\infty}(\mathbb{R}^{+})$, we define $v_{0,g}\in C_{0}^{\infty}(\mathbb{R}^{+})$ via
\begin{equation}\label{T7.3}
	F_{\nu}g(x)=(2\bar{T})^{\frac{1}{2}}e^{-\frac{i}{2}(\nu+1)\pi}e^{i\bar{T}x^{2}}v_{0,g}(2\bar{T}x),~~x\in\mathbb{R}^{+}.
\end{equation} 
From (\ref{T3.3}), (\ref{T3.4}) (with $(T,u_{0})=(\bar{T},v_{0,g})$) and (\ref{T7.3}), we find that
\begin{equation}\label{T7.4}
	g(x)=e^{-ix^{2}/4\bar{T}}u(x,\bar{T};v_{0,g}),~~x\in\mathbb{R}^{+}.
\end{equation} 
%Indeed, let
%\begin{equation}\label{T7.5}
%	f_{g}(x):=e^{-ix^{2}/4\hat{T}}u(x,\hat{T};v_{0,g}),~~x\in\mathbb{R}^{+}.
%\end{equation}
%Then, by (\ref{T3.3}), (\ref{T3.4}) (with $(T,u_{0})=(\hat{T},v_{0,g})$) and (\ref{T7.3}), we find that
%\begin{equation*}
%F_{\nu}f_{g}(x)=(2\hat{T})^{\frac{1}{2}}e^{-\frac{i}{2}(\nu+1)\pi}e^{i\hat{T}x^{2}}v_{0,g}(2\hat{T}x)=F_{\nu}g(x),~~x\in\mathbb{R}^{+},
%\end{equation*}
%which implies that $f_{g}=g$. This along with (\ref{T7.5}), leads to (\ref{T7.4}).\\
By (\ref{T7.4}), the conservation law (for the Schr\"{o}dinger equation), (\ref{T7.2}) and (\ref{T7.3}), we get that
\begin{equation*}
\begin{split}
	\int_{\mathbb{R}^{+}}|g(x)|^{2}dx&=\int_{\mathbb{R}^{+}}|u(x,\bar{T};v_{0,g})(x)|^{2}dx=\int_{\mathbb{R}^{+}}|v_{0,g}(x)|^{2}dx\\
	&\leq \bar{C}\left( \int_{[0,\bar{b}]^c}|u(x,\bar{T};v_{0,g})|^{2}dx\right) ^{\bar{\gamma}}\left( \int_{\mathbb{R}^{+}}e^{\bar{\lambda}\bar{\alpha} (x)}|v_{0,g}(x)|^{2}dx\right) ^{1-\bar{\gamma}}\\
	&=\bar{C}\left( \int_{[0,\bar{b}]^c}|g(x)|^{2}dx\right) ^{\bar{\gamma}}\left( \int_{\mathbb{R}^{+}}e^{\bar{\lambda}\bar{\alpha} (2\bar{T}x)}|F_{\nu}g(x)|^{2}dx\right) ^{1-\bar{\gamma}}.
\end{split}
\end{equation*}
From this, using a standard density argument, we obtain that for each $g\in L^{2}(\mathbb{R}^{+})$ with supp $F_{\nu}g$ compact,
\begin{equation*}
		\int_{\mathbb{R}^{+}}|g(x)|^{2}dx\leq \bar{C}\left( \int_{[0,\bar{b}]^c}|g(x)|^{2}dx\right) ^{\bar{\gamma}}\left( \int_{\mathbb{R}^{+}}e^{\bar{\lambda}\bar{\alpha} (2\bar{T}x)}|F_{\nu}g(x)|^{2}dx\right) ^{1-\bar{\gamma}}.
\end{equation*}
Since $\bar{\alpha}(\cdotp)$ is increasing and because the Hankel transform is an isometry, the above implies that for each $N\geq 1$ and each $g\in L^{2}(\mathbb{R}^{+})$ with supp $F_{\nu}g\subset[0,N]$,
 \begin{equation}\label{T7.6}
 \begin{split}
 	\int_{\mathbb{R}^{+}}|g(x)|^{2}dx&\leq \bar{C}\left( \int_{[0,\bar{b}]^c}|g(x)|^{2}dx\right) ^{\bar{\gamma}}\left( \int_{\mathbb{R}^{+}}e^{\bar{\lambda}\bar{\alpha} (2\bar{T}x)}|F_{\nu}g(x)|^{2}dx\right)^{1-\bar{\gamma}}\\
 	&=\bar{C}e^{(1-\bar{\gamma})\bar{\lambda}\bar{\alpha} (2\bar{T}N)}\left( \int_{[0,\bar{b}]^c}|g(x)|^{2}dx\right) ^{\bar{\gamma}}\left(  \int_{\mathbb{R}^{+}}|F_{\nu}g(x)|^{2}dx\right) ^{1-\bar{\gamma}}.
 \end{split}
 \end{equation}
 So (\ref{T7.6}) implies that for $N\geq1$ and each $g\in L^{2}(\mathbb{R}^{+})$ with supp $F_{\nu}g\subset[0,N]$,
 \begin{equation}\label{T7.7}
 		\int_{\mathbb{R}^{+}}|g(x)|^{2}dx\leq \bar{C}^{\frac{1}{\bar{\gamma}}}e^{\frac{1-\bar{\gamma}}{\bar{\gamma}}\bar{\lambda}\bar{\alpha} (2\bar{T}N)}\int_{[0,\bar{b}]^c}|g(x)|^{2}dx.
 \end{equation}
 In addition, according to Lemma \ref{L.S.1}, there are $C_{0}$, $N_{0}>0$, such that for each $N\geq N_{0}$ there is $f_{N}\in L^{2}(\mathbb{R}^{+})\backslash \{0\}$ with supp $F_{\nu}f_{N}\subset[0,N]$ such that
  \begin{equation}\label{T7.8}
 	e^{C_{0}N}\int_{[0,\bar{b}]^c}|f_{N}(x)|^{2}dx\leq N^{\frac{m}{2}}\int_{\mathbb{R}^{+}}|f_{N}(x)|^{2}dx,
 \end{equation}
 where $m$ is the smallest integer for $\nu-m\leq0$.\\
 By (\ref{T7.7}) and (\ref{T7.8}), we get that for each $N\geq N_{0}$,
 \begin{equation*}
 	e^{C_{0}N}\leq N^{\frac{m}{2}}\bar{C}^{\frac{1}{\bar{\gamma}}}e^{\frac{1-\bar{\gamma}}{\bar{\gamma}}\bar{\lambda}\bar{\alpha} (2\bar{T}N)},
 \end{equation*}
 from which it follows that
 \begin{equation*}
 	0<\frac{\bar{\gamma}C_{0}}{2(1-\bar{\gamma})\bar{\lambda}\bar{T}}\leq \lim_{N\rightarrow \infty}\frac{\bar{\alpha}(2\bar{T}N)}{2\bar{T}N}.
 \end{equation*}
 This contradicts $\lim_{s\rightarrow \infty}\frac{\bar{\alpha}(s)}{s}=0$. Hence, the conclusion is true.
 \end{proof}
 \section{Applications}
 Now,  we turn to the applications. We mainly consider the applications of Theorem \ref{T1}\textendash \ref{T9} to different kinds of controllability for the Schrödinger equations.  Based on an abstract lemma \cite[Lemma 5.1]{WWZ} concerning the equivalence between observability and controllability, we can obtain the following results. For the meaning and understanding of the following controllability properties, we refer readers to the literature \cite{WWZ}.
 
 \begin{theorem}\label{T61}
 Let $A, B$ be two measurable sets in $\mathbb{R}^{+}$ with finite measure and $T>0$. Let $0\leq t_{1}<t_{2}\leq T$. Consider the following impulse controlled Schrödinger equation:
\begin{equation}\label{T61.1}
\begin{cases}
i\partial_tu-H_{\nu}u=\delta_{\{t=t_{1}\}}\chi_{A^{c}}(x)f_{1}+\delta_{\{t=t_{2}\}}\chi_{B^{c}}(x)f_{2},\,\,x\in \mathbb{R}^+,~ t\in(0,T), \\
u(x,0)=u_0\in L^{2}(\mathbb{R}^{+}).
\end{cases}
\end{equation}
Denote by $u_{1}(\cdotp,\cdotp,u_{0},f_{1},f_{2})$ the solution to the equation (\ref{T61.1}). Then for any $u_{0}$, $u_{T} \in L^{2}(\mathbb{R}^{+})$, there exists a pair of controls $(f_{1},f_{2})\in L^{2}(\mathbb{R}^{+})\times L^{2}(\mathbb{R}^{+})$ such that
\begin{equation}\label{T61.2}
u_{1}(x,T,u_{0},f_{1},f_{2})=u_{T},
\end{equation}
and 
\begin{equation}\label{T61.3}
\lVert f_{1}\rVert^{2}_{L^{2}(\mathbb{R}^{+})}+\lVert f_{2}\rVert^{2}_{L^{2}(\mathbb{R}^{+})}\leq C\lVert u_{T}-e^{-iH_{\nu}T}u_{0}\rVert^{2}_{L^{2}(\mathbb{R}^{+})},
\end{equation}
where the constant $C=C(\nu,A,B,t_{2}-t_{1})$ is given by Theorem \ref{T1}.
\end{theorem} 
\begin{proof}
We sketch the proof here. Consider the following dual equation
\begin{equation}\label{T61.4}
\begin{cases}
i\partial_t\psi-H_{\nu}\psi=0,\,\,x\in \mathbb{R}^+,~ t\in(0,T), \\
u(x,T)=g\in L^{2}(\mathbb{R}^{+}).
\end{cases}
\end{equation}
Write $\psi(\cdotp,\cdotp,T,g)$ for the solution to (\ref{T61.4}). Then Theorem \ref{T1} implies that 
\begin{equation}\label{T61.5}
\int_{\mathbb{R}^{+}}|g|^{2}dx\leq C\left( \int_{A^{c}}|\psi(\cdotp,t_{1},T,g)|^{2}dx+\int_{B^{c}}|\psi(\cdotp,t_{2},T,g)|^{2}dx\right).
\end{equation}
Now we define the state transformation operator $R:L^{2}(\mathbb{R}^{+})\rightarrow L^{2}(\mathbb{R}^{+})$ and the observation operator $O:L^{2}(\mathbb{R}^{+})\times L^{2}(\mathbb{R}^{+})\rightarrow L^{2}(\mathbb{R}^{+})$ as follows:
\begin{equation}\label{T61.6}
Rg=g;~~Of=\left( \chi_{A^{c}}(x)\psi(\cdotp,t_{1},T,g),~\chi_{B^{c}}(x)\psi(\cdotp,t_{2},T,g)\right) .
\end{equation}
From (\ref{T61.5}) and (\ref{T61.6}), we find that for each $g\in L^{2}(\mathbb{R}^{+})$,
\begin{equation}\label{T61.7}
\lVert Rg\rVert^{2}_{L^{2}(\mathbb{R}^{+})}\leq C\lVert Og\rVert^{2}_{L^{2}(\mathbb{R}^{+})\times L^{2}(\mathbb{R}^{+})}+\frac{1}{k}\lVert Rg\rVert^{2}_{L^{2}(\mathbb{R}^{+})},~~~k\in\mathbb{N}^{+}.
\end{equation}
By Lemma 5.1 in \cite{WWZ}, there is a pair $(f_{1k},f_{2k})\in L^{2}(\mathbb{R}^{+})\times L^{2}(\mathbb{R}^{+})$, $k\in\mathbb{N}^{+}$ such that the following dual inequality holds
\begin{equation}\label{T61.8}
\frac{1}{C}\lVert (f_{1k},f_{2k})\rVert^{2}_{L^{2}(\mathbb{R}^{+})\times L^{2}(\mathbb{R}^{+})}+k\lVert R^{*}g-O^{*}(f_{1k},f_{2k})\rVert^{2}_{L^{2}(\mathbb{R}^{+})}\leq \lVert g\rVert^{2}_{L^{2}(\mathbb{R}^{+})},~~~k\in\mathbb{N}^{+},
\end{equation}
where 
\begin{equation}\label{T61.9}
R^{*}g=g;~~~O^{*}(f_{1k},f_{2k})=u(\cdotp,T,0,f_{1k},f_{2k}).
\end{equation}
Define 
\begin{equation}\label{T61.10}
g=u_{T}-e^{-iH_{\nu}T}u_{0},~~~x\in \mathbb{R}^{+}.
\end{equation}
Then (\ref{T61.2}) and (\ref{T61.3}) are followed by choosing a weak convergence subsequence in (\ref{T61.9}) and a limiting procedure.
\end{proof}
Similarly, combining Theorems \ref{T3}-\ref{T9} with Lemma 5.1 in \cite{WWZ}, we get the following controllability results:

 \begin{theorem}\label{T62}
 Let $b,\lambda>0$ and $T>s\geq0$. Consider the following impulse controlled Schrödinger equation:
\begin{equation}\label{T62.1}
\begin{cases}
i\partial_tu-H_{\nu}u=\delta_{\{t=s\}}\chi_{[0,b]^{c}}(x)f,\,\,x\in \mathbb{R}^+,~ t\in(0,T), \\
u(x,0)=u_0\in L^{2}(\mathbb{R}^{+}).
\end{cases}
\end{equation}
Denote by $u_{2}(\cdotp,\cdotp,u_{0},f)$ the solution to the equation (\ref{T62.1}). Then for any $\varepsilon>0$ and $u_{0}$, $u_{T} \in L^{2}(\mathbb{R}^{+})$, there exists a control $f\in L^{2}(\mathbb{R}^{+})$ such that\\
\\$~~~~~~~$ $\varepsilon^{(1-q)/q}\int_{\mathbb{R}^{+}}|f(x)|^{2}dx+\varepsilon^{-1}\int_{\mathbb{R}^{+}}e^{-\lambda x}|u_{2}(\cdotp,T,u_{0},f)-u_{T}(\cdotp)|^{2}dx$
\begin{equation}\label{T62.2}
\leq C\left( 1+\frac{b^{2\nu+2}}{(\lambda(T-s))^{2\nu+2}}\right)\int_{\mathbb{R}^{+}}|u_{T}(x)-e^{-iH_{\nu}T}u_{0}(x)|^{2}dx ,
\end{equation}
where  $q:=\theta^{1+\frac{b}{\lambda(T-s)}}\in(0,1)$, $C$ and $\theta$ be given by Theorem \ref{T3}(i).
\end{theorem}

\begin{theorem}\label{T63}
Given any interval $A=[a_{1}, a_{2}]$, $B=[b_{1},b_{2}]\subset \mathbb{R}^{+}$, $a=a_{2}-a_{1}$, $b=b_{2}-b_{1}$, and $a,b,\lambda,T>0$. Consider the following impulse controlled Schrödinger equation:
\begin{equation}\label{T63.1}
\begin{cases}
i\partial_tu-H_{\nu}u=\delta_{\{t=0\}}\chi_{B}(x)f,\,\,x\in \mathbb{R}^+, t\in(0,T), \\
u(x,0)=u_0\in L^{2}(\mathbb{R}^{+}).
\end{cases}
\end{equation}
Denote by $u_{3}(\cdotp,\cdotp,u_{0},f)$ the solution to the equation (\ref{T63.1}). Then for any $\varepsilon>0$ and $u_{0} \in \tilde{L}^{2}(A;\mathbb{C})$, there exists a control $f\in L^{2}(\mathbb{R}^{+})$ such that\\
\\$~~~~~~~$ $\varepsilon^{(1-\theta^{p})/\theta^{p}}\int_{\mathbb{R}^{+}}|f(x)|^{2}dx+\varepsilon^{-1}\int_{\mathbb{R}^{+}}e^{-\lambda x}|u_{3}(\cdotp,T,u_{0},f)|^{2}dx$
\begin{equation}\label{T63.2}
\leq C(a_{2}^{2\nu+2}-a_{1}^{2\nu+2})((\lambda T)\wedge b)^{-(2\nu+2)}\int_{A}|u_{0}(x)|^{2}dx ,
\end{equation}
where  $C$, $\theta$ and $p$ be given by Theorem \ref{T4}, $\tilde{L}^{2}(A;\mathbb{C}):=\{g\in L^{2}(\mathbb{R}^{+};\mathbb{C}):g=0 ~on~A^{c}\}$.
\end{theorem}

\begin{theorem}\label{T64}
Let $b,N>0$ and $0\leq s <T.$ Consider the following impulse controlled Schrödinger equation :
\begin{equation}\label{T64.1}
\begin{cases}
i\partial_tu-H_{\nu}u=\delta_{\{t=s\}}\chi_{[0,b]^{c}}(x)f,\,\,x\in \mathbb{R}^+, t\in(0,T), \\
u(x,0)=u_0\in L^{2}(\mathbb{R}^{+}).
\end{cases}
\end{equation}
Denote by $u_{4}(\cdotp,\cdotp,u_{0},f)$ the solution to the equation (\ref{T64.1}). Then for any $u_{0},u_{T} \in L^{2}(\mathbb{R}^{+})$, there exists a control $f\in L^{2}(\mathbb{R}^{+})$ such that\\
\begin{equation}\label{T64.2}
u_{4}(x,T,u_{0},f)=u_{T},~~~x\in[0,N]
\end{equation}
and 
\begin{equation}\label{T64.3}
\lVert f\rVert^{2}_{L^{2}(\mathbb{R}^{+})}\leq e^{C(1+\frac{bN}{T-s})}\lVert u_{T}-e^{-iH_{\nu}T}u_{0}\rVert^{2}_{L^{2}(\mathbb{R}^{+})},
\end{equation}
where the constant $C=C(\nu)$ is given by Theorem \ref{T5}.
\end{theorem}

\begin{theorem}\label{T65}
Given any interval $B=[b_{1},b_{2}]\subset \mathbb{R}^{+}$, $b=b_{2}-b_{1}$, $\lambda_{1},\lambda_{2}, T>0$. Consider the following impulse controlled Schrödinger equation :
\begin{equation}\label{T65.1}
\begin{cases}
i\partial_tu-H_{\nu}u=\delta_{\{t=0\}}\chi_{B}(x)f,\,\,x\in \mathbb{R}^+, t\in(0,T), \\
u(x,0)=u_0\in L^{2}(\mathbb{R}^{+}).
\end{cases}
\end{equation}
Denote by $u_{5}(\cdotp,\cdotp,u_{0},f)$ the solution to the equation (\ref{T65.1}). Then for any $\varepsilon\in(0,1)$ and $u_{0} \in W_{\lambda_{2}}$, there exists a control $f\in L^{2}(\mathbb{R}^{+})$ such that\\
\\$~~~~~~~$ $\frac{1}{\varepsilon}e^{(\frac{1}{\varepsilon})^{1+\frac{1}{C\lambda_{2}((\lambda_{1}T)\wedge \frac{b}{2})}}}\int_{\mathbb{R}^{+}}|f(x)|^{2}dx+\frac{1}{\varepsilon}\int_{\mathbb{R}^{+}}e^{-\lambda_{1} x}|u_{5}(\cdotp,T,u_{0},f)|^{2}dx$
\begin{equation}\label{T65.2}
~~~~~~~~~~~~~~~~~~~~~~~~~~~~~\leq C(x_{0},b,\lambda_{1},\lambda_{2},T)\int_{\mathbb{R}^{+}}e^{\lambda_{2}x}|u_{0}(x)|^{2}dx ,
\end{equation}
where  $C(x_{0},b,\lambda_{1},\lambda_{2},T)$  and $C$ be given by Theorem \ref{T8}, $W_{\lambda_{2}}:=\{f\in L^{2}(\mathbb{R}^{+}):\int_{\mathbb{R}^{+}}e^{\lambda_{2}x}|f(x)|^{2}dx<\infty\}$.
\end{theorem}

Finally, we are going to show the last Theorem. Before, we introduce a space. For each $\lambda>0$, we denote $R_{\lambda}$ for the completion of $C_{0}^{\infty}(\mathbb{R}^+)$ in the norm 
\begin{equation*}
\lVert f\rVert_{R_{\lambda}}:=\left( \int_{\mathbb{R}^{+}}e^{\lambda x}|f(x)|^{2}dx+\lVert f\rVert^{2}_{H^{4([\nu]+3)}(\mathbb{R}^{+})}+\int_{0}^{\infty}\frac{1}{x^{4([\nu]+3)}} |f(x)|^{2}dx\right)^{1/2},~~f\in C_{0}^{\infty}(\mathbb{R}^+). 
\end{equation*}
Write $R_{\lambda}^{*}$ for the dual space of $R_{\lambda}$ with respect to the pivot space $L^{2}(\mathbb{R}^{+})$.
\begin{theorem}\label{T66}
Given any interval $B=[b_{1},b_{2}]\subset \mathbb{R}^{+}$, $b=b_{2}-b_{1}$, $\lambda >0$ and $T>s\geq0$. Consider the following impulse controlled Schrödinger equation :
\begin{equation}\label{T66.1}
\begin{cases}
i\partial_tu-H_{\nu}u=\delta_{\{t=s\}}\chi_{B}(x)f,\,\,x\in \mathbb{R}^+, t\in(0,T), \\
u(x,0)=u_0\in L^{2}(\mathbb{R}^{+}).
\end{cases}
\end{equation}
Denote by $u_{6}(\cdotp,\cdotp,u_{0},f)$ the solution to the equation (\ref{T66.1}). Then for any $\varepsilon\in(0,1)$ and $u_{0},u_{T} \in L^{2}(\mathbb{R}^{+})$, there exists a control $f\in L^{2}(\mathbb{R}^{+})$ such that\\
\\$~~~~~~~$ $\varepsilon^{-1}e^{-e^{\varepsilon^{-2}}}\int_{\mathbb{R}^{+}}|f(x)|^{2}dx+\varepsilon^{-1}\lVert u_{6}(\cdotp,T,u_{0},f)-u_{T}(\cdotp)\rVert^{2}_{R_{\lambda}^{*}}$
\begin{equation}\label{T66.2}
~~~~~~~~~~~~~~~~~~~~~~~~~~~~~\leq C(x_{0},b,\lambda,T-s)\lVert u_{T}-e^{-iH_{\nu}T}u_{0}\rVert^{2}_{L^{2}(\mathbb{R}^{+})},
\end{equation}
where  $C(x_{0},b,\lambda,T-s)$ be given by Theorem \ref{T9}. 
\end{theorem}

\end{document}